\documentclass[12pt]{amsart}
\usepackage{graphicx}
\usepackage{amscd}
\usepackage[top=1in,bottom=1in,left=1.15in,right=1.15in]{geometry}
\usepackage{amsmath,amssymb,amsfonts}
\usepackage{amsthm}
\usepackage{color}
\usepackage{hyperref}
\usepackage{tikz-cd}
\usepackage{epstopdf}
\usetikzlibrary{matrix}
\usepackage{verbatim}
\usepackage{array}
\usepackage{mathtools}

\newtheorem{theo}{Theorem}[section]
\newtheorem{prop}[theo]{Proposition}
\newtheorem{lemma}[theo]{Lemma}
\newtheorem{defn}[theo]{Definition}
\newtheorem{rem}[theo]{Remark}

\newtheorem{ex}[theo]{Example}

\allowdisplaybreaks

\def\diaCrossP{\unitlength.08em
  \begin{minipage}{15\unitlength}
    \begin{picture}(15,15)
      \put(0,0){\vector(1,1){15}}
      \qbezier(15,0)(15,0)(10,5)
      \qbezier(5,10)(0,15)(0,15)
      \put(0,15){\vector(-1,1){0}}
    \end{picture}
  \end{minipage}
}

\def\diaCrossN{\unitlength.08em
  \begin{minipage}{15\unitlength}
    \begin{picture}(15,15)
      \put(15,0){\vector(-1,1){15}}
      \qbezier(0,0)(0,0)(5,5)
      \qbezier(10,10)(15,15)(15,15)
      \put(15,15){\vector(1,1){0}}
    \end{picture}
  \end{minipage}
}

\def\diaCircle{\unitlength.1em
  \begin{minipage}{15\unitlength}
    \begin{picture}(15,15)
      \put(7.5,10){\circle{8}}
      \put(7.5,0){\vector(0,1){5}}
       \put(7.5,5){\line(0,1){5.5}}
    \end{picture}
  \end{minipage}
}

\def\Ptwist{\unitlength.1em
  \begin{minipage}{15\unitlength}
    \begin{picture}(15,15)
     \put(0,7.5){\line(1,0){6}}
     \put(9,7.5){\line(1,0){6}}
      \put(4,9.5){\line(1,0){3}}
      \qbezier(7,9.5)(7,9.5)(9,5.5)
      \put(9, 5.5){\line(1,0){3}}
    \end{picture}
  \end{minipage}
}

\def\Ntwist{\unitlength.1em
  \begin{minipage}{15\unitlength}
    \begin{picture}(15,15)
     \put(0,7.5){\line(1,0){6}}
     \put(9,7.5){\line(1,0){6}}
      \put(4,5.5){\line(1,0){3}}
      \qbezier(7,5.5)(7,5.5)(9,9.5)
      \put(9, 9.5){\line(1,0){3}}
    \end{picture}
  \end{minipage}
}

\usetikzlibrary{decorations.markings}

\tikzset{->-/.style={decoration={
  markings,
  mark=at position .5 with {\arrow{>}}},postaction={decorate}}}

\tikzset{-<-/.style={decoration={
  markings,
  mark=at position .5 with {\arrow{<}}},postaction={decorate}}}

\begin{document}
\title{An Alexander polynomial for MOY graphs}
\author{Yuanyuan Bao \and Zhongtao Wu}
\address{
Graduate School of Mathematical Sciences, the University of Tokyo, 3-8-1 Komaba, Tokyo 153-8914, Japan
}
\email{bao@ms.u-tokyo.ac.jp}

 \address{
Department of Mathematics, the Chinese University of Hong Kong, Shatin, Hong Kong, China
}
\email{ztwu@math.cuhk.edu.hk}

\begin{abstract}
We introduce an Alexander polynomial for MOY graphs. For a framed trivalent MOY graph $\mathbb{G}$, we refine the construction and obtain a framed ambient isotopy invariant $\Delta_{(\mathbb{G},c)}(t)$. The invariant $\Delta_{(\mathbb{G}, c)}(t)$ satisfies a series of relations, which we call MOY-type relations, and conversely these relations determine $\Delta_{(\mathbb{G}, c)}(t)$. Using them we provide a graphical definition of the Alexander polynomial of a link. Finally, we discuss some properties and applications of our invariants.
\end{abstract}

\maketitle

\tableofcontents
\newpage


\section{Introduction}
\subsection{Main results}
We study a class of spatial graphs, which we call MOY graphs. The terminology was first used by Wu \cite{MR3234803} in his categorification of the MOY calculus introduced by Murakami-Ohtsuki-Yamada \cite{MR1659228}. Roughly speaking, an MOY graph is an oriented spatial graph in the $3$-sphere $S^3$ satisfying the following conditions:

\begin{enumerate}
\item \emph{Transverse orientation}: For every vertex $v$, there is a straight line $L_v$ that separates the incoming and outgoing edges.

\item \emph{Balanced coloring}: Each edge is colored by a non-negative integer such that for every vertex $v$, the sum of the colors of the edges entering $v$ equals that on the edges leaving $v$.
\end{enumerate}

While we will give the precise definition in Section 2, we exhibit an example of MOY graph in Fig. ~\ref{moygraph}.  We remark that an oriented spatial graph with a transverse orientation is called \emph{transverse spatial graph} in Harvey-O'Donnol \cite[Definition 2.2]{MR3677933}, where an alternative definition using the formal description of simplicial complex was given.

\begin{figure}[h!]
\begin{tikzpicture}[baseline=-0.65ex, thick, scale=1]
\draw (-1, -1.75) [->-] to (0, -0.75);
\draw (-0.5, -1.75) [->-] to (0, -0.75);
\draw (0, -0.75) [->] to (1, 0.25);
\draw (1, -1.75) [->-] to (0, -0.75);
\draw (0, -0.75) [->] to (-1, 0.25);
\draw (0, -0.75) [->] to (0.5, 0.25);
\draw (0, -0.75) node[circle,fill,inner sep=1.5pt]{};
\draw [dashed] (-0.7, -0.75)--(0.7, -0.75);
\draw (1.25, -0.75) node{$L_v$};
\draw (0.1, -1.5) node{$......$};
\draw (-0.1, 0) node{$......$};
\end{tikzpicture} \hspace{2cm}
\begin{tikzpicture}[baseline=-0.65ex, thick, scale=1]
\draw (0,-1)  to [out=90,in=270] (0.5,-0.33);
\draw (0,-1) to [out=270,in=180] (1.5,-2);
\draw (1.5,-2) to [out=0,in=0] (1.5,1);
\draw (0.5, -0.33) [->-] to (0.5,0.33);
\draw (0.5, 0.33) [->] to [out=90,in=0] (-0.5,1);
\draw (1,-1)  to [out=90,in=270] (0.5,-0.33);
\draw (1,-1) to [out=270,in=45] (0.6,-1.6);
\draw (0.3,-1.8) to [out=225,in=0] (-0.5,-2);
\draw (-0.5,-2) to [out=180,in=180] (-0.5,1);
\draw (1.5,1) [<-] to [out=180,in=90] (0.5,0.33);
\draw (0,1.25) node {$j$};
\draw (1,1.25) node {$i$};
\draw (1.2,0) node {$i+j$};
\draw (0.5, -0.33) node[circle,fill,inner sep=1pt]{};
\draw (0.5, 0.33) node[circle,fill,inner sep=1pt]{};
\end{tikzpicture}
	\caption{The local picture of a vertex for an MOY graph (left) and an example of MOY graph (right).}
\label{moygraph}
\end{figure}

We then introduce the Kauffman state sum $\langle D,c \rangle$ for a given graph diagram $D$ of $G$ with a coloring $c$, which generalizes the Kauffman state sum introduced in \cite{MR712133} for a knot. Based on the state sum we define an \emph{Alexander polynomial}
\begin{equation*}
\Delta_{(G, c)}(t):=\frac{ \langle D,c\rangle}{(t^{-1/2}-t^{1/2})^{\vert V \vert-1}},
\end{equation*}
where $\vert V \vert$ is the number of vertices of $G$.  The Alexander polynomial is a rational function of $t^{1/2}$ and $t^{-1/2}$, and is, like its classical counterpart for a link, well-defined up to $t^k$.

When we fix a framing on $G$ and obtain a framed graph $\mathbb{G}$, it is possible to construct a \emph{normalized Alexander polynomial} of $\mathbb{G}$ that eliminates the aforementioned ambiguity. For the reader who is familiar with the definition of Jones polynomial of knots, this is the same idea as multiplying a writhe factor to Kauffman's bracket polynomial. We will focus on the study of the normalized Alexander polynomial of trivalent framed MOY graph, as it is directly related to the topic of MOY calculus that interests us the most.  Here is our first main result.

\begin{theo}
\label{main1}
For a framed trivalent MOY graph $\mathbb{G}$, we can fix a normalization of the Alexander polynomial $\Delta_{(\mathbb{G}, c)}(t)$ so that it is a well-defined rational function invariant under framed ambient isotopy.
\end{theo}

While our invariant is similar to some existing invariants, which we will discuss presently, we have the following three contributions: (1) The state sum formula is new; (2) We prove the topological invariance combinatorially; namely, we prove its invariance under Reidemeister moves by applying the state sum formula;  (3) The normalization of the Alexander polynomial in the framed case is new.

\medskip
Our second main result is in Section 4. It is inspired by Murakami-Ohtsuki-Yamada's relations in \cite{MR1659228}, where they provided a graphical definition for the $U_q(\mathfrak{sl}_n)$-polynomial invariants of a link for all $n\geq 2$, which is called MOY calculus. We summarize Section 4 as follows.

\begin{theo}
\label{main2}
For a framed trivalent MOY graph $\mathbb{G}$, where each edge is colored by a positive integer,
the invariant $\Delta_{(\mathbb{G}, c)}(t)$ satisfies a series of relations, which we call MOY-type relations; conversely, these relations determine $\Delta_{(\mathbb{G}, c)}(t)$. Using them we provide a graphical definition for the Alexander polynomial of a link.
\end{theo}

We remark that when $n=0$, the original MOY relations in \cite{MR1659228} gives a trivial invariant that is identically 0 for all graphs, instead of the desired Alexander invariant.  There are various solutions to this issue, for example, by cutting open the graph diagrams.  Nevertheless, the relations  in our theorem seem to be the ones that resemble most closely the original constructions.

\vspace{3mm}

\subsection{Relation with existing invariants}
The Alexander polynomial of a knot was first defined and studied by Alexander in 1920s. Now we know there are many different interpretations coming from different backgrounds for the same invariant. Unlike the case of a knot,  there exist diverse versions of Alexander polynomials for a spatial graph, which are usually distinct as invariants.

For clarity, we summarize the relation of $\langle D,c\rangle$, $\Delta_{(G, c)}(t)$ and $\Delta_{(\mathbb{G}, c)}(t)$ with some existing invariants of spatial graphs.\\
\begin{itemize}
\item The state sum $\langle D,c\rangle$ simplifies the state sum proposed by the first named author in \cite{bao}. We provide an explanation in Section 2. \\
\item The state sum $\langle D,c\rangle$ is also an elaboration of the Alexander polynomial defined in \cite{MR3677933}, as both the state sum in \cite{bao} and the present paper and the Alexander polynomial in \cite{MR3677933} are special cases of the torsion invariant defined in Friedl-Juh\'asz-Rasmussen \cite{MR2805998}.\\
\item As an application of Theorem \ref{main2}, the first named author showed in a subsequent paper \cite[Theorem 3.2]{MR4001658} that the Alexander polynomial $\Delta_{(\mathbb{G},c)}(t)$ is equivalent to Viro's $\mathfrak{gl}(1\vert 1)$-Alexander polynomial of a framed trivalent graph defined in \cite{MR2255851}, in the case that the graph has no sink or source vertices and each edge is colored by a positive integer. The key observation is that Viro's $\mathfrak{gl}(1\vert 1)$-Alexander polynomial also satisfies an adapted version of the MOY-type relations, which totally characterize the invariant. \\
\item For the $\theta_{n}$-graph, Litherland defined an invariant similar to $\langle D,c\rangle$ in \cite{MR994083}.\\
\end{itemize}

\begin{rem}
\rm
For other versions of Alexander polynomial of spatial graphs, e.g, Kinoshita \cite{MR0102819}, as long as the definitions depend solely on the complement of the graph in $S^3$, they are not equivalent to the one considered in this paper.
\end{rem}

\vspace{3mm}
\subsection{Organization of the paper}

In Section 2, we introduce the preliminaries of graphs, diagrams and Kauffman states, based on which we define a state sum $\langle D, c \rangle$.  We then compare it with a different state sum proposed in \cite{bao}.   The proof of their equivalence is rather lengthy, technical and isolated from the remainder of the paper , so we postpone it to the Appendix.

In Section 3, we prove the invariance of $\langle D, c \rangle$ under various Reidemeister moves for graph diagrams. Then we give the definition of $\Delta_{(G,c)}(t)$ and exhibit its topological invariance.  Furthermore,  we prove Theorem \ref{main1}.

In Section 4, we elaborate on Theorem \ref{main2}.  More precisely, we wrote down a set of 10 MOY-type relations in Theorem \ref{moytheo} and prove that they uniquely characterize $\Delta_{(\mathbb{G}, c)}(t)$ in Theorem \ref{theorem:axiom}.

In Section 5, we discuss some properties and topological applications of our Alexander polynomial.  In particular, we give in Theorem \ref{condition} a necessary condition for the planarity of a graph in terms of the coefficient of its Alexander polynomial.

\vspace{3mm}
\noindent{\bf Acknowledgements.}
We would like to thank Hitoshi Murakami for helpful discussions and suggestions. The first author was partially supported by Grant-in-Aid for Research Activity Start-up (No. 26887010).
The second author is partially supported by grant from the Research Grants Council of Hong Kong Special
Administrative Region, China (Project No. 14309016 and 14301317).


\section{The Kauffman state sum $\langle D,c \rangle$}

\subsection{MOY graphs}

We recall the definition of an MOY graph. Our exposition and terminology is slightly different from the one used in \cite{MR3234803}.

\begin{defn}
\rm
\begin{enumerate}

\item An \emph{abstract MOY graph} is an oriented graph that equipped with \emph{a balanced coloring} $c: E\to \mathbb{Z}_{\geq 0}$ such that for each vertex $v$,
$$\sum_{\text{$e$: pointing into $v$}} c(e)=\sum_{\text{$e$: pointing out of $v$}} c(e).$$
We allow closed loops without vertices as our abstract graphs.

\item An {\it MOY graph diagram} in $\mathbb{R}^2$ is an immersion of an abstract MOY graph into $\mathbb{R}^2$, with crossing information and a \emph{transverse orientation}: through each vertex $v$, there is a straight line $L_v$ that separates the edges entering $v$ and the edges leaving $v$.

\begin{figure}[h!]
\begin{tikzpicture}[baseline=-0.65ex, thick, scale=1]
\draw (-1, -1.75) [->-] to (0, -0.75);
\draw (-0.5, -1.75) [->-] to (0, -0.75);
\draw (0, -0.75) [->] to (1, 0.25);
\draw (1, -1.75) [->-] to (0, -0.75);
\draw (0, -0.75) [->] to (-1, 0.25);
\draw (0, -0.75) [->] to (0.5, 0.25);
\draw (0, -0.75) node[circle,fill,inner sep=1.5pt]{};
\draw [dashed] (-0.7, -0.75)--(0.7, -0.75);
\draw (1.25, -0.75) node{$L_v$};
\draw (0.1, -1.5) node{$......$};
\draw (-0.1, 0) node{$......$};
\end{tikzpicture}
\label{fig:e3}
\end{figure}

\item Two MOY graph diagrams are said to be {\it equivalent} if they can be connected by a finite sequence of Reidemeister moves in Fig. \ref{fig:e25} and isotopies in $\mathbb{R}^2$.

\begin{figure}[h!]
I: \quad \begin{tikzpicture}[baseline, thick, scale=0.4]
\draw (1,-1)   -- (0.2,-0.2);
\draw (-1, -1) -- (1, 1);
\draw (-0.2,0.2) --  (-1,1) ;
\draw (-1, 1) arc (90:270:1);
\end{tikzpicture}\quad  $\longleftrightarrow$ \quad
\begin{tikzpicture}[baseline=-0.65ex, thick, scale=0.4]
\draw (-1,-1)   -- (-0.2,-0.2);
\draw (1, -1) -- (-1, 1) ;
\draw (0.2,0.2)  --  (1,1);
\draw (-1, 1) arc (90:270:1);
\end{tikzpicture} \quad  $\longleftrightarrow$ \quad
\begin{tikzpicture}[baseline=-0.65ex, thick, scale=0.4]
\draw (0, -1) -- (0, 1) ;
\end{tikzpicture} \\  \vspace{3mm}
II: \quad \begin{tikzpicture}[baseline, thick, scale=0.4]
\draw (-1, 2) arc (90:270:2);
\draw (-3.5, 2) arc (90:60:2);
\draw (-1.5, 0) arc (0:40:2);
\draw (-1.5, 0) arc (0:-40:2);
\draw (-3.5, -2) arc (270:300:2);
\end{tikzpicture} \quad  $\longleftrightarrow$ \quad
\begin{tikzpicture}[baseline, thick, scale=0.4]
\draw (-1, 2) arc (110:250:2);
\draw (-5, 2) arc (70:-70:2);
\end{tikzpicture} \\  \vspace{3mm}
III: \quad\begin{tikzpicture}[baseline, thick, scale=0.4]
\draw (-2, 2) -- (-0.4, 0.4) ;
\draw (-2, 1) -- (-1.5, 1) ;
\draw (1.5, 1) -- (2, 1) ;
\draw (-0.5, 1) -- (0.5, 1) ;
\draw (0.4, -0.4) -- (2, -2) ;
\draw (2, 2) -- (-2, -2) ;
\end{tikzpicture}\quad  $\longleftrightarrow$ \quad
\begin{tikzpicture}[baseline, thick, scale=0.4]
\draw (-2, 2) -- (-0.4, 0.4) ;
\draw (-2, -1) -- (-1.5, -1) ;
\draw (1.5, -1) -- (2, -1) ;
\draw (-0.5, -1) -- (0.5, -1) ;
\draw (0.4, -0.4) -- (2, -2) ;
\draw (2, 2) -- (-2, -2) ;
\end{tikzpicture}\quad\quad\quad
\begin{tikzpicture}[baseline, thick, scale=0.4]
\draw (-2, 2) -- (2, -2) ;
\draw (-2, 1) -- (-1.5, 1) ;
\draw (1.5, 1) -- (2, 1) ;
\draw (-0.5, 1) -- (0.5, 1) ;
\draw (0.4, -0.4) -- (2, -2) ;
\draw (2, 2) -- (0.4, 0.4) ;
\draw (-2, -2) -- (-0.4, -0.4) ;
\end{tikzpicture}\quad  $\longleftrightarrow$ \quad
\begin{tikzpicture}[baseline, thick, scale=0.4]
\draw (-2, 2) -- (2, -2) ;
\draw (-2, -1) -- (-1.5, -1) ;
\draw (1.5, -1) -- (2, -1) ;
\draw (-0.5, -1) -- (0.5, -1) ;
\draw (2, 2) -- (0.4, 0.4) ;
\draw (-2, -2) -- (-0.4, -0.4) ;
\end{tikzpicture}\\\vspace{4mm}
IV: \quad \begin{tikzpicture}[baseline=-0.65ex, thick, scale=1]
\draw (-1, -1) -- (0, 0);
\draw (-0.5, -1) -- (0, 0);
\draw (0, 0) -- (1, 1);
\draw (1, -1) -- (0, 0);
\draw (0, 0) -- (-1,1);
\draw (0, 0) -- (0.5, 1);
\draw (0, 0) node[circle,fill,inner sep=1.5pt]{};
\draw [dashed] (-0.7, 0)--(0.7, 0);
\draw (0.1, -0.5) node{$...$};
\draw (-0.1, 0.5) node{$...$};
\draw (-1.2, 0.7) -- (-0.9, 0.7);
\draw (0.9, 0.7) -- (1.2, 0.7);
\draw (0.65, 0.7) -- (0.45, 0.7);
\draw (-0.6, 0.7) -- (0.2, 0.7);
\end{tikzpicture}  $\longleftrightarrow$
\begin{tikzpicture}[baseline=-0.65ex, thick, scale=1]
\draw (-1, -1) --  (0, 0);
\draw (-0.5, -1) -- (0, 0);
\draw (0, 0) -- (1, 1);
\draw (1, -1)--  (0, 0);
\draw (0, 0) --  (-1,1);
\draw (0, 0) --  (0.5, 1);
\draw (0, 0) node[circle,fill,inner sep=1.5pt]{};
\draw [dashed] (-0.7, 0)--(0.7, 0);
\draw (0.1, -0.5) node{$...$};
\draw (-0.1, 0.5) node{$...$};
\draw (-1.2, -0.7) -- (-0.9, -0.7);
\draw (0.9, -0.7) -- (1.2, -0.7);
\draw (-0.65, -0.7) -- (-0.45, -0.7);
\draw (0.6, -0.7) -- (-0.2, -0.7);
\end{tikzpicture}\quad\quad
\begin{tikzpicture}[baseline=-0.65ex, thick, scale=1]
\draw (-1, -1) --  (0, 0);
\draw (-0.5, -1) -- (0, 0);
\draw (0, 0) -- (0.6, 0.6);
\draw (0.8, 0.8) --  (1, 1);
\draw (1, -1) --  (0, 0);
\draw (0, 0) -- (-0.6,0.6);
\draw (-0.8, 0.8) --  (-1,1);
\draw (0, 0) -- (0.3, 0.6);
\draw (0.4, 0.8) --  (0.5, 1);
\draw (0, 0) node[circle,fill,inner sep=1.5pt]{};
\draw [dashed] (-0.7, 0)--(0.7, 0);
\draw (0.1, -0.5) node{$...$};
\draw (-0.1, 0.5) node{$...$};
\draw (-1.2, 0.7) -- (1.2, 0.7);
\end{tikzpicture}  $\longleftrightarrow$
\begin{tikzpicture}[baseline=-0.65ex, thick, scale=1]
\draw (-1, -1) -- (-0.8, -0.8);
\draw (-0.6, -0.6) --  (0, 0);
\draw (-0.5, -1) -- (-0.4, -0.8);
\draw (-0.3, -0.6) --  (0, 0);
\draw (0, 0) -- (1, 1);
\draw (0.6, -0.6) -- (0, 0);
\draw (1, -1) -- (0.8, -0.8);
\draw (0, 0) -- (-1,1);
\draw (0, 0) -- (0.5, 1);
\draw (0, 0) node[circle,fill,inner sep=1.5pt]{};
\draw [dashed] (-0.7, 0)--(0.7, 0);
\draw (0.1, -0.5) node{$...$};
\draw (-0.1, 0.5) node{$...$};
\draw (-1.2, -0.7) -- (1.2, -0.7);
\end{tikzpicture}\\\vspace{4mm}
V: \quad \begin{tikzpicture}[baseline=-0.65ex, thick, scale=1]
\draw (-1.5, -1) -- (0, 0);
\draw (0, 0) -- (1, 1);
\draw (1.5, -1) -- (0, 0);
\draw (0, 0) -- (-1,1);
\draw (0, 0) -- (0.5, 1);
\draw (0, 0) node[circle,fill,inner sep=1.5pt]{};
\draw [dashed] (-0.7, 0)--(0.7, 0);
\draw (0.6, -0.6) node{$...$};
\draw (-0.6, -0.6) node{$...$};
\draw (-0.1, 0.5) node{$...$};
\draw (0,0) to [out=225,in=90] (-0.25,-0.4) to [out=270,in=315] (-0.1,-0.65) to (0.5, -1);
\draw (-0.5, -1) -- (-0.1, -0.75);
\draw (0,0) to [out=315,in=90] (0.25,-0.4) to [out=270,in=225] (0.1,-0.65);
\end{tikzpicture}
$\longleftrightarrow$
\begin{tikzpicture}[baseline=-0.65ex, thick, scale=1]
\draw (-1.5, -1) -- (0, 0);
\draw (-0.5, -1) -- (0, 0);
\draw (0.5, -1) -- (0, 0);
\draw (0, 0) -- (1, 1);
\draw (1.5, -1) -- (0, 0);
\draw (0, 0) -- (-1,1);
\draw (0, 0) -- (0.5, 1);
\draw (0, 0) node[circle,fill,inner sep=1.5pt]{};
\draw [dashed] (-0.7, 0)--(0.7, 0);
\draw (0.6, -0.6) node{$...$};
\draw (-0.6, -0.6) node{$...$};
\draw (-0.1, 0.5) node{$...$};
\end{tikzpicture}$\longleftrightarrow$
\begin{tikzpicture}[baseline=-0.65ex, thick, scale=1]
\draw (-1.5, -1) -- (0, 0);
\draw (0, 0) -- (1, 1);
\draw (1.5, -1) -- (0, 0);
\draw (0, 0) -- (-1,1);
\draw (0, 0) -- (0.5, 1);
\draw (0, 0) node[circle,fill,inner sep=1.5pt]{};
\draw [dashed] (-0.7, 0)--(0.7, 0);
\draw (0.6, -0.6) node{$...$};
\draw (-0.6, -0.6) node{$...$};
\draw (-0.1, 0.5) node{$...$};
\draw (0,0) to [out=315,in=90] (0.25,-0.4) to [out=270,in=45] (0,-0.75) to (-0.5, -1);
\draw (0.5, -1) -- (0.1, -0.75);
\draw (0,0) to [out=225,in=90] (-0.25,-0.4) to [out=270,in=315] (-0.1,-0.65);
\end{tikzpicture}
\caption{Reidemeister moves for MOY graph diagrams. Suppressed orientations of the edges can be added in all compatible ways.}
\label{fig:e25}
\end{figure}

\item An {\it MOY graph} (or a \emph{knotted MOY graph} in Wu's terminology in \cite{MR3234803}) is an equivalence class of MOY graph diagrams.
\end{enumerate}

\begin{rem}\rm
An MOY graph can be regarded as a transverse spatial graph in $S^3$ with a balanced coloring. For the definition of a transverse spatial graph, see \cite{MR3677933}.
\end{rem}

\end{defn}
A vertex of valence 1 in a graph is called an \emph{end point}; a vertex of valence greater than 1 is called an \emph{internal vertex}.  A graph is \emph{closed} if it has no end points.  A balanced coloring is called \emph{non-trivial} if it is not identically 0, i.e., $c(e)>0$ for some $e\in E$; a balanced coloring is called \emph{positive} if each edge is colored by a positive integer, i.e., $c(e)>0$ for all $e\in E$. All graphs considered in the present paper are assumed to be closed and equipped with a non-trivial coloring. a graph is \emph{trivalent} if all of its internal vertices have valence 3.  Here are two important classes of MOY graphs.

\begin{ex}
\rm
\label{sing}
\begin{enumerate}
\item Each singular knot or link gives rise to an MOY graph with a vertex corresponding to each singular crossing and a balanced coloring $c(e)=1$ for every edge $e$ (Fig. \ref{singtriva} left).
\item An oriented trivalent graph without source or sink has an induced transverse orientation (Fig. \ref{singtriva} right). With a balanced coloring, this gives an MOY graph.
\end{enumerate}
\end{ex}

\begin{figure}[h!]
\begin{tikzpicture}[baseline=-0.65ex, thick, scale=0.6]
\draw (-1, -0.25) [->] to (1, 1.75);
\draw (1, -0.25) [->] to (-1, 1.75);
\draw (0, 0.75) node[circle,fill,inner sep=1.5pt]{};
\draw [dashed] (-1, 0.75)--(1, 0.75);
\end{tikzpicture} \hspace{3cm}
\begin{tikzpicture}[baseline=-0.65ex, thick, scale=1.2]
\draw  (0, 0) [->-] to (0,0.5);
\draw   (0,0.5) [->]  to  (0.66,1);
\draw   (0,0.5) [->]  to  (-0.66,1);
\draw (0, 0.5) node[circle,fill,inner sep=1.5pt]{};
\draw [dashed] (-0.5, 0.5)--(0.5, 0.5);
\end{tikzpicture} \hspace{1.5cm}
\begin{tikzpicture}[baseline=-0.65ex, thick, scale=1.2]
\draw  (0, 0.5) [->] to (0,1);
\draw   (0,0.5) [-<-]  to  (0.66,0);
\draw   (0,0.5) [-<-]  to  (-0.66,0);
\draw (0, 0.5) node[circle,fill,inner sep=1.5pt]{};
\draw [dashed] (-0.5, 0.5)--(0.5, 0.5);
\end{tikzpicture}
\caption{A singular crossing (left) and trivalent vertices (right).}
\label{singtriva}
\end{figure}

\subsection{Kauffman states}

Let $D$ be a graph diagram of a given MOY graph $G$. To define the state sum $\langle D, c\rangle$, we construct the set of states on $D$, which we call Kauffman states.

\begin{defn}
\rm
Starting from an MOY graph diagram $D$, we can obtain a {\it decorated diagram} $(D,\delta)$ by putting a base point $\delta$ on one of the edges in $D$ and drawing a circle around each vertex of $D$. See Example \ref{decorated}.\\

\begin{enumerate}
\item  $\operatorname{Cr}(D)$: denotes the set of crossings, including the types \diaCrossP and \diaCrossN which are the double points of the diagram and the type \diaCircle which are the intersection points around each vertex between the incoming edges with the circle.\\
\item $\operatorname{Re}(D)$: denotes the set of regions, including the {\it regular regions} of $\mathbb{R}^{2}$ separated by $D$ and the {\it circle regions} around the vertices. Note that there is exactly one circle region around each vertex. {\it Marked regions} are the regions adjacent to the base point $\delta$, and the others are called {\it unmarked regions}.\\

\item Corners: For a crossing of type \diaCrossP or \diaCrossN, there are four corners around it, and we call them the {\it north}, {\it south}, {\it west}, and {\it east corners} of the crossing. Around a crossing of type \diaCircle there are three corners, and we call the one inside the circle region the {\it north} corner, the one on the left of the crossing the {\it west} corner and the one on the right the {\it east} corner.  Note also that every corner belongs to a unique region in $\operatorname{Re}(D)$.  \\
\begin{figure}[h!]
\begin{tikzpicture}[baseline=-0.65ex, thick, scale=0.9]
\draw (-1,-1) [->] to (1,1);
\draw (1.3, 1) node {$i$};
\draw (1,-1) -- (0.2,-0.2);
\draw (-0.2,0.2) [->] to (-1,1);
\draw (0, 0.5) node {N};
\draw (0, -0.5) node {S};
\draw (0.5, 0) node {E};
\draw (-0.5, 0) node {W};
\end{tikzpicture}\hspace{2cm}
\begin{tikzpicture}[baseline=-0.65ex, thick, scale=0.9]
\draw (0, 0.5) ellipse (1.5cm and 0.8cm);
\draw (0,-1) [->-] to (0,-0.3);
\draw (-0.4, -0.6) node {W};
\draw (0.4, -0.6) node {E};
\draw (0, 0) node {N};
\end{tikzpicture}
\end{figure}

\end{enumerate}
\end{defn}

\begin{ex}
\label{decorated}
\rm
In Fig. \ref{fig:e5}, we exhibit an MOY graph and a decorated diagram of it. We have
$|\operatorname{Cr}(D)|=4$: one of which is of the type \diaCrossN, and three of which are of the types \diaCircle;
$|\operatorname{Re}(D)|=6$: four of which are regular regions, and two of which are circle regions. The base point $\delta$ is adjacent to two regions that we mark out by $\star$.

\begin{figure}[h!]
\begin{tikzpicture}[baseline=-0.65ex, thick, scale=1.2]
\draw (0,-1)  to [out=90,in=270] (0.5,-0.33);
\draw (0,-1) to [out=270,in=180] (1.5,-2);
\draw (1.5,-2) to [out=0,in=0] (1.5,1);
\draw (0.5, -0.33) [->-] to [out=90,in=270] (0.5,0.33);
\draw (0.5, 0.33) [->] to [out=90,in=0] (-0.5,1);
\draw (1,-1)  to [out=90,in=270] (0.5,-0.33);
\draw (1,-1) to [out=270,in=45] (0.6,-1.6);
\draw (0.3,-1.8) to [out=225,in=0] (-0.5,-2);
\draw (-0.5,-2) to [out=180,in=180] (-0.5,1);
\draw (1.5,1) [<-] to [out=180,in=90] (0.5,0.33);
\draw (0.5, -0.33) node[circle,fill,inner sep=1pt]{};
\draw (0.5, 0.33) node[circle,fill,inner sep=1pt]{};
\end{tikzpicture}\hspace{2cm}
\begin{tikzpicture}[baseline=-0.65ex, thick, scale=1.2]
\draw (0,-1)  to [out=90,in=270] (0.5,-0.33);
\draw (0,-1) to [out=270,in=180] (1.5,-2);
\draw (1.5,-2) to [out=0,in=0] (1.5,1);
\draw (0.5, -0.33) [->-] to [out=90,in=270] (0.5,0.33);
\draw (0.5, 0.33) [->] to [out=90,in=0] (-0.5,1);
\draw (1,-1)  to [out=90,in=270] (0.5,-0.33);
\draw (1,-1) to [out=270,in=45] (0.6,-1.6);
\draw (0.3,-1.8) to [out=225,in=0] (-0.5,-2);
\draw (-0.5,-2) to [out=180,in=180] (-0.5,1);
\draw (1.5,1) [<-] to [out=180,in=90] (0.5,0.33);
\draw (0.5, -0.33) node[circle,fill,inner sep=1pt]{};
\draw (0.5, -0.33) circle (0.3);
\draw (0.5, 0.33) node[circle,fill,inner sep=1pt]{};
\draw (0.5, 0.33) circle (0.3);
\draw (-1.33,0) node {$*$};
\draw (-1,0) node {$\delta$};
\draw (-0.5,-0.5) node {$\star$};
\draw (-1.7,-0.5) node {$\star$};
\draw (1,-0.6) node {$\longleftarrow$};
\draw (2.1,-0.6) node {a crossing};
\draw (0.9,0.3) node {$\longleftarrow$};
\draw (2.3,0.3) node {a circle region};
\end{tikzpicture}
	\caption{Left: an MOY graph diagram $D$. Right: a decorated diagram $(D, \delta)$.}
	\label{fig:e5}
\end{figure}
\end{ex}

Calculating the Euler characteristic of $\mathbb{R}^2$ using $D$ gives the following lemma.
\begin{lemma}
\label{reandcr}
$\vert \operatorname{Re}(D) \vert = \vert \operatorname{Cr}(D) \vert+2$ if and only if $D$ is a connected graph diagram.
\end{lemma}

A generic base point $\delta$ is adjacent to at most two distinct regions, which will be denoted by $R_u$ and $R_{v}$. It is possible that $R_u$ and $R_{v}$ are the same regions.

\begin{defn}
\label{states}
\rm
A {\it Kauffman state}, or simply, a {\it state} for a decorated diagram $(D, \delta)$ is a bijective map $$s: \, \operatorname{Cr}(D)\rightarrow \operatorname{Re}(D)\backslash \{R_u, R_{v}\},$$ which sends a crossing in $\operatorname{Cr}(D)$ to one of its corners. Let $S(D, \delta)$ denote the set of all states.
\end{defn}

\begin{rem}
\rm
It is possible that $S(D, \delta)$ is an empty set. This, for example, occurs when $D$ is not connected or the base point $\delta$ is only adjacent to a single region.
\end{rem}

\begin{ex}
\rm
The graph in Fig. \ref{fig:e5} has three Kauffman states, one of which is illustrated in Fig. \ref{fig:e5state1} by putting $\bullet$'s at the chosen corners.

\begin{figure}[h!]
\begin{tikzpicture}[baseline=-0.65ex, thick, scale=1.3]
\draw (0,-1)  to [out=90,in=270] (0.5,-0.33);
\draw (0,-1) to [out=270,in=180] (1.5,-2);
\draw (1.5,-2) to [out=0,in=0] (1.5,1);
\draw (0.5, -0.33) [->-] to [out=90,in=270] (0.5,0.33);
\draw (0.5, 0.33) [->] to [out=90,in=0] (-0.5,1);
\draw (1,-1)  to [out=90,in=270] (0.5,-0.33);
\draw (1,-1) to [out=270,in=45] (0.6,-1.6);
\draw (0.3,-1.8) to [out=225,in=0] (-0.5,-2);
\draw (-0.5,-2) to [out=180,in=180] (-0.5,1);
\draw (1.5,1) [<-] to [out=180,in=90] (0.5,0.33);
\draw (0.5, -0.33) node[circle,fill,inner sep=1pt]{};
\draw (0.5, -0.33) circle (0.3);
\draw (0.5, 0.33) node[circle,fill,inner sep=1pt]{};
\draw (0.5, 0.33) circle (0.3);
\draw (-1.35,0) node {$*$};
\draw (-1,0) node {$\delta$};
\draw (-0.5,-0.5) node {$\star$};
\draw (-1.7,-0.5) node {$\star$};
\draw (0.8,-0.6) node {$\bullet$};
\draw (0.5,0.15) node {$\bullet$};
\draw (0.45,-0.5) node {$\bullet$};
\draw (0.45,-1.5) node {$\bullet$};
\end{tikzpicture}
	\caption{A Kauffman state associated to $(D, \delta)$ is marked out by $\bullet$'s .}
	\label{fig:e5state1}
\end{figure}
\end{ex}

\vspace{3mm}

\subsection{Kauffman state sum}

Suppose $(D, \delta)$ is a connected decorated diagram.  Before introducing the state sum, we define the index of a regular region.

For each regular region of $(D, \delta)$,
the {\it index} of it is defined by the following rules.
\begin{enumerate}

\item The index of the unbounded region is set to be $0$.

\item The indices of the other regular regions are inductively determined by the rule as exhibited below: traversing an edge with color $i$ in its orientation, let the difference of the index of its left-hand side region and that of its right-hand side region be $i$.
\end{enumerate}
\begin{figure}[h!]
\begin{tikzpicture}[baseline=-0.65ex, thick, scale=1.8]
\draw (0,0) [->] to (0,1);
\draw (0.3, 1) node {$i$};
\draw (0.6, 0.5) node {$n-i$};
\draw (-0.5, 0.5) node {$n$};
\end{tikzpicture}
\end{figure}
The definition of a balanced coloring ensures that the above rules give rise to a well-defined index for each regular region.

\begin{defn}
\rm
Suppose $\delta$ is on an edge with non-zero color $i$, and the indices of the regions adjacent to $\delta$ are $n$ and $n-i$. Define
\begin{equation*}\label{delta}
\vert \delta \vert=t^{n-i}-t^{n}.
\end{equation*}

\end{defn}

\vspace{2mm}

\begin{defn}\label{def:statesum}
\rm
Suppose $(D, \delta)$ is a connected decorated diagram with $N$ crossings $C_1, C_2, \cdots, C_N$ in $\operatorname{Cr}(D)$ and $N+2$ regions $R_1, R_2, \cdots, R_{N+2}$ in $\operatorname{Re}(D)$. Let $c$ be a non-trivial coloring on $D$. Suppose $\delta$ is on an edge with non-zero color $i$. The Kauffman state sum $\langle D, c\rangle$ is defined by the following recipe.

\begin{enumerate}
\item Define the local contributions $M_{C_p}^{\triangle}$ and $A_{C_p}^{\triangle}$ associated to each corner $\triangle$ around the crossing $C_p$ as in Fig. ~\ref{fig:e1}. Here $M_{C_p}^{\triangle}$ does not depends on the coloring, while $A_{C_p}^{\triangle}$ does.
\begin{figure}[h!]
\begin{tikzpicture}[baseline=-0.65ex, thick, scale=0.9]
\draw (-1,-1) [->] to (1,1);
\draw (1,-1) -- (0.2,-0.2);
\draw (-0.2,0.2) [->] to (-1,1);
\draw (0, 0.5) node {$-1$};
\draw (0, -0.5) node {$1$};
\draw (0.5, 0) node {$1$};
\draw (-0.5, 0) node {$1$};
\end{tikzpicture}\hspace{1.3cm}
\begin{tikzpicture}[baseline=-0.65ex, thick, scale=0.9]
\draw (1,-1) [->] to (-1,1);
\draw (-1,-1) -- (-0.2,-0.2);
\draw (0.2,0.2) [->] to (1,1);
\draw (0, 0.5) node {$-1$};
\draw (0, -0.5) node {$1$};
\draw (0.5, 0) node {$1$};
\draw (-0.5, 0) node {$1$};
\end{tikzpicture}
\hspace{1.3cm}
\begin{tikzpicture}[baseline=-0.65ex, thick, scale=0.9]
\draw (0, 0.5) ellipse (1.5cm and 0.8cm);
\draw (0,-1) [->-] to (0,-0.3);
\draw (-0.4, -0.6) node {$1$};
\draw (0.4, -0.6) node {$1$};
\draw (0, 0) node {$1$};
\end{tikzpicture}\\
\vspace{5mm}

\begin{tikzpicture}[baseline=-0.65ex, thick, scale=0.9]
\draw (-1,-1) [->] to (1,1);
\draw (1.3, 1) node {$i$};
\draw (1,-1) -- (0.2,-0.2);
\draw (-0.2,0.2) [->] to (-1,1);
\draw (0, 0.6) node {$t^{-i}$};
\draw (0, -0.5) node {$1$};
\draw (0.5, 0) node {$1$};
\draw (-0.6, 0) node {$t^{-i}$};
\end{tikzpicture}\hspace{1cm}
\begin{tikzpicture}[baseline=-0.65ex, thick, scale=0.9]
\draw (1,-1) [->] to (-1,1);
\draw (-1.3, 1) node {$i$};
\draw (-1,-1) -- (-0.2,-0.2);
\draw (0.2,0.2) [->] to (1,1);
\draw (0, 0.5) node {$t^{i}$};
\draw (0, -0.5) node {$1$};
\draw (0.5, 0) node {$t^{i}$};
\draw (-0.5, 0) node {$1$};
\end{tikzpicture}
\hspace{1cm}
\begin{tikzpicture}[baseline=-0.65ex, thick, scale=0.9]
\draw (0, 0.5) ellipse (1.5cm and 0.8cm);
\draw (0,-1) [->-] to (0,-0.3);
\draw (-0.5, -0.6) node {$t^{-i/2}$};
\draw (0.5, -0.6) node {$t^{i/2}$};
\draw (0, 0.2) node {$t^{-i/2}-t^{i/2}$};
\draw (0.2, -1.2) node {$i$};
\end{tikzpicture}
	\caption{The local contributions $M_{C_p}^{\triangle}$ (top) and $A_{C_p}^{\triangle}$ (bottom). Here $i$ indicates the color of the corresponding edge.}
	\label{fig:e1}
\end{figure}

\item For each state $s\in S(D, \delta)$, let $$M(s):=\prod_{p=1}^{N}M_{C_p}^{s(C_{p})},$$
$$A(s):=\prod_{p=1}^{N}A_{C_p}^{s(C_{p})}. $$

\item The {\it Kauffman state sum} is defined as \begin{equation}
\label{alexander}
\langle D, c\rangle:=\vert \delta\vert^{-1}\sum_{s\in S(D, \delta)} M(s)\cdot A(s).
\end{equation}

\end{enumerate}
For all the cases that $S(D, \delta)=\emptyset$, we let $\langle D, c\rangle =0$.
\end{defn}

\medskip

\begin{rem}
\rm
While a singular knot can be regarded as a special MOY graph, the local contribution defined in Ozsv{\'a}th-Stipsicz-Szab{\'o} \cite[Figure 6]{MR2529302} for the Alexander polynomial of a singular knot is different from the aforementioned construction in Fig. ~\ref{fig:e1}.  The requirement of invariance under Reidemeister moves for MOY graphs makes it difficult to adopt a direct generalization from the former.
\end{rem}

\medskip

\begin{ex}
\rm
Continuing with our previous example, we find a total of three states $s_1, s_2, s_3$ and represent them by $\bullet$'s in Fig. \ref{fig:e5state}.  Denote $\{k\}=t^{k/2}-t^{-k/2}$.  Table \ref{table1} summarizes the computation of each state using Fig. \ref{fig:e1}.
\begin{table}[h!]
\label{table1}
\begin{tabular}{|c|c|c|c|} \hline
state & $M(s)$ & $A(s)$ & $M(s)A(s)$\\ \hline \hline
$s_1$ & $-1$ & $\displaystyle \{i+j\}\{i\} t^{{j}/{2}} t^{i}$ & $-\{i+j\}\{i\} t^{i+j/{2}}$ \\ \hline
$s_2$ & $1$ & $\{i+j\}\{i\} t^{{-j}/{2}} t^{i}$ & $\{i+j\}\{i\} t^{i-j/{2}}$\\ \hline
$s_3$ & $1$ & $\{i+j\}\{j\} t^{{i}/{2}} t^{i}$ & $\{i+j\}\{j\} t^{3i/2}$\\ \hline
\end{tabular}
\bigskip
 \caption{All Kauffman states and the corresponding $M(s)$ and $A(s)$.   }
  \label{table1}
\end{table}

Since $\vert \delta \vert=1-t^j$, we have
\begin{eqnarray*}
\langle D, c\rangle &=& (1-t^j)^{-1}[-\{i+j\}\{i\} t^{i+j/{2}}+\{i+j\}\{i\} t^{i-j/{2}}+\{i+j\}\{j\} t^{3i/2}]\\
&= & \frac{\{i+j\}[-\{i\}\{j\}t^{i}+\{j\}t^{3i/2}]}{-\{j\}t^{j/2}}\\
&=& -\{i+j\}t^{(i-j)/2}.
\end{eqnarray*}

\begin{figure}[h!]
\begin{tikzpicture}[baseline=-0.65ex, thick, scale=1.2]
\draw (0,-1)  to [out=90,in=270] (0.5,-0.33);
\draw (0,-1) to [out=270,in=180] (1.5,-2);
\draw (1.5,-2) to [out=0,in=0] (1.5,1);
\draw (0.5, -0.33) [->-] to [out=90,in=270] (0.5,0.33);
\draw (0.5, 0.33) [->] to [out=90,in=0] (-0.5,1);
\draw (1,-1)  to [out=90,in=270] (0.5,-0.33);
\draw (1,-1) to [out=270,in=45] (0.6,-1.6);
\draw (0.3,-1.8) to [out=225,in=0] (-0.5,-2);
\draw (-0.5,-2) to [out=180,in=180] (-0.5,1);
\draw (1.5,1) [<-] to [out=180,in=90] (0.5,0.33);
\draw (0.5, -0.33) node[circle,fill,inner sep=1pt]{};
\draw (0.5, -0.33) circle (0.3);
\draw (0.5, 0.33) node[circle,fill,inner sep=1pt]{};
\draw (0.5, 0.33) circle (0.3);
\draw (-1.33,0) node {$*$};
\draw (-1,0) node {$\delta$};
\draw (-0.5,-0.5) node {$\star$};
\draw (-1.7,-0.5) node {$\star$};
\draw (0.85,-0.6) node {$\bullet$};
\draw (0.5,0.15) node {$\bullet$};
\draw (0.45,-0.5) node {$\bullet$};
\draw (0.5,-1.5) node {$\bullet$};
\draw (0,1.25) node {$j$};
\draw (1,1.25) node {$i$};
\draw (1.5,0) node {$i+j$};
\draw (0.5,-2.5) node {$s_1$};
\end{tikzpicture}\hspace{5mm}
\begin{tikzpicture}[baseline=-0.65ex, thick, scale=1.2]
\draw (0,-1)  to [out=90,in=270] (0.5,-0.33);
\draw (0,-1) to [out=270,in=180] (1.5,-2);
\draw (1.5,-2) to [out=0,in=0] (1.5,1);
\draw (0.5, -0.33) [->-] to [out=90,in=270] (0.5,0.33);
\draw (0.5, 0.33) [->] to [out=90,in=0] (-0.5,1);
\draw (1,-1) to [out=90,in=270] (0.5,-0.33);
\draw (1,-1) to [out=270,in=45] (0.6,-1.6);
\draw (0.3,-1.8) to [out=225,in=0] (-0.5,-2);
\draw (-0.5,-2) to [out=180,in=180] (-0.5,1);
\draw (1.5,1) [<-] to [out=180,in=90] (0.5,0.33);
\draw (0.5, -0.33) node[circle,fill,inner sep=1pt]{};
\draw (0.5, -0.33) circle (0.3);
\draw (0.5, 0.33) node[circle,fill,inner sep=1pt]{};
\draw (0.5, 0.33) circle (0.3);
\draw (-1.33,0) node {$*$};
\draw (-1,0) node {$\delta$};
\draw (-0.5,-0.5) node {$\star$};
\draw (-1.7,-0.5) node {$\star$};
\draw (0.6,-0.8) node {$\bullet$};
\draw (0.5,0.15) node {$\bullet$};
\draw (0.45,-0.5) node {$\bullet$};
\draw (0.8,-1.7) node {$\bullet$};
\draw (0.5,-2.5) node {$s_2$};
\end{tikzpicture}
\hspace{5mm}
\begin{tikzpicture}[baseline=-0.65ex, thick, scale=1.2]
\draw (0,-1)  to [out=90,in=270] (0.5,-0.33);
\draw (0,-1) to [out=270,in=180] (1.5,-2);
\draw (1.5,-2) to [out=0,in=0] (1.5,1);
\draw (0.5, -0.33) [->-] to [out=90,in=270] (0.5,0.33);
\draw (0.5, 0.33) [->] to [out=90,in=0] (-0.5,1);
\draw (1,-1)  to [out=90,in=270] (0.5,-0.33);
\draw (1,-1) to [out=270,in=45] (0.6,-1.6);
\draw (0.3,-1.8) to [out=225,in=0] (-0.5,-2);
\draw (-0.5,-2) to [out=180,in=180] (-0.5,1);
\draw (1.5,1) [<-] to [out=180,in=90] (0.5,0.33);
\draw (0.5, -0.33) node[circle,fill,inner sep=1pt]{};
\draw (0.5, -0.33) circle (0.3);
\draw (0.5, 0.33) node[circle,fill,inner sep=1pt]{};
\draw (0.5, 0.33) circle (0.3);
\draw (-1.33,0) node {$*$};
\draw (-1,0) node {$\delta$};
\draw (-0.5,-0.5) node {$\star$};
\draw (-1.7,-0.5) node {$\star$};
\draw (0.4,-0.8) node {$\bullet$};
\draw (0.5,0.15) node {$\bullet$};
\draw (0.55,-0.5) node {$\bullet$};
\draw (0.8,-1.7) node {$\bullet$};
\draw (0.5,-2.5) node {$s_3$};
\end{tikzpicture}
\caption{The three states for the decorated diagram.}\label{fig:e5state}
\end{figure}

\end{ex}

\vspace{3mm}

\subsection{Relation with the Alexander matrix}

Suppose $(D, \delta)$ is a connected decorated diagram with $N$ crossings $C_1, C_2, \cdots, C_N$ in $\operatorname{Cr}(D)$ and $N+2$ regions $R_1, R_2, \cdots, R_{N+2}$ in $\operatorname{Re}(D)$. Let $c$ be the coloring on $D$.

\begin{defn} \label{def:Alexmatrix}
\rm
The {\it Alexander matrix} is an $N\times (N+2)$ matrix, denoted $A(D,c)$, whose $i$-th row corresponds to the crossing $C_i$ and the $j$-th column corresponds to the region $R_j$. The $(i,j)$-entry of $A(D, c)$ is the sum of $m_{C_i}^{\triangle}A_{C_i}^{\triangle}$ for $\triangle$ belonging to $R_j$, where $m_{C_i}^{\triangle}$ is a local contribution specified in Fig. \ref{fig:e4}.
\end{defn}

\begin{figure}[h!]
\begin{tikzpicture}[baseline=-0.65ex, thick, scale=0.9]
\draw (-1,-1) [->] to (1,1);
\draw (1,-1) -- (0.2,-0.2);
\draw (-0.2,0.2) [->] to (-1,1);
\draw (0, 0.5) node {$1$};
\draw (0, -0.5) node {$1$};
\draw (0.5, 0) node {$-1$};
\draw (-0.6, 0) node {$-1$};
\end{tikzpicture}\hspace{1.5cm}
\begin{tikzpicture}[baseline=-0.65ex, thick, scale=0.9]
\draw (1,-1) [->] to (-1,1);
\draw (-1,-1) -- (-0.2,-0.2);
\draw (0.2,0.2) [->] to (1,1);
\draw (0, 0.5) node {$-1$};
\draw (0, -0.5) node {$-1$};
\draw (0.5, 0) node {$1$};
\draw (-0.5, 0) node {$1$};
\end{tikzpicture}
\hspace{1.5cm}
\begin{tikzpicture}[baseline=-0.65ex, thick, scale=0.9]
\draw (0, 0.5) ellipse (1.5cm and 0.8cm);
\draw (0,-1) [->-] to (0,-0.3);
\draw (-0.5, -0.6) node {$-1$};
\draw (0.4, -0.6) node {$1$};
\draw (0, 0) node {$1$};
\end{tikzpicture}
	\caption{The local contribution $m_{C_p}^{\triangle}$, where $C_p$ is a crossing and $\triangle$ is a corner around $C_p$.}
	\label{fig:e4}
\end{figure}

\begin{ex}
\rm
Continuing with the previous example, we label  the $4$ crossings and $6$ regions as in Fig. \ref{fig:exAlexandermatrix}. We see that $R_5$ and $R_6$ are circle regions, and $C_1, C_2, C_3$ are crossings of type  \diaCircle while $C_4$ is a \diaCrossN crossing.

\begin{figure}[h!]
\begin{tikzpicture}[baseline=-0.65ex, thick, scale=1.5]
\draw (0,-1)  to [out=90,in=270] (0.5,-0.33);
\draw (0,-1) to [out=270,in=180] (1.5,-2);
\draw (1.5,-2) to [out=0,in=0] (1.5,1);
\draw (0.5, -0.33) [->-] to [out=90,in=270] (0.5,0.33);
\draw (0.5, 0.33) [->] to [out=90,in=0] (-0.5,1);
\draw (1,-1)  to [out=90,in=270] (0.5,-0.33);
\draw (1,-1) to [out=270,in=45] (0.6,-1.6);
\draw (0.3,-1.8) to [out=225,in=0] (-0.5,-2);
\draw (-0.5,-2) to [out=180,in=180] (-0.5,1);
\draw (1.5,1) [<-] to [out=180,in=90] (0.5,0.33);
\draw (0.5, -0.33) node[circle,fill,inner sep=1pt]{};
\draw (0.5, -0.36) ellipse (0.3 and 0.2);
\draw (0.5, 0.33) node[circle,fill,inner sep=1pt]{};
\draw (0.5, 0.3) ellipse (0.3 and 0.2);
\draw (-0.25,-0.4) node {$R_6$};
\draw (-1.7,-0.5) node {$R_1$};
\draw (-1.4,0.5) node {$j$};
\draw (2.4,0.5) node {$i$};
\draw (0.8, 0.05) node {$\longleftarrow$};
\draw (1.2,0.05) node {$C_1$};
\draw (0.9, -0.6) node {$\longleftarrow$};
\draw (0.4, -0.7) node {$\uparrow$};
\draw (0.4,-1) node {$C_2$};
\draw (1.3,-0.6) node {$C_3$};
\draw (0.8, -1.7) node {$\longleftarrow$};
\draw (1.2,-1.7) node {$C_4$};
\draw (0.4,-1.4) node {$R_3$};
\draw (1.7,-1.3) node {$R_4$};
\draw (-0.5,-1.5) node {$R_2$};
\draw (-0.25, 0.3) node {$R_5$};
\draw (0.2, 0.3) node {$\longrightarrow$};
\draw (0.2, -0.4) node {$\longrightarrow$};
\end{tikzpicture}
\caption{Calculating the Alexander matrix with the given labeling of the regions and crossings.}
\label{fig:exAlexandermatrix}
\end{figure}

Now we calculate the Alexander matrix $A$. By definition, if $C_i$ and $R_j$ are not adjacent to each other, the $(i,j)$-entry $a_{ij}$ of $A$ is zero. Therefore $a_{11}=a_{13}=a_{16}=a_{21}=a_{24}=a_{25}=a_{31}=a_{32}=a_{35}=a_{45}=a_{46}=0$.

Since the crossing $C_1$ meets $R_2$ only at its west corner, we have $a_{12}=m_{C_1}^{W}A_{C_1}^{W}=-1\cdot t^{-(i+j)/2}$. Similarly we can calculate the other entries. The matrix is
\[A=\begin{pmatrix}
0 & -t^{-(i+j)/2} & 0 & t^{(i+j)/2} & t^{-(i+j)/2}-t^{(i+j)/2} & 0\\
0 & -t^{-i/2} & t^{i/2} & 0 & 0 & t^{-i/2}-t^{i/2} \\
0 & 0 & -t^{-j/2} & t^{j/2} & 0 & t^{-j/2}-t^{j/2}\\
-1& 1 & -t^{i} & t^{i} & 0 & 0 \\
\end{pmatrix}.
\]
\end{ex}

\vspace{3mm}

Suppose the base point $\delta$ is adjacent to two regions $R_u$ and $R_v$. Let $A(D, c)\backslash(u, v)$ denote the square matrix obtained from the Alexander matrix $A(D, c)$ by removing the $u$-th and $v$-th columns, which correspond to $R_u$ and $R_v$, respectively. Calculating the determinant of $A(D, c)\backslash(u, v)$ we see that
\begin{equation}
\label{old}
\det A(D, c)\backslash(u, v) = \sum_{s\in S(D, \delta)}\mathrm{sign}(s)\cdot m(s)\cdot A(s),
\end{equation}
where $S(D, \delta)$ is the same set of Kauffman states as before, and $\mathrm{sign}(s)$ is the sign of the state $s$ as a permutation with respect to the given orders of the elements in $\operatorname{Cr}(D)$ and $\operatorname{Re}(D) \backslash \{R_u, R_{v}\}$.  The relation of the old state sum (\ref{old}) and $\langle D, c\rangle$ is as follows.

\begin{prop}[Appendix]
\label{equivalence}
For any $s\in S(D, \delta)$, the value $$\frac{M(s)}{\mathrm{sign}(s)\cdot m(s)}\in \{1, -1\}$$ does not depend on the choice of the state $s$.  As a result,
$$\langle D, c\rangle= \pm \vert \delta \vert^{-1} \cdot \det A(D, c)\backslash(u, v),$$
where the sign $\pm$ only depends on the ordering in $\operatorname{Cr}(D)$ and $\operatorname{Re}(D) \backslash \{R_u, R_{v}\}$.
\end{prop}

\begin{rem}
\rm
Besides the above combinatorial definition of $A(D, c)$, the Alexander matrix also has the following topological description. Let $G$ be the spatial graph in $S^3$ determined by the graph diagram $D$. Denote $X$ the complement of $G$ in $S^3$. Using the coloring $c$ we can construct a map $$\psi_c: \pi_1(X) \to \mathbb{Z},$$ where the meridian of each edge is sent to the color of the edge. One can construct an abelian covering space $p: \widetilde X \to X$ from $\psi_c$. Then, $A(D, c)$ is a presentation matrix for the module $H_1 (\widetilde X, p^{-1} (\partial_{\mathrm{in}} (X)))$, where $\partial_{\mathrm{in}}(X)\subset \partial X$ is a certain submanifold associated with the transverse orientation of $D$ and $\delta$. See \cite{bao} for more details, where the state sum \eqref{old} was first proposed.
\end{rem}

\vspace{3mm}

\subsection{The choice of the base point}
In this section, we prove the independence of $\langle D,c \rangle$ on the choice of the base point $\delta$.

\begin{prop}
\label{initial}
The state sum $\langle D,c \rangle$ does not depend on the choice of the location of $\delta$ among the edges with non-zero colors.
\end{prop}

\begin{proof}
The proof is adapted from \cite[Theorem 3.3]{MR712133}.
Suppose $\delta$ and $\delta'$ are two different choices of base points. Suppose that $R_u, R_v$ are the two regions adjacent to $\delta$ and $R_u$ is on the left-hand side when traversing the edge containing $\delta$ in its orientation, and that $R_w, R_z$ are the two regions adjacent to $\delta'$. Let $\langle D,c \rangle_{\delta}$ be the state sum calculated by using $\delta$. Our goal is to show that $\langle D,c \rangle_{\delta}=\langle D,c \rangle_{\delta'}$. The strategy is as follows. We first prove that $\langle D,c \rangle_{\delta}=\pm\langle D,c \rangle_{\delta'}$ using the Alexander matrix in Definition \ref{def:Alexmatrix}. Then we remove the ambiguity of the sign.

Without loss of generality, we assume that $R_1, R_2, \cdots, R_{M}$ are the regular regions and $R_{M+1}, R_{M+2}, \cdots, R_{N+2}$ are the circle regions in $\operatorname{Re}(D)$. Consider the Alexander matrix $A=A(D,c)$. Let $A_q$ be the $q$-th column corresponding to the region $R_q$. We claim that $\sum_{q=1}^{N+2}A_q=0$ since for each crossing $C_p$ the sum of $m_{C_p}^{\triangle}A_{C_p}^{\triangle}$ over all corners of $C_p$ is zero. From this identity we have \[ \det A \backslash(u, v)=(-1)^{w-v}\det A \backslash(u, w)-(-1)^{w-u}\det A \backslash(v, w) \tag{$\ast$} \label{eq}\] for $1\leq u< v< w \leq M$.

On the other hand, we have $\sum_{q=1}^{M}t^{\operatorname{ind}(R_q)}A_q=0$. Therefore $$\sum_{q=1}^{M}t^{\operatorname{ind}(R_q)-\operatorname{ind}(R_u)}A_q=0$$ for any fixed $1\leq u \leq M$.

As a result, for any $1\leq u, w \leq M$, we have
$$-t^{\operatorname{ind}(R_w)-\operatorname{ind}(R_u)}A_w=\sum_{\substack{1\leq q\leq M \\ q \neq w}}t^{\operatorname{ind}(R_q)-\operatorname{ind}(R_u)}A_q.$$ Bringing the scalar $-t^{\operatorname{ind}(R_w)-\operatorname{ind}(R_u)}$ to the column $A_w$ in $A\backslash(u, v)$, we get
\begin{align*}
&-t^{\operatorname{ind}(R_w)-\operatorname{ind}(R_u)}\det A\backslash(u, v)\\
&=(-1)^{w-u}\det A\backslash(v, w)-t^{\operatorname{ind}(R_v)-\operatorname{ind}(R_u)}(-1)^{w-v}\det A\backslash(u, w)\\
\end{align*}
for $1\leq u< v< w \leq M$. This together with (\ref{eq}) implies that
$$(t^{\operatorname{ind}(R_w)}-t^{\operatorname{ind}(R_v)})\det A\backslash(u, v)=(-1)^{w-u}(t^{\operatorname{ind}(R_v)}-t^{\operatorname{ind}(R_u)})\det A\backslash(v, w).$$ Applying the above relation twice, we end up with
$$(t^{\operatorname{ind}(R_z)}-t^{\operatorname{ind}(R_w)})\det A\backslash(u, v)=(-1)^{(w+z)-(u+v)}(t^{\operatorname{ind}(R_v)}-t^{\operatorname{ind}(R_u)})\det A\backslash(w, z),$$ for any $1\leq u< v\leq w < z \leq M$.

By definition we have $\vert \delta \vert =t^{\operatorname{ind}(R_v)}-t^{\operatorname{ind}(R_u)}$ and $\vert \delta' \vert =t^{\operatorname{ind}(R_z)}-t^{\operatorname{ind}(R_w)}$. Choose an ordering of crossings in $\operatorname{Cr}(D)$ so that $\det A\backslash(u, v)=\vert \delta\vert \langle D,c\rangle_{\delta}$. Then we have $\det A\backslash(w, z)=(-1)^{\epsilon}\vert \delta'\vert \langle D,c\rangle_{\delta'}$ for $\epsilon=0 \text{ or } 1$. The sign $\epsilon$ depends only on the ordering of the regions and crossings.  In summary, we have proved that
\[\langle D, c\rangle_{\delta}=(-1)^{(w+z)-(u+v)+\epsilon} \langle D, c\rangle_{\delta'}=\pm \langle D, c\rangle_{\delta'} \tag{$\ast\ast$} \label{eqq}.\]

Next, we remove ``$\pm$" from \eqref{eqq}. Let us first consider the case that $D$ is a plane diagram. As $\langle D,c\rangle_{\delta}$ is zero if and only if $\langle D, c\rangle_{\delta'}$ is zero, we assume that $\langle D,c\rangle_{\delta}$ is non-zero.  Note that there are only crossing of type \diaCircle for a plane diagram.
From Formula (\ref{alexander}) of $\langle D,c\rangle$, we see that the coefficient of the lowest degree term of $\vert \delta\vert\langle D,c\rangle_{\delta}$ is positive. Since $\operatorname{ind}(R_u)>\operatorname{ind}(R_v)$ and $\operatorname{ind}(R_w)>\operatorname{ind}(R_z)$, the coefficients of the lowest degree terms of both $\vert \delta' \vert \vert \delta \vert \langle D, c\rangle_{\delta}$ and $\vert \delta \vert \vert \delta' \vert  \langle D,c\rangle_{\delta'}$ are also positive. This together with (\ref{eqq}) implies that $$\langle D, c\rangle_{\delta}=\langle D, c\rangle_{\delta'}.$$ Hence, the location of $\delta$ does not affect the value of $\langle D,c\rangle$ when $D$ is a plane MOY diagram.

For an arbitrary MOY graph diagram that possibly has crossings of type \diaCrossP or \diaCrossN, we connect it with plane graphs by the following skein-type relations.
\begin{align*}
&\left<\begin{tikzpicture}[baseline=-0.65ex, thick, scale=0.5]
\draw (1,-1)   -- (0.2,-0.2);
\draw (-1, -1) [->] -- (1, 1) node[above]{$i$};
\draw (-0.2,0.2) [->] to (-1,1)  node[above]{$j$};
\end{tikzpicture} \right>
  =\frac{-t^{-\frac{i+j}{2}}}{\{i\}\{j\}}\cdot
\left<
\begin{tikzpicture}[baseline=-0.65ex, thick, scale=0.5]
\draw (0,-2) node[below]{$i$} [->-] to (0, 0);
\draw (0,0) --  (0, 1);
\draw (0, 1) [->] to (0,2) node[above]{$j$};
\draw (0,1) node[circle,fill,inner sep=1pt]{};
\draw (2,0) node[circle,fill,inner sep=1pt]{};
\draw (2,-2) [->-] node[below]{$j$} to (2,0);
\draw (2,0) [->] to (2,2) node[above]{$i$};
\draw (2,-0) [->-] to (0, 1);
\draw (1, -0.3) node {$j-i$};
\end{tikzpicture}\right>
- \, \frac{t^{-\frac{j}{2}}}{\{i\}}\cdot
\left<\begin{tikzpicture}[baseline=-0.65ex, thick, scale=0.5]
\draw (-1,-1)  node[below]{$i$} -- (0,0);
\draw (1,-1) node[below]{$j$} -- (0,0);
\draw (0,0) [->] -- (-1, 1) node[above]{$j$};
\draw (0,0) [->] to (1,1)  node[above]{$i$};
\draw (0,0) node[circle,fill,inner sep=1pt]{};
\end{tikzpicture}\right>,\\
&\left<
\begin{tikzpicture}[baseline=-0.65ex, thick, scale=0.5]
\draw (1,-1) [->]  -- (-1,1)  node[above]{$j$};
\draw (-1, -1) to (-0.2, -0.2);
\draw (0.2, 0.2) [->] -- (1, 1) node[above]{$i$};
\end{tikzpicture}\right>
  =\frac{-t^{\frac{i+j}{2}}}{\{i\}\{j\}}\cdot
\left<
\begin{tikzpicture}[baseline=-0.65ex, thick, scale=0.5]
\draw (0,-2) node[below]{$i$} [->-] to (0, 0);
\draw (0,0) --  (0, 1);
\draw (0, 1) [->] to (0,2) node[above]{$j$};
\draw (0,1) node[circle,fill,inner sep=1pt]{};
\draw (2,0) node[circle,fill,inner sep=1pt]{};
\draw (2,-2) [->-] node[below]{$j$} to (2,0);
\draw (2,0) [->] to (2,2) node[above]{$i$};
\draw (2,-0) [->-] to (0, 1);
\draw (1, -0.3) node {$j-i$};
\end{tikzpicture}\right>
- \, \frac{t^{\frac{j}{2}}}{\{i\}}\cdot
\left<\begin{tikzpicture}[baseline=-0.65ex, thick, scale=0.5]
\draw (-1,-1)  node[below]{$i$} -- (0,0);
\draw (1,-1) node[below]{$j$} -- (0,0);
\draw (0,0) [->] -- (-1, 1) node[above]{$j$};
\draw (0,0) [->] to (1,1)  node[above]{$i$};
\draw (0,0) node[circle,fill,inner sep=1pt]{};
\end{tikzpicture}\right>,\\
\end{align*}
where $j\geq i\geq 0$ and $\{i\}=t^{i/2}-t^{-i/2}$. For $j\leq i$ we can also establish similar relations. Such kinds of relations will be discussed in great details in Section \ref{moysection}, so we omit their proofs here.  Note that all the coefficients appearing on the right-hand side of the equations do not depend on $\delta$, and $\vert \delta \vert$ is identical for all the graphs showing here.
As we just proved that $\langle D,c\rangle$ is independent of the choice of $\delta$ in the case of plane graphs, we conclude that the same is true for general MOY graph diagrams from the relations above.

\end{proof}


\section{The Alexander polynomial} 
In this section, we study the invariance of the Kauffman state sum under the Reidemeister moves. This enables us to define an Alexander polynomial $\Delta_{(G, c)}(t)$ for all MOY graphs.  Then we provide a normalization $\Delta_{(\mathbb{G}, c)}(t)$ for a framed trivalent graph $\mathbb{G}$.

\subsection{The invariance of the Kauffman state sum $\langle D, c\rangle$}

\begin{prop}
\label{invtheo}
The Kauffman state sum $\langle D , c \rangle $ is invariant under the Reidemeister moves (II) -- (V)  in Fig. \ref{fig:e25}, and its variations under Reidemeister move (I) are given as below.
\begin{equation*}\label{r1}
t^{i} \left< \begin{tikzpicture}[baseline, thick, scale=0.4]
\draw (1,-1)   -- (0.2,-0.2);
\draw (-1, -1) [->] -- (1, 1) node[above]{$i$};
\draw (-0.2,0.2) --  (-1,1) ;
\draw (-1, 1) arc (90:270:1);
\end{tikzpicture} \right >
=\left< \begin{tikzpicture}[baseline=-0.65ex, thick, scale=0.4]
\draw (1,-1)   -- (0.2,-0.2);
\draw (-1, -1) -- (1, 1) ;
\draw (-0.2,0.2)  [->]--  (-1,1)node[above]{$i$} ;
\draw (1,-1) arc (-90:90:1);
\end{tikzpicture} \right >
=\left< \begin{tikzpicture}[baseline=-0.65ex, thick, scale=0.4]
\draw (-1,-1)   -- (-0.2,-0.2);
\draw (1, -1) -- (-1, 1) ;
\draw (0.2,0.2) [->] --  (1,1) node[above]{$i$};
\draw (-1, 1) arc (90:270:1);
\end{tikzpicture} \right >
=t^{-i} \left< \begin{tikzpicture}[baseline=-0.65ex, thick, scale=0.4]
\draw (-1,-1)   -- (-0.2,-0.2);
\draw (1, -1) [->] -- (-1, 1) node[above]{$i$};
\draw (0.2,0.2)  --  (1,1);
\draw (1,-1) arc (-90:90:1);
\end{tikzpicture} \right >
=\left< \begin{tikzpicture}[baseline=-0.65ex, thick, scale=0.4]
\draw (0, -1) [->] -- (0, 1) node[above]{$i$};
\end{tikzpicture} \right >.
\end{equation*}

\end{prop}

\begin{proof}

The main trick is to apply Proposition \ref{initial} and choose the location of the base point $\delta$ that is most convenient for calculations.

For move (I), we assume that it occurs away from the base point. Take the following figure for example, where the left-hand diagram is $D$ and the right-hand one is $D'$. There is a one-to-one correspondence between the states of $D$ and those of $D'$ since the additional region of $D$ has only one corner. It then follows from (\ref{alexander}) and the local contributions described in Fig. \ref{fig:e1} that $\langle D, c \rangle=t^{i}\langle D' , c \rangle$. The proofs for the other cases of move (I) follow from similar arguments.
\begin{figure}[h!]
\begin{tikzpicture}[baseline=-0.65ex, thick, scale=0.4]
\draw (-1,-1)   -- (-0.2,-0.2);
\draw (1, -1) [->] -- (-1, 1) node[above]{$i$};
\draw (0.2,0.2)  --  (1,1);
\draw (1,-1) arc (-90:90:1);
\draw (0.6, 0) node {$\bullet$};
\end{tikzpicture} $\longleftrightarrow$
\begin{tikzpicture}[baseline=-0.65ex, thick, scale=0.4]
\draw (0, -1) [->] -- (0, 1) node[above]{$i$};
\end{tikzpicture}
\end{figure}

The proof of the invariance of $\langle D, c \rangle$ under moves (II), (III) is similar to the proof of \cite[Theorem 4.3]{MR712133}. For move (II), we place the base point in a position as below. Suppose the left-hand graph is $D$ and the right-hand one is $D'$.  We split the set $S(D, \delta)$ as a disjoint union of $A$ and $B$, where $B$ consists of all the states which assign the bottom crossing of $D$ to its south corner and $A$ is the complement of $B$. The contribution of $A$ to $\langle D , c\rangle$ is $0$ since the states are paired up (white states and black states) and the contribution from each pair adds up to $0$. There is a one-to-one correspondence between $B$ and $S(D', \delta)$.
Since for each state in $B$ the contribution from the two crossings involved is $1$, we conclude that $\langle D , c\rangle=\langle D', c\rangle$.

\begin{figure}[h!]
\begin{tikzpicture}[baseline, thick, scale=0.4]
\draw (-1, 2) arc (90:270:2);
\draw (-3.5, 2) arc (90:60:2);
\draw (-1.5, 0) arc (0:40:2);
\draw (-1.5, 0) arc (0:-40:2);
\draw (-3.5, -2) arc (270:300:2);
\draw (-2.25, 1) node[circle,fill,inner sep=1.5pt]{};
\draw (-2.25, -1) node {$\circ$};
\draw (-1.6, 1.4) node {$\circ$};
\draw (-1.6, -1.4) node[circle,fill,inner sep=1.5pt]{};
\draw (-3.5, 2) node {$\ast$};
\draw (-3.5, 1.3) node {$\delta$};
\draw (-5, 0) node {$A:$};
\end{tikzpicture} \quad  $\bigcup$ \quad
\begin{tikzpicture}[baseline, thick, scale=0.4]
\draw (-1, 2) arc (90:270:2);
\draw (-3.5, 2) arc (90:60:2);
\draw (-1.5, 0) arc (0:40:2);
\draw (-1.5, 0) arc (0:-40:2);
\draw (-3.5, -2) arc (270:300:2);
\draw (-2.25, 1) node[circle,fill,inner sep=1.5pt]{};
\draw (-2.25, -2) node[circle,fill,inner sep=1.5pt]{};
\draw (-3.5, 2) node {$\ast$};
\draw (-3.5, 1.3) node {$\delta$};
\draw (-5, 0) node {$B:$};
\end{tikzpicture}$\quad \longleftrightarrow$ \quad
\begin{tikzpicture}[baseline, thick, scale=0.4]
\draw (-1, 2) arc (110:250:2);
\draw (-5, 2) arc (70:-70:2);
\draw (-5, 2) node {$\ast$};
\draw (-5, 1.3) node {$\delta$};
\end{tikzpicture}
\end{figure}

For move (III), we consider the move in the following figure for illustration, and the invariance under other cases can be proved similarly. We assume that $\delta$ is located on the edge with color $i$. Around the triangle where the move occurs, we have three crossings and five local regions not adjacent to $\delta$ (some of them can be adjacent to $\delta$ or belong to the same region globally), so for each state there are exactly two blank local regions in the sense that there is no crossing assigned to them around the triangle. Suppose the left-hand diagram is $D$ and the right-hand one is $D'$. We compare the states of $D$ and $D'$ that have the same blank local regions around the triangle and are identical outside the triangle.

\begin{figure}[h!]
\begin{tikzpicture}[baseline, thick, scale=0.4]
\draw (-3, 3) [<-] to (-0.4, 0.4) ;
\draw (-3.5, 1.5) [<-] to (-2.2, 1.5) ;
\draw (2.2, 1.5) -- (3.5, 1.5) ;
\draw (-1, 1.5) -- (1, 1.5) ;
\draw (0.4, -0.4) -- (2, -2) ;
\draw (3, 3) [<-] to (-2, -2) ;
\draw (-2, -2) node {$\ast$};
\draw (-2.2, -1.1) node {$\delta$};
\draw (3, 0) node {$R_1$};
\draw (0, 3) node {$R_2$};
\draw (-2.4, 1.9) node[circle,fill,inner sep=1.5pt]{};
\draw (0, 0.5) node[circle,fill,inner sep=1.5pt]{};
\draw (2.4, 1.9) node[circle,fill,inner sep=1.5pt]{};
\draw (3.5, 3) node {$i$};
\draw (3, -2) node {$j$};
\draw (4, 1.5) node {$k$};
\end{tikzpicture}\quad $\longleftrightarrow$ \quad
\begin{tikzpicture}[baseline, thick, scale=0.4]
\draw (-2, 3) [<-] to (-0.4, 1.4) ;
\draw (-3.5, -0.5) [<-] to (-2.2, -0.5) ;
\draw (2.2, -0.5) -- (3.5, -0.5) ;
\draw (-1, -0.5) -- (1, -0.5) ;
\draw (0.4, 0.6) -- (3, -2) ;
\draw (2, 3) [<-] to (-3, -2) ;
\draw (-3, -2) node {$\ast$};
\draw (-3.4, -1.3) node {$\delta$};
\draw (3.5, -1.2) node {$R_1$};
\draw (0, 2.5) node {$R_2$};
\draw (-1.8, -0.1) node[circle,fill,inner sep=1.5pt]{};
\draw (0, 0.5) node[circle,fill,inner sep=1.5pt]{};
\draw (-0.5, 1) node {$\circ$};
\draw (1.8, -0.1) node[circle,fill,inner sep=1.5pt]{};
\draw (-0.7, -0.1) node {$\circ$};
\draw (-5, 0) node {$A:$};
\end{tikzpicture} \quad  $\bigcup$ \quad
\begin{tikzpicture}[baseline, thick, scale=0.4]
\draw (-2, 3) [<-] to (-0.4, 1.4) ;
\draw (-3.5, -0.5) [<-] to (-2.2, -0.5) ;
\draw (2.2, -0.5) -- (3.5, -0.5) ;
\draw (-1, -0.5) -- (1, -0.5) ;
\draw (0.4, 0.6) -- (3, -2) ;
\draw (2, 3) [<-] to (-3, -2) ;
\draw (-3, -2) node {$\ast$};
\draw (-3.4, -1.3) node {$\delta$};
\draw (3.5, -1.2) node {$R_1$};
\draw (0, 2.5) node {$R_2$};
\draw (0.5, 1) node[circle,fill,inner sep=1.5pt]{};
\draw (0.7, -0.1) node[circle,fill,inner sep=1.5pt]{};
\draw (-1.8, -0.1) node[circle,fill,inner sep=1.5pt]{};
\draw (-5, 0) node {$B:$};
\end{tikzpicture}
\end{figure}

For instance, consider the states whose blank regions are $R_1$ and $R_2$.
Each of such states of $D$ has a unique local assignment as in the left-hand figure. One can check that the contribution from the three crossings  is $t^{-(i+j)}$. On $D'$ there are three types of states with the indicated blank regions. The set $A$ comprises the first two types of states.  Their contribution to $\langle D' , c\rangle$ vanishes since the states are paired up (white states and black states) and the contribution from each pair is zero. For each state of $B$, the contribution from the three crossings is $t^{-(i+j)}$. As a result, for the states with the blank regions $R_1$ and $R_2$, their contributions to the state sum on both sides are the same. For the states with different blank regions, we can do similar calculations to verify that the contributions from both sides are identical as well. Consequently, we have  $\langle D ,c \rangle=\langle D' , c \rangle$.

For move (V), the twist creates a crossing on the diagram. If the two edges of the crossing are outgoing edges, we assume that $\delta$ is away from the twist. Since the newly created crossing must be assigned to its south corner whose contribution is $1$, we see that $\langle D,c\rangle$ is invariant. Now we consider the case that the two edges of the crossing are incoming edges. We assume that $\delta$ is located as in the figure below. Suppose the left-hand diagram is $D$ and the right-hand one is $D'$. The assignment of the two crossings of type \diaCircle on $D$ is unique, and their contribution to $\langle D , c\rangle$ is $t^{j/2}(t^{-i/2}-t^{i/2})$. On $D'$ there are three types of local assignments, and the sum of their contributions to $\langle D', c \rangle$ is $$t^{-i/2+j}(t^{-j/2}-t^{j/2})+t^{j/2+j}(t^{-i/2}-t^{i/2})-t^{i/2+j}(t^{-j/2}-t^{j/2})=t^{j/2}(t^{-i/2}-t^{i/2}).$$ Therefore we see that $\langle D , c \rangle=\langle D' , c \rangle$.

\begin{figure}[h!]
\begin{tikzpicture}[baseline=-0.65ex, thick, scale=1.4]
\draw (0, 0) [->]-- (0.5, 0.5);
\draw (0, 0) [->]-- (-0.5,0.5);
\draw (0, 0) [->]-- (0.2, 0.5);
\draw (0, 0) circle (0.3);
\draw (0, 0) node[circle,fill,inner sep=1pt]{};
\draw [dashed] (-0.5, 0)--(0.5, 0);
\draw (0.6, -0.6) node{$...$};
\draw (-0.6, -0.6) node{$...$};
\draw (-0.1, 0.4) node{$...$};
\draw (0,0) [-<-] to (0.5, -1);
\draw (0,0) [-<-] to (-0.5, -1);
\draw (-0.4, -0.8) node{$\ast$};
\draw (-0.7, -0.8) node{$\delta$};
\draw (-0.2, -0.8) node{$i$};
\draw (0.6, -0.8) node{$j$};
\draw (-0.1, -0.2) node[circle,fill,inner sep=1.5pt]{};
\draw (0.25, -0.3) node[circle,fill,inner sep=1.5pt]{};
\end{tikzpicture} \quad $\longleftrightarrow$ \quad
\begin{tikzpicture}[baseline=-0.65ex, thick, scale=1.4]
\draw (0, 0) [->]-- (0.5, 0.5);
\draw (0, 0) [->]-- (-0.5,0.5);
\draw (0, 0) [->]-- (0.2, 0.5);
\draw (0, 0) circle (0.3);
\draw (0, 0) node[circle,fill,inner sep=1pt]{};
\draw [dashed] (-0.5, 0)--(0.5, 0);
\draw (0.6, -0.6) node{$...$};
\draw (-0.6, -0.6) node{$...$};
\draw (-0.1, 0.4) node{$...$};
\draw (0,0) to [out=225,in=90] (-0.25,-0.4) to [out=270,in=315] (-0.1,-0.65);
\draw (0.5, -1) [->-] to (-0.1,-0.65);
\draw (-0.5, -1) [->-] to (-0.1, -0.75);
\draw (0,0) to [out=315,in=90] (0.25,-0.4) to [out=270,in=225] (0.1,-0.65);
\draw (-0.2, -0.8) node{$\ast$};
\draw (-0.5, -0.8) node{$\delta$};
\draw (-0.15, -0.15) node[circle,fill,inner sep=1.5pt]{};
\draw (0.15, -0.35) node[circle,fill,inner sep=1.5pt]{};
\draw (0.2, -0.7) node[circle,fill,inner sep=1.5pt]{};
\end{tikzpicture}   $\bigcup$
\begin{tikzpicture}[baseline=-0.65ex, thick, scale=1.4]
\draw (0, 0) [->]-- (0.5, 0.5);
\draw (0, 0) [->]-- (-0.5,0.5);
\draw (0, 0) [->]-- (0.2, 0.5);
\draw (0, 0) circle (0.3);
\draw (0, 0) node[circle,fill,inner sep=1pt]{};
\draw [dashed] (-0.5, 0)--(0.5, 0);
\draw (0.6, -0.6) node{$...$};
\draw (-0.6, -0.6) node{$...$};
\draw (-0.1, 0.4) node{$...$};
\draw (0,0) to [out=225,in=90] (-0.25,-0.4) to [out=270,in=315] (-0.1,-0.65);
\draw (0.5, -1) [->-] to (-0.1,-0.65);
\draw (-0.5, -1) [->-] to (-0.1, -0.75);
\draw (0,0) to [out=315,in=90] (0.25,-0.4) to [out=270,in=225] (0.1,-0.65);
\draw (-0.2, -0.8) node{$\ast$};
\draw (-0.5, -0.8) node{$\delta$};
\draw (0.15, -0.15) node[circle,fill,inner sep=1.5pt]{};
\draw (-0.15, -0.35) node[circle,fill,inner sep=1.5pt]{};
\draw (0.2, -0.7) node[circle,fill,inner sep=1.5pt]{};
\end{tikzpicture}
 $\bigcup$
\begin{tikzpicture}[baseline=-0.65ex, thick, scale=1.4]
\draw (0, 0) [->]-- (0.5, 0.5);
\draw (0, 0) [->]-- (-0.5,0.5);
\draw (0, 0) [->]-- (0.2, 0.5);
\draw (0, 0) circle (0.3);
\draw (0, 0) node[circle,fill,inner sep=1pt]{};
\draw [dashed] (-0.5, 0)--(0.5, 0);
\draw (0.6, -0.6) node{$...$};
\draw (-0.6, -0.6) node{$...$};
\draw (-0.1, 0.4) node{$...$};
\draw (0,0) to [out=225,in=90] (-0.25,-0.4) to [out=270,in=315] (-0.1,-0.65);
\draw (0.5, -1) [->-] to (-0.1,-0.65);
\draw (-0.5, -1) [->-] to (-0.1, -0.75);
\draw (0,0) to [out=315,in=90] (0.25,-0.4) to [out=270,in=225] (0.1,-0.65);
\draw (-0.2, -0.8) node{$\ast$};
\draw (-0.5, -0.8) node{$\delta$};
\draw (-0.15, -0.15) node[circle,fill,inner sep=1.5pt]{};
\draw (0.3, -0.25) node[circle,fill,inner sep=1.5pt]{};
\draw (0, -0.6) node[circle,fill,inner sep=1.5pt]{};
\end{tikzpicture}
\end{figure}

For the invariance under move (IV), we again argue by comparing the states of two diagrams before and after the move.  The proof is tedious nonetheless straightforward, so we will only write down the idea but omit the details.  Take the following move for example, where the transverse edge passes over other edges. We assume that $\delta$ is on the transverse edge with color $i$ as indicated. Consider the crossings of the left-hand diagram $D$ lying on the transverse edge. There are essentially only two different assignments of them that we need to consider, the first of which assigns the crossing $C_1$ to its east corner, and the second to its north corner. Other assignments, if any, always appear in pairs and the contribution of each pair to $\langle D, c \rangle$ vanishes.

Consider the states where $C_{1}$ is assigned to its east corner. Then the states of the right-hand diagram $D'$ which have the same blank regions must assign $C_1$ to its east corner as well, as shown below. We can then check that their contributions to the state sum on both sides are the same.

\vspace{2mm}
\noindent
\begin{minipage}{\linewidth}
\makebox[\linewidth]
{\includegraphics[keepaspectratio=true,scale=0.4]{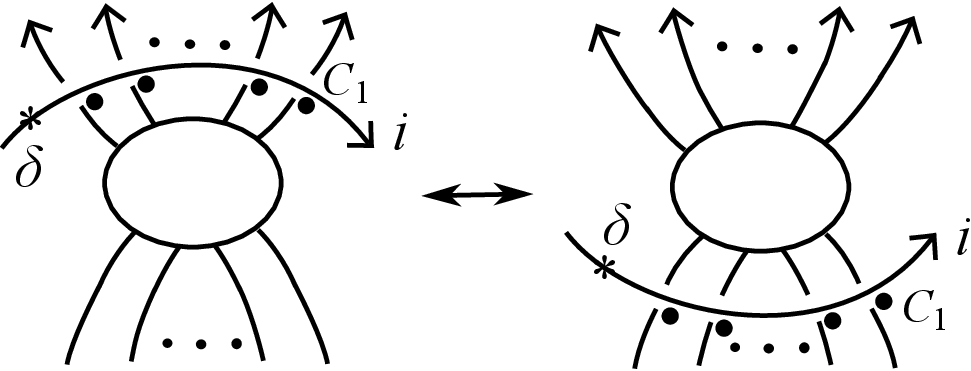}}
\end{minipage}
\vspace{2mm}

Consider the states of $D$ where $C_{1}$ is assigned to its north corner and $C_{2}$ is not assigned to its east corner. In this case, the states of $D'$ with the same blank regions are divided into two sets $A$ and $B$, where the each state in $A$ assigns $C_1$ to its north corner and each state in $B$ assigns $C_1$ to its west or south corner, as shown below. The states in $B$, no matter how many, always appear in pairs and the sum of their contributions to the state sum is zero.  Hence, we also have in this case that the contribution from the states of $D$ equals that from $A$.

\vspace{2mm}
\noindent
\begin{minipage}{\linewidth}
\makebox[\linewidth]
{\includegraphics[keepaspectratio=true,scale=0.4]{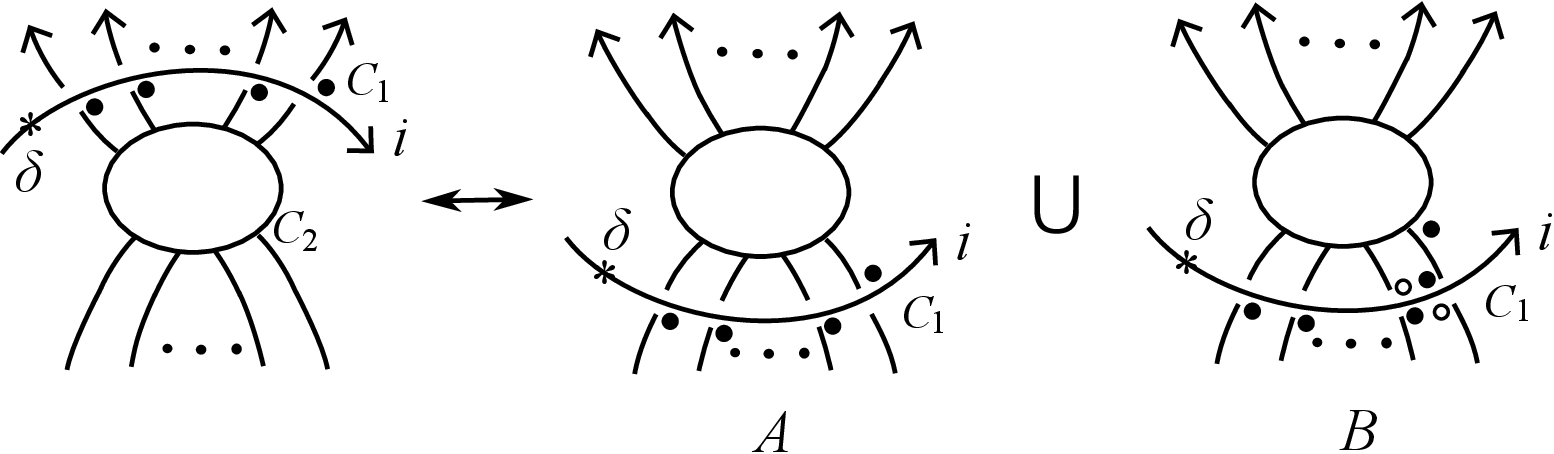}}
\end{minipage}
\vspace{2mm}

Consider the states of $D$ where $C_{1}$ is assigned to its north corner and $C_{2}$ is assigned to its east corner. For each of such states, there must be a blank region under the circle region. We shadow this region, as the example shown below. The states of $D'$ which provide the same blank regions are divided into two sets $A$ and $B$, where each state in $A$ sends the crossing in the northwest corner of the shadow region to its south corner, and each state in $B$ sends the crossing to its west or south corner. The states in $B$, no matter how many, can always be paired up and the sum of their contributions to the state sum is zero. This again shows that the contribution from the states of $D$ in this case equals that from $A$. Combining the discussion above, we complete the proof of the invariance of $\langle D , c \rangle$ under move (IV).

\vspace{2mm}
\noindent
\begin{minipage}{\linewidth}
\makebox[\linewidth]
{\includegraphics[keepaspectratio=true,scale=0.4]{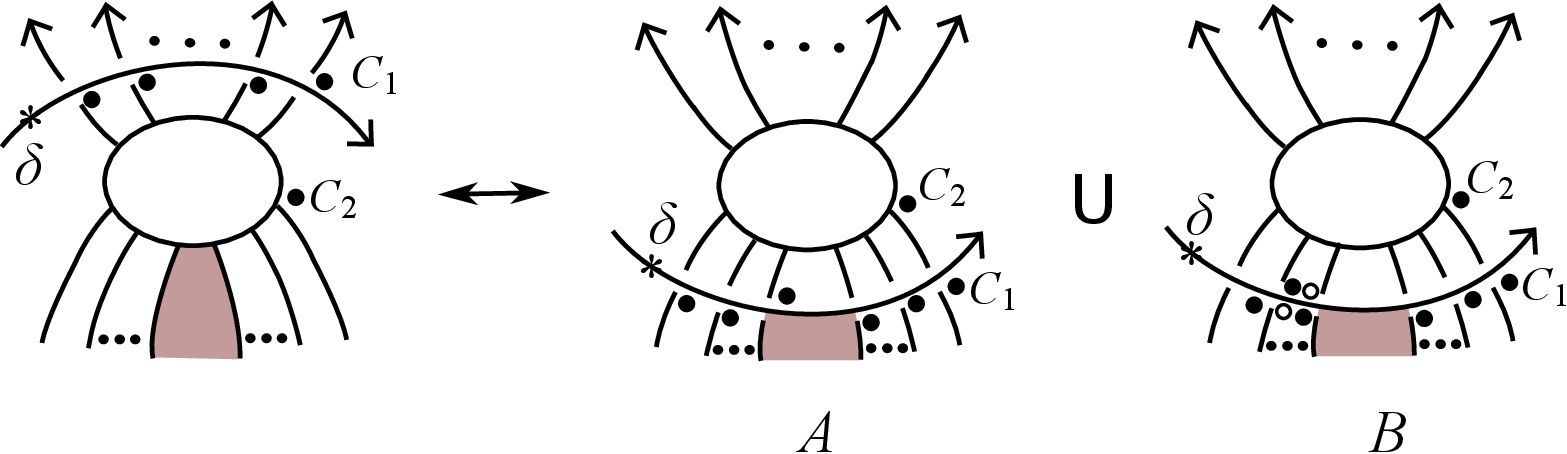}}
\end{minipage}
\vspace{2mm}

\end{proof}

\begin{defn}\label{Alex}
\rm
Let $(G, c)$ be an MOY graph. The {\it Alexander polynomial} of $(G,c)$, denoted $\Delta_{(G, c)}(t)$, can be defined from any of its graph diagram $(D,c)$ by
\begin{equation}\label{Def:Alex}
\Delta_{(G, c)}(t) := \frac{\langle D,c \rangle}{(t^{-1/2}-t^{1/2})^{\vert V \vert-1}} ,
\end{equation}
where $\vert V \vert$ is the number of vertices of $G$.
\end{defn}

The term $(t^{-1/2}-t^{1/2})^{\vert V \vert-1}$ in the denominator of Formula (\ref{Def:Alex}) is clearly an invariant of the graph. The advantage for introducing this factor in the definition of the Alexander polynomial will become transparent when we discuss the properties of $\Delta_{(D, c)}(t)$ for plane graphs (Theorem \ref{condition}).

Together with Proposition \ref{initial} and Proposition \ref{invtheo}, it implies that $\Delta_{(G, c)}(t)$ is well-defined up to a power of $t$, just like the classical Alexander polynomial of links. 

\vspace{3mm}

\subsection{The normalization of the Alexander polynomial}

In this part, we present a way of normalizing the Alexander polynomial $\Delta_{(G, c)}(t)$ that eliminates the $ t^k$-ambiguity in its definition when we fix a framing on $G$.
Whereas our idea is applicable to all framed MOY graphs, we will focus on the class of trivalent graphs, because it is the natural setting for the topic in Section 4 on MOY calculus.

Recall that a {\it framing} of $G$ is an embedded compact surface $F\subset S^3$ in which $G$ is sitting as a deformation retract.  A {\it framed graph} is a graph equipped with a framing.  More precisely, each vertex of $G$ is replaced by a disk in $F$ where the vertex is the center of the disk, and each edge of $G$ is replaced by a strip $[0, 1]\times [0, 1]$ where $[0, 1]\times \{0, 1\}$ is attached to the boundaries of its adjacent vertex disks and the edge is $\{\frac{1}{2}\}\times [0, 1]$. Each graph diagram of $G$ in $\mathbb{R}^{2}$ has a {\it blackboard framing}, whose projection in $\mathbb{R}^{2}$ is the tubular neighborhood of the graph diagram in $\mathbb{R}^{2}$. The difference between a generic framing and a blackboard framing can be represented on the diagram by introducing the symbols \Ptwist and \Ntwist, which mean a positive half twist and a negative half twist respectively at the fragment. Therefore we can use a graph diagram with \Ptwist and \Ntwist to represent a framed graph, as indicated in Fig. \ref{fig:e27}.

\begin{figure}[h!]
	\centering
		\includegraphics[width=0.6\textwidth]{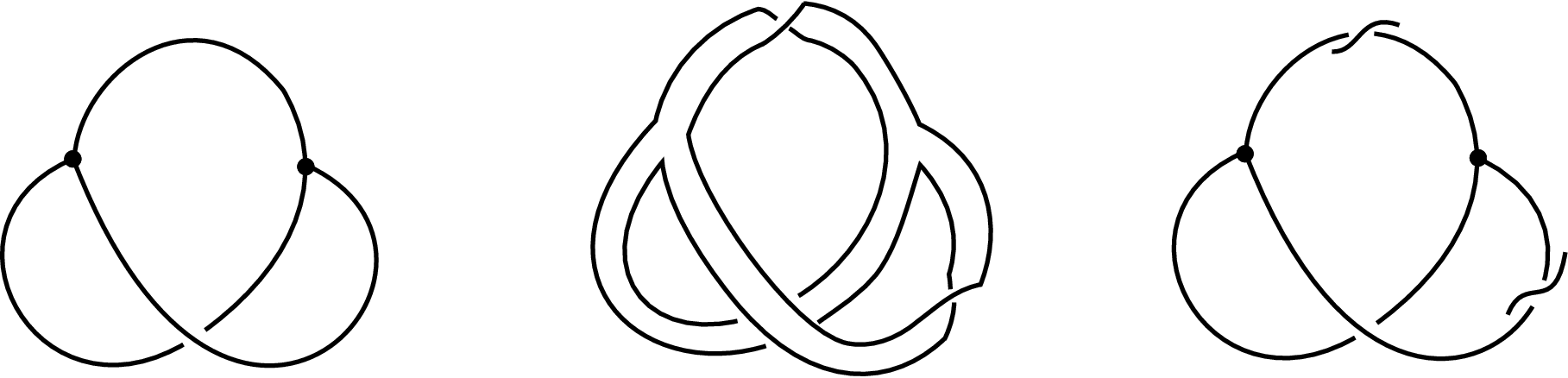}
	\caption{A framed graph and its graph diagram.}
	\label{fig:e27}
\end{figure}

Our idea of normalization is simple and natural.  Starting from a framed graph diagram, we hope to construct an appropriate factor that cancels the $t^{\pm i}$ factor in the Alexander polynomial coming from Reidemeister move (I) in Proposition \ref{invtheo} while keeping invariant under the other types of Reidemeister moves.

For framed trivalent graphs, the following result is well-known to experts (see for example \cite{MR2255851}).

\begin{lemma}
Any two graph diagrams for a framed trivalent graph can be connected by a sequence of Reidemeister moves in Fig. \ref{fig:e14}.
\end{lemma}

\begin{figure}[h!]
	\centering
		\includegraphics[width=0.6\textwidth]{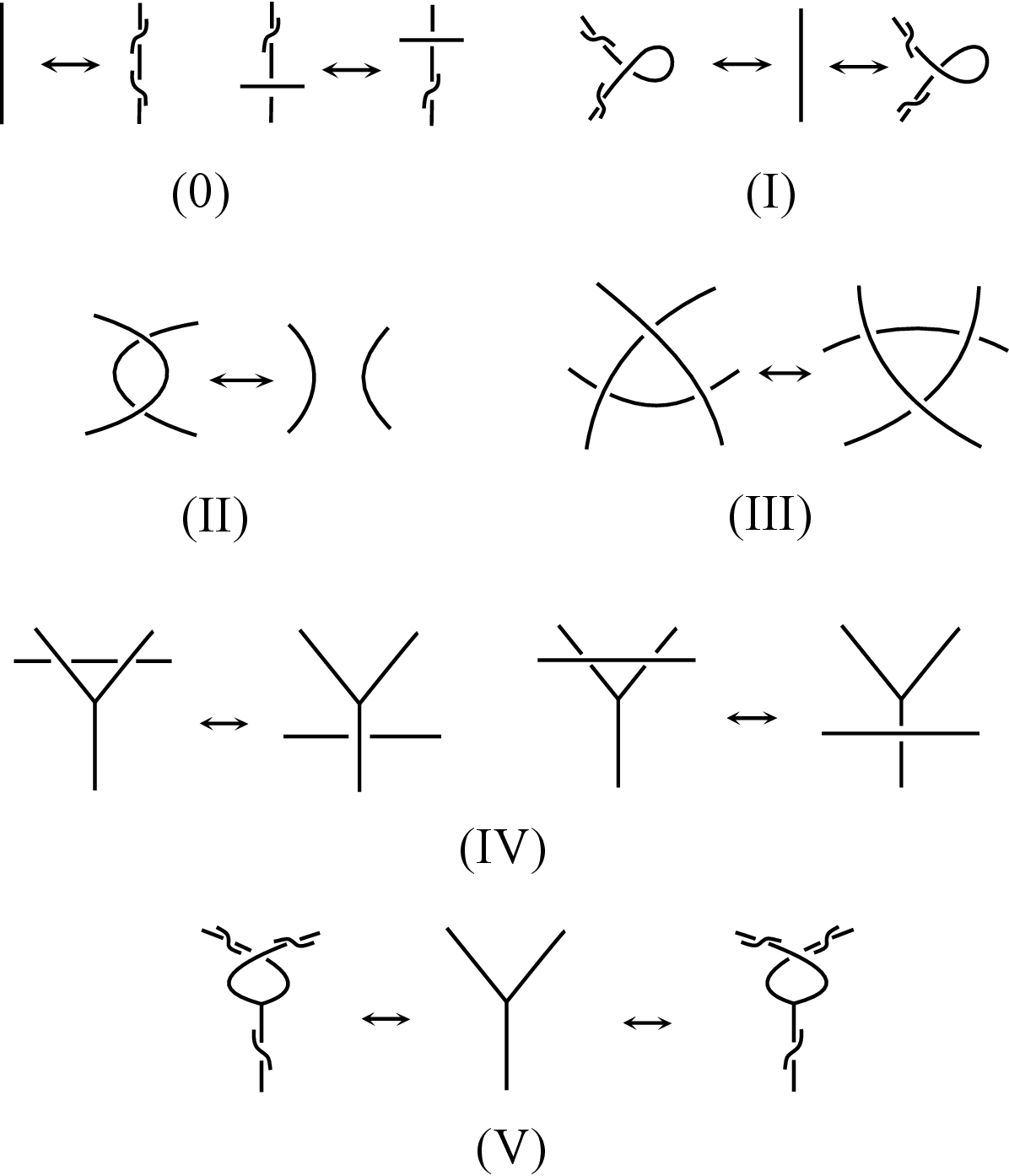}
	\caption{The Reidemeister moves for framed trivalent graphs.}
	\label{fig:e14}
\end{figure}

From now on, we use the blackboard-bold letters $\mathbb{G}$ and $\mathbb{D}$ to denote a framed trivalent graph and diagram, respectively.  We first define a framing factor.

\begin{defn}
\rm
For each edge $e$ with color $i$, let
\begin{equation*}\mathcal{F}(e)=t^{i/4[(\sharp \Ntwist) \, -\, (\sharp \Ptwist)]},
\end{equation*}
where $\sharp \Ntwist$ and $\sharp \Ptwist$ denote the number of \Ntwist and that of \Ptwist on $e$, respectively. The {\it framing factor} associated to $(\mathbb{D},c)$ is define as
\begin{equation*}\label{def:framing}
\mathcal{F}(\mathbb{D},c)=\prod_{e\in E} \mathcal{F}(e),
\end{equation*}
where $E$ is the edge set of $\mathbb{D}$.
\end{defn}

From $D$ we can obtain a diagram of an oriented link $L_{(D, c)}$ by the local transformation depicted in Fig. \ref{fig:e24}, where we replace each edge of color $i$ by $i$ parallel strands.

\begin{figure}
\begin{tikzpicture}[baseline=-0.65ex, thick, scale=0.8]
\draw (0,-1)  [->]-- (0,1);
\draw (0.3, 0.8) node {$4$} ;
\draw (1, 0) node {$\Longrightarrow$} ;
\draw (2,-1)  [->]-- (2,1);
\draw (2.5,-1)  [->]-- (2.5,1);
\draw (3,-1)  [->]-- (3,1);
\draw (3.5,-1)  [->]-- (3.5,1);
\end{tikzpicture}\quad \quad\quad
\begin{tikzpicture}[baseline=-0.65ex, thick, scale=0.8]
\draw (0,-1)  [->-]-- (0,0);
\draw (0,0)  [->]-- (1,1);
\draw (0,0)  [->]-- (-1,1);
\draw (1.3, 0.8) node {$3$} ;
\draw (-1.3, 0.8) node {$1$} ;
\draw (0.3, -0.8) node {$4$} ;
\end{tikzpicture}$\Longrightarrow$
\begin{tikzpicture}[baseline=-0.65ex, thick, scale=0.8]
\draw (1.2,-1) to [out=90, in=270] (1.2,0) [->] to [out=90,in=225] (2,1);
\draw (0.7,-1) to [out=90, in=270] (0.7,0) [->] to [out=90,in=225] (1.5,1);
\draw (-0.3,-1) to [out=90, in=270] (-0.3,0) [->] to [out=90,in=315] (-1.1,1);
\draw (0.2,-1) to [out=90, in=270] (0.2,0) [->] to [out=90,in=225] (1,1);
\end{tikzpicture}
	\caption{The way we obtain the link $L_{(D, c)}$ from a trivalent diagram $D$.}
	\label{fig:e24}
\end{figure}

\begin{defn}
\rm
We define the {\it curliness} of $D$ with respect to $c$ by
\begin{align*}\label{def:colredcurliness}
\mathcal{C}(D, c)=t^{1/2[(\sharp \,\,
\begin{tikzpicture}[baseline=-0.65ex,scale=0.2]
\draw[
        decoration={markings, mark=at position 0.315 with {\arrow{>}}},
        postaction={decorate}
        ]
(0,0) circle (1);
\end{tikzpicture}) \, - \,
(\sharp \,\,
\begin{tikzpicture}[baseline=-0.65ex,scale=0.2]
\draw[
        decoration={markings, mark=at position 0.315 with {\arrow{<}}},
        postaction={decorate}
        ]
(0,0) circle (1);
\end{tikzpicture})]
},
\end{align*}
where $\sharp \,\,
\begin{tikzpicture}[baseline=-0.65ex,scale=0.2]
\draw[
        decoration={markings, mark=at position 0.315 with {\arrow{>}}},
        postaction={decorate}
        ]
(0,0) circle (1);
\end{tikzpicture}$ and $\sharp \,\,
\begin{tikzpicture}[baseline=-0.65ex,scale=0.2]
\draw[
        decoration={markings, mark=at position 0.315 with {\arrow{<}}},
        postaction={decorate}
        ]
(0,0) circle (1);
\end{tikzpicture}$ denote respectively the numbers of counter-clockwise and clockwise oriented curves that we obtain after resolving the crossings of $L_{(D, c)}$ in the oriented way.
\end{defn}

As the notation indicates, $\mathcal{C}(D, c)$ does not depend on framing. Roughly speaking, $\mathcal{C}(D, c)$ measures the total ``winding number'' of the diagram $D$ with the weight $c$.  Our next lemma is evident from this viewpoint.

\begin{lemma}
\label{curli}
The curliness $\mathcal{C}(D, c)$ satisfies the following local relation:
\begin{equation*}
\mathcal{C}\left( \begin{tikzpicture}[baseline=-0.65ex, thick, scale=0.5]
\draw (1,-1) to [out=180, in=270] (0.5,0) [->-] to [out=90,in=180] (1,1);
\draw (-1,-1) to [out=0, in=270] (-0.5,0) [-<-] to [out=90,in=0] (-1,1);
\draw (1.2, 0) node {$i$};
\draw (-1.2, 0) node {$i$};
\end{tikzpicture} \right )
=t^{-i/2}\mathcal{C}\left( \begin{tikzpicture}[baseline=-0.65ex, thick, scale=0.5]
\draw (-1,-1) to [out=90, in=180] (0, -0.5) [-<-] to [out=0,in=90] (1,-1);
\draw (-1,1) to [out=270, in=180] (0, 0.5) [->-] to [out=0,in=270] (1,1);
\draw (1.2, 0.3) node {$i$};
\draw (-1.2, -0.3) node {$i$};
\end{tikzpicture} \right ).
\end{equation*}
The curliness $\mathcal{C}(D, c)$ is invariant under moves (II) and (III). Its changes under move (I) are given as follows.
\begin{equation*}\label{r1}
t^{-i/2}\mathcal{C} \left( \begin{tikzpicture}[baseline, thick, scale=0.35]
\draw (1,-1)   -- (0.2,-0.2);
\draw (-1, -1) [->] -- (1, 1) node[above]{$i$};
\draw (-0.2,0.2) --  (-1,1) ;
\draw (-1, 1) arc (90:270:1);
\end{tikzpicture} \right )
=t^{-i/2}\mathcal{C}\left( \begin{tikzpicture}[baseline=-0.65ex, thick, scale=0.35]
\draw (-1,-1)   -- (-0.2,-0.2);
\draw (1, -1) -- (-1, 1) ;
\draw (0.2,0.2) [->] --  (1,1) node[above]{$i$};
\draw (-1, 1) arc (90:270:1);
\end{tikzpicture} \right )
=t^{i/2}\mathcal{C}\left( \begin{tikzpicture}[baseline=-0.65ex, thick, scale=0.35]
\draw (1,-1)   -- (0.2,-0.2);
\draw (-1, -1) -- (1, 1) ;
\draw (-0.2,0.2)  [->]--  (-1,1)node[above]{$i$} ;
\draw (1,-1) arc (-90:90:1);
\end{tikzpicture} \right )
=t^{i/2} \mathcal{C}\left( \begin{tikzpicture}[baseline=-0.65ex, thick, scale=0.35]
\draw (-1,-1)   -- (-0.2,-0.2);
\draw (1, -1) [->] -- (-1, 1) node[above]{$i$};
\draw (0.2,0.2)  --  (1,1);
\draw (1,-1) arc (-90:90:1);
\end{tikzpicture} \right )
=\mathcal{C}\left( \begin{tikzpicture}[baseline=-0.65ex, thick, scale=0.35]
\draw (0, -1) [->] -- (0, 1) node[above]{$i$};
\end{tikzpicture} \right ).
\end{equation*}
\end{lemma}
\begin{proof}
One can check the relations by direct computations.
\end{proof}

\medskip

\begin{defn}\label{def:normalizedAlexanderpolynomial}
\rm
For a framed trivalent MOY graph diagram $(\mathbb{D},c)$, define the {\it normalized Alexander polynomial} by
\begin{equation}
\Delta_{(\mathbb{D}, c)}(t):=\frac{ \langle D,c\rangle}{(t^{-1/2}-t^{1/2})^{\vert V \vert-1}}\cdot \mathcal{F}(\mathbb{D},c)\mathcal{C}(D, c).
\end{equation}
\end{defn}

\medskip

Now we prove Theorem \ref{main1}.
\begin{theo}[Theorem 1.1]
\label{framed}
The function $\Delta_{(\mathbb{D}, c)}(t)$ is invariant under the Reidemeister moves in Fig. \ref{fig:e14}. Therefore, it is a topological invariant of the corresponding framed trivalent graph $(\mathbb{G},c)$.
\end{theo}

\begin{proof}
It is straightforward to check that $\Delta_{(\mathbb{D}, c)}(t)$ is invariant under move (0).

For move (I), we prove the move in the following figure for illustration, where $\mathbb{D}$ denotes the left-hand side diagram and $\mathbb{D}'$ denotes the right-hand side diagram.  One can check that $\mathcal{F}(\mathbb{D},c)=t^{-i/2}\mathcal{F}(\mathbb{D}',c)$ and $\mathcal{C}(D, c)=t^{-i/2}\mathcal{C}(D', c)$, while Proposition \ref{invtheo} implies that the state sums satisfy $\langle D, c\rangle=t^{i}\langle D', c\rangle$. Therefore $\Delta_{(\mathbb{D}, c)}(t)=\Delta_{(\mathbb{D}', c)}(t)$. The proof for the other cases of move (I) are similar. However, note that for the other cases of move (I), $\mathcal{F}(\mathbb{D},c)$ and $\mathcal{C}(D, c)$ do not necessarily change in the same way.

\vspace{2mm}
\noindent
\begin{minipage}{\linewidth}
\makebox[\linewidth]
{\includegraphics[keepaspectratio=true,scale=0.4]{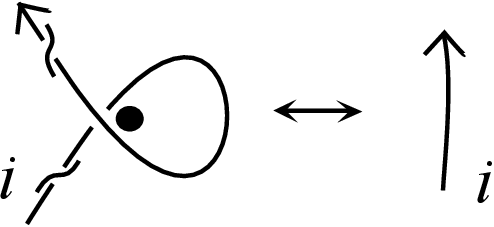}}
\end{minipage}
\vspace{2mm}

From Lemma \ref{curli}, we see that $\mathcal{F}(\mathbb{D},c)$ and $\mathcal{C}(D, c)$ are invariant under moves (II) (III) and (IV). Proposition \ref{invtheo} implies the same fact for $\langle D, c\rangle$. Therefore, $\Delta_{(\mathbb{D}, c)}(t)$ is also invariant under moves (II) (III) and (IV).

For move (V), we divide the discussion into two cases depending on the orientation of the edges. The first case is when the two edges where the twist occurs point both inward or outward. This belongs to move (V) in Fig. \ref{fig:e25}. Proposition \ref{invtheo} implies that $\langle D,c\rangle$ is invariant, and we can check from the definition that $\mathcal{F}(\mathbb{D},c)$ and $\mathcal{C}(D, c)$ are also invariant. As a result, $\Delta_{(\mathbb{D}, c)}(t)$ is invariant.

The second case is when one edge around the twist points inward and the other edge points outward. As this move can be realized as a composition of the moves discussed above, the invariance of $\Delta_{(\mathbb{D}, c)}$ follows. See the following figure.

\vspace{2mm}
\noindent
\begin{minipage}{\linewidth}
\makebox[\linewidth]
{\includegraphics[keepaspectratio=true,scale=0.8]{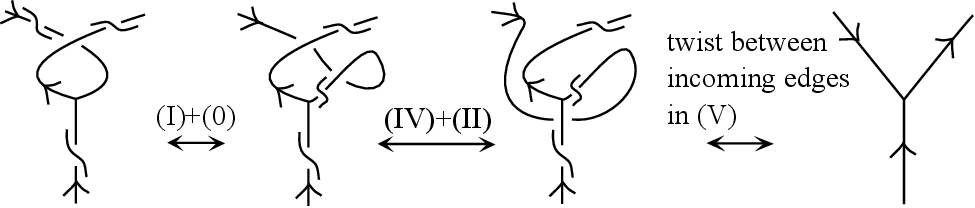}}
\end{minipage}
\vspace{2mm}

\end{proof}

Thus, we define the normalized Alexander polynomial of framed trivalent graph $(\mathbb{G},c)$ by
\begin{equation}\label{framedAlexander}
\Delta_{(\mathbb{G},c)}(t):=\Delta_{(\mathbb{D},c)}(t)=\frac{ \langle D,c\rangle}{(t^{-1/2}-t^{1/2})^{\vert V \vert-1}}\cdot \mathcal{F}(\mathbb{D},c)\mathcal{C}(D, c),
\end{equation}
where $\mathbb{D}$ is an arbitrary diagram that represents $\mathbb{G}$.

\begin{rem}
\rm
In summary, for an MOY graph $G$ without framing, we have only an Alexander polynomial $\Delta_{(G,c)}(t)$ that is well-defined up to $t^k$. When we consider a framed trivalent graph $\mathbb{G}$, we obtain a normalized Alexander polynomial $\Delta_{(\mathbb{G},c)}(t)$ without $t^k$-ambiguity.
We will use both $\Delta_{(G,c)}(t)$ and $\Delta_{(\mathbb{G},c)}(t)$
in subsequent sections for different purposes.
\end{rem}


\section{MOY-type relations}
\label{moysection}

In this section, we prove Theorem \ref{main2}. This result originates from Murakami-Ohtsuki-Yamada's graphic calculus \cite{MR1659228} for $U_q(\mathfrak{sl}_n)$-polynomial invariants. Here we give a version of the theory for $\Delta_{(\mathbb{G}, c)}(t)$.

\subsection{MOY-type relations.}
\begin{theo}
\label{moytheo}
Suppose $\mathbb{D}$ is a framed trivalent graph diagram with a positive coloring $c$.
The normalized Alexander polynomial $\Delta_{(\mathbb{D}, c)}(t)$ satisfies the following relations, where $(\mathbb{D})$ represents $\Delta_{(\mathbb{D}, c)}(t)$,and $[k]=\displaystyle \frac{t^{k/2}-t^{-k/2}}{t^{1/2}-t^{-1/2}}$ for $k\in \mathbb{Z}$.
\begin{align*}\label{moy}
 (i) \quad
&\left(\begin{tikzpicture}[baseline=-0.65ex, thick, scale=0.6]
\draw[
        decoration={markings, mark=at position 0.625 with {\arrow{<}}},
        postaction={decorate}
        ]
(0,0) circle (1);
\draw (1.3,0) node {$i$};
\end{tikzpicture}\right)
=\left(\begin{tikzpicture}[baseline=-0.65ex, thick, scale=0.6]
\draw[
        decoration={markings, mark=at position 0.625 with {\arrow{>}}},
        postaction={decorate}
       ]
(0,0) circle (1);
\draw (1.3,0) node {$i$};
\end{tikzpicture}\right)
=\frac{1}{[i]}, \text{ for $i>0$}.\\
(ii) \quad
 & \text{$(\mathbb{D})=0$ if $\mathbb{D}$ is a disconnected graph.}\\
(iii) \quad
&\left(\begin{tikzpicture}[baseline=-0.65ex, thick, scale=0.6]
\draw (0, -1) -- (0, 0);
\draw (0, 0.3)  -- (0, 1) node[above]{$i$};
\draw (0.2, 0.5) -- (0.2, 0.2) -- (-0.2, 0.1) -- (-0.2, -0.2);
\end{tikzpicture}\right)
= t^{-i/4} \cdot \left(\begin{tikzpicture}[baseline=-0.65ex, thick, scale=0.5]
\draw (0, -1)  -- (0, 1) node[above]{$i$};
\end{tikzpicture}\right), \quad
\left(\begin{tikzpicture}[baseline=-0.65ex, thick, scale=0.6]
\draw (0, -1) -- (0, 0);
\draw (0, 0.3)  -- (0, 1) node[above]{$i$};
\draw (-0.2, 0.5) -- (-0.2, 0.2) -- (0.2, 0.1) -- (0.2, -0.2);
\end{tikzpicture}\right)
= t^{i/4} \cdot \left(\begin{tikzpicture}[baseline=-0.65ex, thick, scale=0.5]
\draw (0, -1)  -- (0, 1) node[above]{$i$};
\end{tikzpicture}\right).\\
(iv) \quad &\text{When $i\leq j$:}\\
\quad
&\left(\begin{tikzpicture}[baseline=-0.65ex, thick, scale=0.6]
\draw (1,-1)   -- (0.2,-0.2);
\draw (-1, -1) [->] -- (1, 1) node[above]{$i$};
\draw (-0.2,0.2) [->] to (-1,1)  node[above]{$j$};
\end{tikzpicture}\right)
  =\frac{-t^{-\frac{i+j}{2}}}{[i]\,[j]}\cdot
\left(\begin{tikzpicture}[baseline=-0.65ex, thick, scale=1.2]
\draw (0,-1) [->-] to (0, 0.33);
\draw (0, 0.33) [->] to (0,1);
\draw (1,-1) [->-] to (1, -0.33);
\draw (1, -0.33) [->] to (1,1);
\draw (0,0.33) [-<-]  to [out=270,in=180] (0.5,0) to [out=0,in=90] (1,-0.33);
\draw (0,1.25) node {$j$};
\draw (0,-1.25) node {$i$};
\draw (1,1.25) node {$i$};
\draw (0.9,-1.25) node {$j$};
\draw (0.5,0.3) node {$j-i$};
\draw (1, -0.33) node[circle,fill,inner sep=1pt]{};
\draw (0, 0.33) node[circle,fill,inner sep=1pt]{};
\end{tikzpicture}\right)
+ \, \frac{t^{-\frac{j}{2}}}{[i]\,[i+j]}\cdot
\left(
\begin{tikzpicture}[baseline=-0.65ex, thick, scale=1.2]
\draw (0,-1) [->-] to [out=90,in=270] (0.5,-0.33);
\draw (0.5, -0.33) [->-] to [out=90,in=270] (0.5,0.33);
\draw (0.5, 0.33) [->] to [out=90,in=270] (0,1);
\draw (1,-1) [->-] to [out=90,in=270] (0.5,-0.33);
\draw (1,1) [<-] to [out=270,in=90] (0.5,0.33);
\draw (0, -1.25) node {$i$};
\draw (0.9,-1.25) node {$j$};
\draw (0,1.25) node {$j$};
\draw (1,1.25) node {$i$};
\draw (1,0) node {$i+j$};
\draw (0.5, -0.33) node[circle,fill,inner sep=1pt]{};
\draw (0.5, 0.33) node[circle,fill,inner sep=1pt]{};
\end{tikzpicture}\right),\\
&\left(\begin{tikzpicture}[baseline=-0.65ex, thick, scale=0.6]
\draw (1,-1) [->]  -- (-1,1)  node[above]{$j$};
\draw (-1, -1) to (-0.2, -0.2);
\draw (0.2, 0.2) [->] -- (1, 1) node[above]{$i$};
\end{tikzpicture}\right)
  =\frac{-t^{\frac{i+j}{2}}}{[i]\,[j]}\cdot
\left(\begin{tikzpicture}[baseline=-0.65ex, thick, scale=1.2]
\draw (0,-1) [->-] to (0, 0.33);
\draw (0, 0.33) [->] to (0,1);
\draw (1,-1) [->-] to (1, -0.33);
\draw (1, -0.33) [->] to (1,1);
\draw (0,0.33) [-<-]  to [out=270,in=180] (0.5,0) to [out=0,in=90] (1,-0.33);
\draw (0,1.25) node {$j$};
\draw (0,-1.25) node {$i$};
\draw (1,1.25) node {$i$};
\draw (0.9,-1.25) node {$j$};
\draw (0.5,0.3) node {$j-i$};
\draw (1, -0.33) node[circle,fill,inner sep=1pt]{};
\draw (0, 0.33) node[circle,fill,inner sep=1pt]{};
\end{tikzpicture}\right)
+ \, \frac{t^{\frac{j}{2}}}{[i]\,[i+j]}\cdot
\left(
\begin{tikzpicture}[baseline=-0.65ex, thick, scale=1.2]
\draw (0,-1) [->-] to [out=90,in=270] (0.5,-0.33);
\draw (0.5, -0.33) [->-] to [out=90,in=270] (0.5,0.33);
\draw (0.5, 0.33) [->] to [out=90,in=270] (0,1);
\draw (1,-1) [->-] to [out=90,in=270] (0.5,-0.33);
\draw (1,1) [<-] to [out=270,in=90] (0.5,0.33);
\draw (0, -1.25) node {$i$};
\draw (0.9,-1.25) node {$j$};
\draw (0,1.25) node {$j$};
\draw (1,1.25) node {$i$};
\draw (1,0) node {$i+j$};
\draw (0.5, -0.33) node[circle,fill,inner sep=1pt]{};
\draw (0.5, 0.33) node[circle,fill,inner sep=1pt]{};
\end{tikzpicture}\right).\\
\quad &\text{When $j\leq i$:}\\
\quad
&\left(\begin{tikzpicture}[baseline=-0.65ex, thick, scale=0.6]
\draw (1,-1)   -- (0.2,-0.2);
\draw (-1, -1) [->] -- (1, 1) node[above]{$i$};
\draw (-0.2,0.2) [->] to (-1,1)  node[above]{$j$};
\end{tikzpicture}\right)
  =\frac{-t^{-\frac{i+j}{2}}}{[i]\,[j]}\cdot
\left(\begin{tikzpicture}[baseline=-0.65ex, thick, scale=1.2]
\draw (0,-1) [->-] to (0, -0.33);
\draw (0, -0.33) [->] to (0,1);
\draw (1,-1) [->-] to (1, 0.33);
\draw (1, 0.33) [->] to (1,1);
\draw (0,-0.33) [->-]  to [out=90,in=180] (0.5,0) to [out=0,in=270] (1,0.33);
\draw (0,1.25) node {$j$};
\draw (0,-1.25) node {$i$};
\draw (1,1.25) node {$i$};
\draw (0.9,-1.25) node {$j$};
\draw (0.5,0.3) node {$i-j$};
\draw (1, 0.33) node[circle,fill,inner sep=1pt]{};
\draw (0, -0.33) node[circle,fill,inner sep=1pt]{};
\end{tikzpicture}\right)
+ \, \frac{t^{-\frac{i}{2}}}{[j]\,[i+j]}\cdot
\left(
\begin{tikzpicture}[baseline=-0.65ex, thick, scale=1.2]
\draw (0,-1) [->-] to [out=90,in=270] (0.5,-0.33);
\draw (0.5, -0.33) [->-] to [out=90,in=270] (0.5,0.33);
\draw (0.5, 0.33) [->] to [out=90,in=270] (0,1);
\draw (1,-1) [->-] to [out=90,in=270] (0.5,-0.33);
\draw (1,1) [<-] to [out=270,in=90] (0.5,0.33);
\draw (0, -1.25) node {$i$};
\draw (0.9,-1.25) node {$j$};
\draw (0,1.25) node {$j$};
\draw (1,1.25) node {$i$};
\draw (1,0) node {$i+j$};
\draw (0.5, -0.33) node[circle,fill,inner sep=1pt]{};
\draw (0.5, 0.33) node[circle,fill,inner sep=1pt]{};
\end{tikzpicture}\right),\\
&\left(\begin{tikzpicture}[baseline=-0.65ex, thick, scale=0.6]
\draw (1,-1) [->]  -- (-1,1)  node[above]{$j$};
\draw (-1, -1) to (-0.2, -0.2);
\draw (0.2, 0.2) [->] -- (1, 1) node[above]{$i$};
\end{tikzpicture}\right)
  =\frac{-t^{\frac{i+j}{2}}}{[i]\,[j]}\cdot
\left(\begin{tikzpicture}[baseline=-0.65ex, thick, scale=1.2]
\draw (0,-1) [->-] to (0, -0.33);
\draw (0, -0.33) [->] to (0,1);
\draw (1,-1) [->-] to (1, 0.33);
\draw (1, 0.33) [->] to (1,1);
\draw (0,-0.33) [->-]  to [out=90,in=180] (0.5,0) to [out=0,in=270] (1,0.33);
\draw (0,1.25) node {$j$};
\draw (0,-1.25) node {$i$};
\draw (1,1.25) node {$i$};
\draw (0.9,-1.25) node {$j$};
\draw (0.5,0.3) node {$i-j$};
\draw (1, 0.33) node[circle,fill,inner sep=1pt]{};
\draw (0, -0.33) node[circle,fill,inner sep=1pt]{};
\end{tikzpicture}\right)
+ \, \frac{t^{\frac{i}{2}}}{[j]\,[i+j]}\cdot
\left(
\begin{tikzpicture}[baseline=-0.65ex, thick, scale=1.2]
\draw (0,-1) [->-] to [out=90,in=270] (0.5,-0.33);
\draw (0.5, -0.33) [->-] to [out=90,in=270] (0.5,0.33);
\draw (0.5, 0.33) [->] to [out=90,in=270] (0,1);
\draw (1,-1) [->-] to [out=90,in=270] (0.5,-0.33);
\draw (1,1) [<-] to [out=270,in=90] (0.5,0.33);
\draw (0, -1.25) node {$i$};
\draw (0.9,-1.25) node {$j$};
\draw (0,1.25) node {$j$};
\draw (1,1.25) node {$i$};
\draw (1,0) node {$i+j$};
\draw (0.5, -0.33) node[circle,fill,inner sep=1pt]{};
\draw (0.5, 0.33) node[circle,fill,inner sep=1pt]{};
\end{tikzpicture}\right).\\
(v)
\quad
& \left(\begin{tikzpicture}[baseline=-0.65ex,thick, scale=0.7]
\draw (0,-1.5) [->-] to (0,-0.5);
\draw (0,-0.5) [->-] to (0,0.5);
\draw (0,-0.5) [->] to (0,1.5);
\draw (0, -0.5) node[circle,fill,inner sep=1pt]{};
\draw (0, 0.5) node[circle,fill,inner sep=1pt]{};
\draw (0,-0.5) to [out=270, in=270] (1,-0.2) [-<-] to [out=90, in=270] (1,0.2) to [out=90, in=90] (0,0.5);
\draw (-0.3, -1) node {$j$};
\draw (-0.3, 1) node {$j$};
\draw (-0.9, 0) node {$i+j$};
\draw (1.3,0) node {$i$};
\end{tikzpicture}\right) =
\left(\begin{tikzpicture}[baseline=-0.65ex,thick, scale=0.7]
\draw (0,-1.5) [->-] to (0,-0.5);
\draw (0,-0.5) [->-] to (0,0.5);
\draw (0,-0.5) [->] to (0,1.5);
\draw (0, -0.5) node[circle,fill,inner sep=1pt]{};
\draw (0, 0.5) node[circle,fill,inner sep=1pt]{};
\draw (0,-0.5) to [out=270, in=270] (-1,-0.2) [-<-] to [out=90, in=270] (-1,0.2) to [out=90, in=90] (0,0.5);
\draw (0.3, -1) node {$j$};
\draw (0.3, 1) node {$j$};
\draw (0.9, 0) node {$i+j$};
\draw (-1.3,0) node {$i$};
\end{tikzpicture}\right)
=[j]\,[i+j]\cdot
\left(\begin{tikzpicture}[baseline=-0.65ex,thick, scale=0.7]
\draw (0,-1.5) [->] to (0,1.5);
\draw (-0.25,0) node {$j$};
\draw (0.25,0) node {};
\end{tikzpicture}\right). \\
(vi) \quad
&\left(\begin{tikzpicture}[baseline=-0.65ex,thick,scale=0.6]
\draw [->-] (0.5,-1.5) node[below]{$i$}-- (0.5,-0.6);
\draw[->-] (0.5, -0.6) to [out=90, in=270] (0,0) to [out=90,in=270] (0.5,0.6) ;
\draw (0.5, -0.6) [->-] to [out=90, in=270] (1,0) to [out=90,in=270] (0.5,0.6);
\draw [->](0.5,0.6) -- (0.5,1.5) node[above]{$i$};
\draw (0.5,-0.6) node[circle,fill,inner sep=1pt]{};
\draw (0.5,0.6) node[circle,fill,inner sep=1pt]{};
\draw (-1, 0) node {$i-j$};
\draw (1.5, 0) node {$j$};
\end{tikzpicture}\right) = [i]^2\cdot
\left(\begin{tikzpicture}[baseline=-0.65ex,thick,scale=0.6]
\draw [->](0,-1.5) -- (0, 1.5);
\draw (-0.5 , 0) node {$i$};
\draw (0.25, 0) node {};
\end{tikzpicture}\right), \text{ where $i\geq j$.}\\
(vii)
\quad
& \left(\begin{tikzpicture}[baseline=-0.65ex, thick, scale=1.2]
\draw (0,-1) [->-] to [out=90,in=180] (0.5,-0.5);
\draw (0.5, -0.5) [->-] to [out=0,in=180] (1,-0.5) ;
\draw (1, -0.5) [->] to [out=0, in=90] (1.5,-1);
\draw (0.5,-0.5) [-<-] to [out=180, in=270] (0,0) to [out=90,in=180] (0.5,0.5);
\draw (1,-0.5) [->-] to [out=0, in=270] (1.5,0) to [out=90,in=0] (1,0.5);
\draw (0,1) [<-] to [out=270,in=180] (0.5,0.5);
\draw (0.5, 0.5) [-<-] to [out=0,in=180] (1,0.5);
\draw (1, 0.5) [-<-]  to [out=0, in=270] (1.5,1);
\draw (-0.25, -0.75) node {$j$};
\draw (1.75, -0.75) node {$i$};
\draw (-0.25, 0) node {$i$};
\draw (-0.25, 0.75) node {$j$};
\draw (1.75, 0.75) node {$i$};
\draw (1.75, 0) node {$j$};
\draw (0.75,-0.75) node {$i+j$};
\draw (0.75,0.75) node {$i+j$};
\draw (0.5, -0.5) node[circle,fill,inner sep=1pt]{};
\draw (1, -0.5) node[circle,fill,inner sep=1pt]{};
\draw (0.5, 0.5) node[circle,fill,inner sep=1pt]{};
\draw (1, 0.5) node[circle,fill,inner sep=1pt]{};
\end{tikzpicture}\right)
=\frac{[i+j]^3}{[i]}\cdot
\left(\begin{tikzpicture}[baseline=-0.65ex, thick, scale=1.2]
\draw (0,-1) [->-] to [out=90, in=180] (0.5,-0.5);
\draw (0.5, -0.5)  to [out=0,in=180] (1,-0.5) [->]  to [out=0,in=90] (1.5,-1);
\draw (0.5, 0.5) [out=180, in=270]   [->] to (0,1);
\draw (0.5,0.5) to [out=0,in=180] (1,0.5) [-<-] to [out=0,in=270] (1.5,1);
\draw (0.75,-0.5) to [out=180, in=270] (0.25,0) [-<-] to [out=90,in=180] (0.75,0.5);
\draw (-0.25, -0.75) node {$j$};
\draw (1.75, -0.75) node {$i$};
\draw (1, 0) node {$i-j$};
\draw (-0.25, 0.75) node {$j$};
\draw (1.75, 0.75) node {$i$};
\draw (0.75, -0.5) node[circle,fill,inner sep=1pt]{};
\draw (0.75, 0.5) node[circle,fill,inner sep=1pt]{};
\end{tikzpicture}\right)\\
\quad & \hspace{4cm}
+[j]^2\,[i+j]^2\cdot
\left(\begin{tikzpicture}[baseline=-0.65ex, thick, scale=1.2]
\draw (0,-1) [->] to (0,1);
\draw (1,-1) [<-] to (1,1);
\draw (-0.25,0) node {$j$};
\draw (1.25, 0) node {$i$};
\end{tikzpicture}\right), \text{ where $i\geq j$.} \\
(viii)
\quad
& \left(\begin{tikzpicture}[baseline=-0.65ex,thick, scale=0.9]
\draw (0,-1.5) [->-] to (0,-0.5);
\draw (0,-0.5) [->-] to [out=90,in=270] (-1,0.5);
\draw (-1,0.5) [->] to [out=90,in=270] (-2,1.5);
\draw (-1,0.5) [->] to [out=90, in=270] (0,1.5);
\draw (0,-0.5) [->] to [out=90,in=270] (1,0.5) to [out=90,in=270] (1,1.5);
\draw (1.5,-1) node {$i+j+k$};
\draw (1.3, 1.25) node {$k$};
\draw (-1.7,0) node {$i+j$};
\draw (-2.25, 1.25) node {$i$};
\draw (-0.5,1.25) node {$j$};
\draw (0, -0.5) node[circle,fill,inner sep=1pt]{};
\draw (-1, 0.5) node[circle,fill,inner sep=1pt]{};
\end{tikzpicture}\right)=\frac{ [i+j] }  {[j+k]}\cdot
\left(\begin{tikzpicture}[baseline=-0.65ex,thick, scale=0.9]
\draw (0,-1.5) [->-]-- (0,-0.5);
\draw (0,-0.5)  to [out=90,in=270] (-1,0.5) [->] to [out=90,in=270] (-1,1.5);
\draw (1,0.5) [->] to [out=90, in=270] (2,1.5);
\draw (0,-0.5) [->-] to [out=90,in=270] (1,0.5);
\draw (1, 0.5) [->] to [out=90,in=270] (0,1.5);
\draw (1.5,-1) node {$i+j+k$};
\draw (2.25, 1.3) node {$k$};
\draw (1.6,0) node {$j+k$};
\draw (-1.3, 1.25) node {$i$};
\draw (0.5,1.25) node {$j$};
\draw (0, -0.5) node[circle,fill,inner sep=1pt]{};
\draw (1, 0.5) node[circle,fill,inner sep=1pt]{};
\end{tikzpicture}\right), \\
\quad
& \left(\begin{tikzpicture}[baseline=-0.65ex,thick, scale=0.9]
\draw (1,0.5) [->]-- (1,1.5);
\draw (2,-1.5) to [out=90,in=270] (2,-0.5) [->-] to [out=90,in=270] (1,0.5);
\draw (0,-0.5) [->-] to [out=90, in=270] (1,0.5);
\draw (-1,-1.5) [->-] to [out=90,in=270] (0,-0.5);
\draw (1, -1.5) [->-] to [out=90,in=270] (0,-0.5);
\draw (2.5,1) node {$i+j+k$};
\draw (2.25, -1.3) node {$k$};
\draw (-0.6,0) node {$i+j$};
\draw (-1.3, -1.25) node {$i$};
\draw (1.3,-1.25) node {$j$};
\draw (1, 0.5) node[circle,fill,inner sep=1pt]{};
\draw (0, -0.5) node[circle,fill,inner sep=1pt]{};
\end{tikzpicture}\right)=\frac{ [i+j] }  {[j+k]}\cdot
\left(\begin{tikzpicture}[baseline=-0.65ex,thick, scale=0.9]
\draw (0,0.5) [->]-- (0,1.5);
\draw (-1,-1.5) to [out=90,in=270] (-1,-0.5) [->-] to [out=90,in=270] (0,0.5);
\draw (1,-0.5) [->-] to [out=90, in=270] (0,0.5);
\draw (0,-1.5) [->-] to [out=90,in=270] (1,-0.5);
\draw (2, -1.5) [->-] to [out=90,in=270] (1,-0.5);
\draw (1.5,1) node {$i+j+k$};
\draw (2.25, -1.3) node {$k$};
\draw (1.6,0) node {$j+k$};
\draw (-1.3, -1.25) node {$i$};
\draw (0.6,-1.25) node {$j$};
\draw (0, 0.5) node[circle,fill,inner sep=1pt]{};
\draw (1, -0.5) node[circle,fill,inner sep=1pt]{};
\end{tikzpicture}\right).\\
(ix)
\quad
& \left(\begin{tikzpicture}[baseline=-0.65ex, thick, scale=1.2]
\draw (0,-1) [->-] to (0, -0.33);
\draw (0, -0.33) [->-] to (0, 0.33);
\draw (0, 0.33) [->] to (0,1);
\draw (2,-1) [->-] to (2, -0.66);
\draw (2, -0.66) [->-] to (2, 0.66);
\draw (2, 0.66) [->] to (2,1);
\draw (0,-0.33) to [out=270,in=180] (0.5,-0.5) [-<-] to [out=0,in=180] (1.5,-0.5) to [out=0, in=90] (2,-0.66);
\draw (0,0.33) to [out=90, in=180] (0.5,0.5) [->-] to [out=0,in=180] (1.5,0.5) to [out=0, in=270] (2,0.66);
\draw (0,-1.25) node {$i$};
\draw (2,-1.25) node {$j$};
\draw (0.5,0) node {$i+k$};
\draw (0.1,1.25) node {$i+k-l$};
\draw (1, -0.8) node {$k$};
\draw (1.5,0) node {$j-k$};
\draw (1, 0.8) node {$l$};
\draw (2,1.25) node {$j+l-k$};
\draw (0, -0.33) node[circle,fill,inner sep=1pt]{};
\draw (0, 0.33) node[circle,fill,inner sep=1pt]{};
\draw (2, -0.66) node[circle,fill,inner sep=1pt]{};
\draw (2, 0.66) node[circle,fill,inner sep=1pt]{};
\end{tikzpicture}\right)
= \frac{ [j]\,[l]\,[i+k]}{ [i+j]}\cdot
\left(
\begin{tikzpicture}[baseline=-0.65ex, thick, scale=1.2]
\draw (0,-1) [->-] to [out=90,in=270] (0.5,-0.33);
\draw (0.5, -0.33) [->-] to [out=90,in=270] (0.5,0.33);
\draw (0.5, 0.33) [->] to [out=90,in=270] (0,1);
\draw (1,-1) [->-] to [out=90,in=270] (0.5,-0.33);
\draw (1,1) [<-] to [out=270,in=90] (0.5,0.33);
\draw (0, -1.25) node {$i$};
\draw (0.9,-1.25) node {$j$};
\draw (-0.5,1.25) node {$i+k-l$};
\draw (1.5,1.25) node {$j+l-k$};
\draw (1.2,0) node {$i+j$};
\draw (0.5, -0.33) node[circle,fill,inner sep=1pt]{};
\draw (0.5, 0.33) node[circle,fill,inner sep=1pt]{};
\end{tikzpicture}\right)\\
\quad & \hspace{1cm} + [i+k]\,[j-k]\cdot
\left(\begin{tikzpicture}[baseline=-0.65ex, thick, scale=1.2]
\draw (0,-1) [->-] to (0, 0.33);
\draw (0, 0.33) [->] to (0,1);
\draw (1,-1) [->-] to (1, -0.33);
\draw (1, -0.33) [->] to (1,1);
\draw (0,0.33) [-<-]  to [out=270,in=180] (0.5,0) to [out=0,in=90] (1,-0.33);
\draw (-0.5,1.25) node {$i+k-l$};
\draw (0,-1.25) node {$i$};
\draw (1.5,1.25) node {$j+l-k$};
\draw (1,-1.25) node {$j$};
\draw (0.5,0.3) node {$k-l$};
\draw (1, -0.33) node[circle,fill,inner sep=1pt]{};
\draw (0, 0.33) node[circle,fill,inner sep=1pt]{};
\end{tikzpicture}\right), \text{ where $j\geq k \geq l$.}\\
(x)
\quad
&\left(\begin{tikzpicture}[baseline=-0.65ex, thick, scale=1.2]
\draw (0,-1)  to (0, 0.33);
\draw (0, 0.33)  to (0,1);
\draw (1,-1)  to (1, -0.33);
\draw (1, -0.33)  to (1,1);
\draw (0,0.33)  to (1,-0.33);
\draw (0,1.25) node {$i$};
\draw (0,-1.25) node {$i$};
\draw (1,1.25) node {$j$};
\draw (1,-1.25) node {$j$};
\draw (0.5,0.3) node {$0$};
\draw (1, -0.33) node[circle,fill,inner sep=1pt]{};
\draw (0, 0.33) node[circle,fill,inner sep=1pt]{};
\end{tikzpicture}\right)
=[i]\,[j] \cdot
\left(\begin{tikzpicture}[baseline=-0.65ex, thick, scale=1.2]
\draw (0,-1)  to (0,0.5);
\draw (1,-1)  to (1,0.5);
\draw (0,0.75) node {$i$};
\draw (1,0.75) node {$j$};
\end{tikzpicture}\right).\\
\end{align*}
\end{theo}

\begin{proof}
Like other literature on MOY calculus, it is customary to treat a loop as a special vertex-less trivalent MOY graph. We can then evaluate the Kauffman state sum formula in relation (i) as follows.  Note that both the set of crossings and the set of unmarked regions (after we place a base point on the circle) are empty. Hence, the set of Kauffman states $S(D)$ is a singleton set whose unique element is the unique map
$s: \, \emptyset=\operatorname{Cr}(D)\rightarrow \operatorname{Re}(D)\backslash \{R_u, R_{v}\}=\emptyset.$  Then, according to Definition \ref{def:statesum}, $M(s)$ and $A(s)$ are both the trivial product, i.e., 1.  With $|V|=0$, $\mathcal{F}(\mathbb{D},c)=1$, $\mathcal{C}(D,c)=t^{-i/2}$ and $\vert \delta \vert=t^{-i}-1$ for the clockwise diagram and similar computation for the counter-clockwise diagram, Definition \ref{def:normalizedAlexanderpolynomial} gives the desired relation (i).

Relation (ii) follows from the definition in Section 2.2, and (iii) follows from the definition of $\mathcal{F}(\mathbb{D},c)$.

To prove the other relations, the basic strategy is similar to Proposition \ref{invtheo} - we establish a bijection between the terms in the state sum $\langle D ,c \rangle=\vert \delta\vert^{-1}\sum_{s\in S(D, \delta)} M(s)\cdot A(s)$ of the left-hand and right-hand side of each relation.

As an illustration, let us prove the first relation in (iv). Note that we used similar relations in the proof of Proposition \ref{initial} without giving a proof. Suppose the left-hand diagram in the first relation of (iv) is $D$ and the right-hand two diagrams are $D_1$ and $D_2$. Clearly, the factor $\mathcal{F}(\mathbb{D},c) \mathcal{C}(D,c)$ is the same for $D, D_1$ and $D_2$, so we are left to compare $(t^{-1/2}-t^{1/2})^{1-\vert V \vert}\langle D ,c \rangle$ of both sides.

In the following calculations, we again adopt the notation $\{i\}:=t^{i/2}-t^{-i/2}=[i](t^{1/2}-t^{-1/2}).$
Note that a state $s$ of the left-hand diagram $D$ assigns one of the four corners N, E, S, W at the crossings:  
\begin{equation*}
\begin{tikzpicture}[baseline=-0.65ex, thick, scale=0.8]
\draw (1,-1)   -- (0.2,-0.2);
\draw (-1, -1) [->] -- (1, 1) node[above]{$i$};
\draw (-0.2,0.2) [->] to (-1,1)  node[above]{$j$};
\draw (0,0.5) node{$\bullet$};
\draw (0, -1.5) node{N};
\end{tikzpicture}\quad\quad
\begin{tikzpicture}[baseline=-0.65ex, thick, scale=0.8]
\draw (1,-1)   -- (0.2,-0.2);
\draw (-1, -1) [->] -- (1, 1) node[above]{$i$};
\draw (-0.2,0.2) [->] to (-1,1)  node[above]{$j$};
\draw (0.5,0) node{$\bullet$};
\draw (0, -1.5) node{E};
\end{tikzpicture}
\quad\quad
\begin{tikzpicture}[baseline=-0.65ex, thick, scale=0.8]
\draw (1,-1)   -- (0.2,-0.2);
\draw (-1, -1) [->] -- (1, 1) node[above]{$i$};
\draw (-0.2,0.2) [->] to (-1,1)  node[above]{$j$};
\draw (-0.5,0) node{$\bullet$};
\draw (0, -1.5) node{W};
\end{tikzpicture}
\quad\quad
\begin{tikzpicture}[baseline=-0.65ex, thick, scale=0.8]
\draw (1,-1)   -- (0.2,-0.2);
\draw (-1, -1) [->] -- (1, 1) node[above]{$i$};
\draw (-0.2,0.2) [->] to (-1,1)  node[above]{$j$};
\draw (0,-0.5) node{$\bullet$};
\draw (0, -1.5) node{S};
\end{tikzpicture}
\end{equation*}
The corresponding local contributions to the state sum $\langle D,c \rangle$ are $-t^{-i}$, $1$, $t^{-i}$ and $1$, respectively.

For the right-hand side, we can similarly divide all states into N-, E-, W- or S-states according to the local assignments in $D_1$ and $D_2$.  For example, there is no E-state for $D_1$ or N-state for $D_2$; the N-states for $D_1$ and the E-states for $D_2$ have the local assignments as depicted below,
\begin{align*}
\begin{tikzpicture}[baseline=-0.65ex,thick,scale=0.65]
\draw (0,-2) node[below]{$i$} [->-] to (0, 1);
\draw (0, 1) [->] to (0,2) node[above]{$j$};
\draw (0,1) node[circle,fill,inner sep=1pt]{};
\draw (3,0) node[circle,fill,inner sep=1pt]{};
\draw (3,-2) [->-] node[below]{$j$} to (3,0);
\draw (3,0) [->] to (3,2) node[above]{$i$};
\draw (3,-0) [->-] to (0, 1);
\draw (1.5, 0) node {$j-i$};
\draw (3,0) circle (0.6);
\draw (0,1) circle (0.6);
\draw (3, -0.4) node{$\bullet$};
\draw (0, 0.6) node{$\bullet$};
\draw (0.8, 0.9) node{$\bullet$};
\draw (1.5, -3) node{N};
\end{tikzpicture}\quad\quad\quad\quad
\begin{tikzpicture}[baseline=-0.65ex,thick,scale=0.8]
\draw (-1,-1.5) [->-] node[below]{$i$} -- (0,-0.5);
\draw (1,-1.5) [->-] node[below]{$j$} -- (0,-0.5);
\draw (0, -0.5) [->-] to (0, 1);
\draw (0,1) [->] -- (-1, 2);
\draw (0,1) [->] -- (1,2);
\draw (0,-0.5) node[circle,fill,inner sep=1pt]{};
\draw (1.2, 2.3) node {$i$};
\draw (-1.2, 2.3) node {$j$};
\draw (0.8, 0.25) node {$i+j$};
\draw (0,-0.5) circle (0.5);
\draw (0,1) circle (0.5);
\draw (-0.2, -0.7) node{$\bullet$};
\draw (0.6, -0.8) node{$\bullet$};
\draw (0, 0.7) node{$\bullet$};
\draw (0, -2.5) node{E};
\end{tikzpicture}
\end{align*}
and their contributions to the state sum are $\{i\}\{j\}t^{(j-i)/2}$ and $\{i\}\{i+j\}t^{j/2}$, respectively. Note that the number of vertices in $D_1$ and $D_2$ is $2$ greater than that of $D$. After multiplying the coefficients in the first relation of (iv), we see that the local contributions of all N-states and E-states to both sides coincide.
\begin{eqnarray*}
&\text{N-states:}& \quad \quad \displaystyle \frac{-t^{-\frac{i+j}{2}}}{[i]\,[j]} \cdot (t^{-1/2}-t^{1/2})^{-2} \{i\}\{j\}t^{(j-i)/2}=-t^{-i}\\
&\text{E-states:}& \quad \quad \displaystyle \frac{t^{-\frac{j}{2}}}{[i]\,[i+j]} \cdot (t^{-1/2}-t^{1/2})^{-2} \{i\}\{i+j\}t^{j/2}=1.
\end{eqnarray*}

Similarly, for the local assignment (W), there are the corresponding W-states for $D_1$ and $D_2$:
\begin{align*}
\text{W:}\quad \quad
\begin{tikzpicture}[baseline=-0.65ex,thick,scale=0.65]
\draw (0,-2) node[below]{$i$} [->-] to (0, 1);
\draw (0, 1) [->] to (0,2) node[above]{$j$};
\draw (0,1) node[circle,fill,inner sep=1pt]{};
\draw (3,0) node[circle,fill,inner sep=1pt]{};
\draw (3,-2) [->-] node[below]{$j$} to (3,0);
\draw (3,0) [->] to (3,2) node[above]{$i$};
\draw (3,-0) [->-] to (0, 1);
\draw (1.5, 0) node {$j-i$};
\draw (3,0) circle (0.6);
\draw (0,1) circle (0.6);
\draw (3, -0.4) node{$\bullet$};
\draw (-0.2, 0.2) node{$\bullet$};
\draw (0.35, 0.9) node{$\bullet$};
\end{tikzpicture}\quad\quad \text{and}
\begin{tikzpicture}[baseline=-0.65ex,thick,scale=0.8]
\draw (-1,-1.5) [->-] node[below]{$i$} -- (0,-0.5);
\draw (1,-1.5) [->-] node[below]{$j$} -- (0,-0.5);
\draw (0, -0.5) [->-] to (0, 1);
\draw (0,1) [->] -- (-1, 2);
\draw (0,1) [->] -- (1,2);
\draw (0,-0.5) node[circle,fill,inner sep=1pt]{};
\draw (1.2, 2.3) node {$i$};
\draw (-1.2, 2.3) node {$j$};
\draw (0.8, 0.25) node {$i+j$};
\draw (0,-0.5) circle (0.5);
\draw (0,1) circle (0.5);
\draw (0.2, -0.7) node{$\bullet$};
\draw (-0.6, -0.8) node{$\bullet$};
\draw (0, 0.7) node{$\bullet$};
\end{tikzpicture}.
\end{align*}
Their contributions to the state sums are $\{j-i\}\{j\}t^{-i/2}$ and $\{i+j\}\{j\}t^{-i/2}$, respectively. We can check that
\begin{eqnarray*}
\text{W-states:} \quad \quad \displaystyle&& \frac{-t^{-\frac{i+j}{2}}}{[i]\,[j]} \cdot (t^{-1/2}-t^{1/2})^{-2} \{j-i\}\{j\}t^{-i/2}\\
&&+\frac{t^{-\frac{j}{2}}}{[i]\,[i+j]} \cdot (t^{-1/2}-t^{1/2})^{-2}\cdot \{i+j\}\{j\}t^{-i/2}=t^{-i}.
\end{eqnarray*}

For the local assignment (S), there are the corresponding S-states for $D_1$ and $D_2$:
\begin{align*}
\text{S:}\quad \quad
\begin{tikzpicture}[baseline=-0.65ex,thick,scale=0.65]
\draw (0,-2) node[below]{$i$} [->-] to (0, 1);
\draw (0, 1) [->] to (0,2) node[above]{$j$};
\draw (0,1) node[circle,fill,inner sep=1pt]{};
\draw (3,0) node[circle,fill,inner sep=1pt]{};
\draw (3,-2) [->-] node[below]{$j$} to (3,0);
\draw (3,0) [->] to (3,2) node[above]{$i$};
\draw (3,-0) [->-] to (0, 1);
\draw (1.5, 0) node {$j-i$};
\draw (3,0) circle (0.6);
\draw (0,1) circle (0.6);
\draw (3, -0.4) node{$\bullet$};
\draw (0.2, 0.2) node{$\bullet$};
\draw (0.35, 0.9) node{$\bullet$};
\draw (0, 0.65) node{$\circ$};
\draw (0.7, 0.5) node{$\circ$};
\end{tikzpicture}\quad\quad \text{and}
\begin{tikzpicture}[baseline=-0.65ex,thick,scale=0.8]
\draw (-1,-1.5) [->-] node[below]{$i$} -- (0,-0.5);
\draw (1,-1.5) [->-] node[below]{$j$} -- (0,-0.5);
\draw (0, -0.5) [->-] to (0, 1);
\draw (0,1) [->] -- (-1, 2);
\draw (0,1) [->] -- (1,2);
\draw (0,-0.5) node[circle,fill,inner sep=1pt]{};
\draw (1.2, 2.3) node {$i$};
\draw (-1.2, 2.3) node {$j$};
\draw (0.8, 0.25) node {$i+j$};
\draw (0,-0.5) circle (0.5);
\draw (0,1) circle (0.5);
\draw (0.2, -0.7) node{$\bullet$};
\draw (-0.3, -1.1) node{$\bullet$};
\draw (0, 0.7) node{$\bullet$};
\draw (-0.2, -0.7) node{$\circ$};
\draw (0.3, -1.1) node{$\circ$};
\end{tikzpicture}.
\end{align*}
Their contributions to the state sums sum up to $\{j\}^2$ and $\{i+j\}^2$, respectively. We can check that
\begin{eqnarray*}
\text{S-states:} \quad \quad \displaystyle&& \frac{-t^{-\frac{i+j}{2}}}{[i]\,[j]} \cdot (t^{-1/2}-t^{1/2})^{-2} \{j\}^2\\
&&+\frac{t^{-\frac{j}{2}}}{[i]\,[i+j]} \cdot (t^{-1/2}-t^{1/2})^{-2}\cdot \{i+j\}^2=1.
\end{eqnarray*}
Thus, we have exhausted all states of both sides and verified the first relation in (iv).

The rest of the relations can be shown in a similar way. We remark that in the proof of relations (v) and (vii), we need to consider the change of curliness $\mathcal{C}(D, c)$ by applying Lemma \ref{curli}. We omit the details here.

\end{proof}

\subsection{Graphical definition of the Alexander polynomial of a link.}

The MOY-type relations in the previous section provides us a graphical definition of the Alexander polynomial of a link.

\begin{defn}\label{def:Alexanderpolynomiallink}
\rm
Suppose $D$ is a diagram of a link $L$.  Let $w(D)$ be the \emph{writhe} of $D$, i.e., the number of positive crossings \diaCrossP minus the number of negative crossings \diaCrossN. Define
\begin{equation}
\Delta(D):=t^{w(D)/2}(\mathbb{D}),
\end{equation}
where $\mathbb{D}$ is the diagram $D$ equipped with a blackboard framing and a constant coloring $c(e)=1$ for each edge $e$ (which corresponds to the link component in this case).
\end{defn}

The next lemma implies that $\Delta(D)$ is an invariant of the link $L$.

\begin{lemma}
$\Delta(D)$ is invariant under Reidemeister moves (I), (II) and (III).
\end{lemma}

\begin{proof}
For move (I), we apply the relations (ii) (iv) and (v) and obtain
\begin{align*}
\left( \begin{tikzpicture}[baseline, thick, scale=0.35]
\draw (1,-1)   -- (0.2,-0.2);
\draw (-1, -1) [->] -- (1, 1);
\draw (-0.2,0.2) --  (-1,1) ;
\draw (-1, 1) arc (90:270:1);
\end{tikzpicture} \right)
&=-t^{-1}\left( \begin{tikzpicture}[baseline, thick, scale=0.35]
\draw[
        decoration={markings, mark=at position 0.625 with {\arrow{>}}},
        postaction={decorate}
        ]
(0,0) circle (1.3);
\draw (1.4,1.2) node {$1$};
\draw (2.5, -1.5) [->] -- (2.5, 1.5);
\draw (3.2,1.2) node {$1$};
\end{tikzpicture} \right)
+ \frac{t^{-1/2}}{[1]\,[2]}
\left(\begin{tikzpicture}[baseline=-0.65ex, thick, scale=0.4]
\draw (-1,-2)   -- (0,-1);
\draw (1,-2)  -- (0,-1);
\draw (0,0) [->] -- (-1, 1) node[above]{$1$};
\draw (0,0) [->] to (1,1)  node[above]{$1$};
\draw (0,0) node[circle,fill,inner sep=1pt]{};
\draw (0,-1) node[circle,fill,inner sep=1pt]{};
\draw (0.6, -0.25) node {$2$};
\draw (1.3, -2) node {$1$};
\draw (0, -1) [->-] to (0,0);
\draw (-1,1) to [out=180, in=90] (-2,-0.5)  to [out=270,in=180] (-1,-2);
\end{tikzpicture}\right)\\
&=\frac{t^{-1/2}[1]\,[2] }{[1]\,[2]}
\left(\begin{tikzpicture}[baseline=-0.65ex, thick, scale=0.4]
\draw (2.5, -1.5) [->] -- (2.5, 1.5);
\draw (3.2,1.2) node {$1$};
\end{tikzpicture}\right)
=t^{-1/2}
\left(\begin{tikzpicture}[baseline=-0.65ex, thick, scale=0.4]
\draw (2.5, -1.5) [->] -- (2.5, 1.5);
\end{tikzpicture}\right).
\end{align*}
The remaining cases of move (I) can be proved in a similar way.

The proof of the invariance under Reidemeister moves (II) and (III) follows the same argument of \cite[Theorem 3.1]{MR1659228}. We leave the detail to the reader.
\end{proof}

Note that Theorem \ref{moytheo} (iv) implies:

\begin{align*}
 \quad
&\left(\begin{tikzpicture}[baseline=-0.65ex, thick, scale=0.7]
\draw (1,-1)   -- (0.2,-0.2);
\draw (-1, -1) [->] -- (1, 1);
\draw (-0.2,0.2) [->] to (-1,1);
\end{tikzpicture}\right)
  =-t^{-1}
\left(
\begin{tikzpicture}[baseline=-0.65ex, thick, scale=0.7]
\draw (0, -1) [->] to (0,1);
\draw (1.5,-1) [->] to (1.5,1);
\draw (-0.5,1) node {$1$};
\draw (2,1) node {$1$};
\end{tikzpicture}\right)
+ \, \frac{t^{-1/2}}{[1]\,[2]}\cdot
\left(\begin{tikzpicture}[baseline=-0.65ex, thick, scale=0.5]
\draw (-1,-1.5)   -- (0,-0.5);
\draw (1,-1.5)  -- (0,-0.5);
\draw (0,0.5) [->] -- (-1, 1.5);
\draw (0,0.5) [->] to (1,1.5);
\draw (0,0.5) node[circle,fill,inner sep=1pt]{};
\draw (0,-0.5) node[circle,fill,inner sep=1pt]{};
\draw (0.6, 0.25) node {$2$};
\draw (-1.3, -1.5) node {$1$};
\draw (1.3, -1.5) node {$1$};
\draw (-1.3, 1.5) node {$1$};
\draw (1.3, 1.5) node {$1$};
\draw (0, -0.5) [->-] to (0,0.5);
\end{tikzpicture}\right), \text{ \quad and }\\
&\left(\begin{tikzpicture}[baseline=-0.65ex, thick, scale=0.7]
\draw (1,-1) [->]  -- (-1,1) ;
\draw (-1, -1) to (-0.2, -0.2);
\draw (0.2, 0.2) [->] -- (1, 1);
\end{tikzpicture}\right)
  =-t
\left(
\begin{tikzpicture}[baseline=-0.65ex, thick, scale=0.7]
\draw (0, -1) [->] to (0,1);
\draw (1.5,-1) [->] to (1.5,1);
\draw (-0.5,1) node {$1$};
\draw (2,1) node {$1$};
\end{tikzpicture}\right)
+ \, \frac{t^{1/2}}{[1]\,[2]}\cdot
\left(\begin{tikzpicture}[baseline=-0.65ex, thick, scale=0.5]
\draw (-1,-1.5)   -- (0,-0.5);
\draw (1,-1.5)  -- (0,-0.5);
\draw (0,0.5) [->] -- (-1, 1.5);
\draw (0,0.5) [->] to (1,1.5);
\draw (0,0.5) node[circle,fill,inner sep=1pt]{};
\draw (0,-0.5) node[circle,fill,inner sep=1pt]{};
\draw (0.6, 0.25) node {$2$};
\draw (-1.3, -1.5) node {$1$};
\draw (1.3, -1.5) node {$1$};
\draw (-1.3, 1.5) node {$1$};
\draw (1.3, 1.5) node {$1$};
\draw (0, -0.5) [->-] to (0,0.5);
\end{tikzpicture}\right).\\
\end{align*}

Thus, the link invariant $\Delta(D)$ satisfies the skein relation:

\begin{equation*}
\Delta\left(\begin{tikzpicture}[baseline=-0.65ex, thick, scale=0.5]
\draw (1,-1)   -- (0.2,-0.2);
\draw (-1, -1) [->] -- (1, 1);
\draw (-0.2,0.2) [->] to (-1,1);
\end{tikzpicture}\right)-
\Delta\left(\begin{tikzpicture}[baseline=-0.65ex, thick, scale=0.5]
\draw (1,-1) [->]  -- (-1,1) ;
\draw (-1, -1) to (-0.2, -0.2);
\draw (0.2, 0.2) [->] -- (1, 1);
\end{tikzpicture}\right)
=(t^{1/2}-t^{-1/2})
\Delta\left(
\begin{tikzpicture}[baseline=-0.65ex, thick, scale=0.5]
\draw (0, -1) [->] to (0,1) ;
\draw (2,-1) [->] to (2,1) ;
\end{tikzpicture}\right).
\end{equation*}
In addition, Theorem \ref{moytheo} (i) implies that $\Delta(D)=1$ when $D$ is a trivial knot diagram.  Altogether, it follows that $\Delta(D)$ is the same as the classical one-variable Alexander polynomial of a link.

\subsection{MOY-type relations characterize $\Delta_{(\mathbb{G},c)}$.}

Grant showed in \cite{MR3125899} that the classical MOY relations in \cite{MR1659228} uniquely determine the MOY polynomials for colored trivalent graphs. Our next theorem gives a parallel result for the Alexander polynomial $\Delta_{(\mathbb{G},c)}(t)$.

\begin{theo}\label{theorem:axiom}
The relations (i)-(x) determine $\Delta_{(\mathbb{G},c)}(t)$ for a framed trivalent graph $\mathbb{G}$ with a positive coloring $c$.
\end{theo}

We first assert that $\Delta_{(\mathbb{G},c)}(t)$ is computable when we restrict to trivalent plane graphs colored with $\{1,2\}$. This is analogous to a result of Grant \cite[Proposition 5.2]{MR3125899} and can be proved verbatim.

\begin{lemma}\label{lemma:trivalent}
The relations (i)-(x) determine $\Delta_{(\mathbb{G},c)}(t)$ when $(\mathbb{G},c)$ is a trivalent framed plane graph colored with $\{1,2\}$.
\end{lemma}

\begin{proof}[Proof of Theorem \ref{theorem:axiom}]
Given a framed trivalent graph diagram $\mathbb{D}$, we first apply (iii) and (iv) to eliminate all framing symbols and resolve all crossings to get plane graphs. Therefore, it suffices to show that $\Delta_{(\mathbb{D},c)}(t)$ is computable when $\mathbb{D}$ is a diagram of a plane graph.

The subsequent argument is adapted from \cite[Theorem 5.1]{MR3125899}.  Our goal is to use a sequence of the MOY-type relations to replace all edges colored with $m$ by edges with colors smaller than $m$, where $m>2$ is the maximal color of the graph (following the idea of Wu \cite{MR3234803}). If this can be achieved, then for any colored diagram $\mathbb{D}$, there is a diagram $\mathbb{D}'$ that is colored only with $\{1,2\}$ such that $(\mathbb{D})=p(t)\cdot(\mathbb{D}')$ for some function $p(t)$. It then follows from Lemma \ref{lemma:trivalent} that $(\mathbb{D})$ is computable.

To this end, suppose that $\mathbb{D}$ contains an edge colored with $m>2$. Since all the edges have positive colorings, it locally looks like
\begin{equation*}
\begin{tikzpicture}[baseline=-0.65ex, thick]
\draw (0,-1) [->-] to [out=90,in=270] (0.5,-0.33);
\draw (0.5, -0.33) [->-] to [out=90,in=270] (0.5,0.33);
\draw (0.5, 0.33) [->] to [out=90,in=270] (0,1);
\draw (1,-1) [->-] to [out=90,in=270] (0.5,-0.33);
\draw (1,1) [<-] to [out=270,in=90] (0.5,0.33);
\draw (0, -1.25) node {$j$};
\draw (0.9,-1.25) node {$m-j$};
\draw (0,1.25) node {$l$};
\draw (1,1.25) node {$m-l$};
\draw (1,0) node {$m$};
\draw (0.5, -0.33) node[circle,fill,inner sep=1pt]{};
\draw (0.5, 0.33) node[circle,fill,inner sep=1pt]{};
\end{tikzpicture}
\end{equation*}
with $0<j, l <m$. If $j \geq 2$ or $l \geq 2$ we apply (vi) and (viii) to obtain the following identity:

\begin{equation}\label{RelationMaximalColor}
\left( \begin{tikzpicture}[baseline=-0.65ex, thick]
\draw (0,-1) [->-] to [out=90,in=270] (0.5,-0.33);
\draw (0.5, -0.33) [->-] to [out=90,in=270] (0.5,0.33);
\draw (0.5, 0.33) [->] to [out=90,in=270] (0,1);
\draw (1,-1) [->-] to [out=90,in=270] (0.5,-0.33);
\draw (1,1) [<-] to [out=270,in=90] (0.5,0.33);
\draw (0, -1.25) node {$j$};
\draw (0.9,-1.25) node {$m-j$};
\draw (0,1.25) node {$l$};
\draw (1,1.25) node {$m-l$};
\draw (1,0) node {$m$};
\draw (0.5, -0.33) node[circle,fill,inner sep=1pt]{};
\draw (0.5, 0.33) node[circle,fill,inner sep=1pt]{};
\end{tikzpicture}\right)
=\frac{1}{[j]^2\,[l]^2}\left(\begin{tikzpicture}[baseline=-0.65ex, thick, scale=0.85]
\draw (0,-3) [->-] to [out=90,in=270] (0,-2.5);
\draw (0, -2.5) [->-] to [out=135,in=270] (-0.5,-2) to [out=90,in=225] (0,-1.5);
\draw (0,-1.5) [->-] to [out=90,in=270] (1,-0.5);
\draw (0, -2.5) node[circle,fill,inner sep=1pt]{};
\draw (0, -1.5) node[circle,fill,inner sep=1pt]{};
\draw (0,-2.5) [->-] to [out=45,in=270] (0.5,-2) to [out=90,in=315] (0,-1.5);
\draw (1, -0.5) node[circle,fill,inner sep=1pt]{};
\draw (1,0.5) [->-] to [out=90,in=270] (0,1.5);
\draw (0, 1.5) [->-] to [out=45,in=270] (0.5,2) to [out=90,in=315] (0,2.5);
\draw (0, 2.5) [->] to [out=90,in=270] (0,3);
\draw (1, 0.5) node[circle,fill,inner sep=1pt]{};
\draw (0, 1.5) node[circle,fill,inner sep=1pt]{};
\draw (0, 2.5) node[circle,fill,inner sep=1pt]{};
\draw (0,1.5) [->-] to [out=135,in=270] (-0.5,2) to [out=90,in=225] (0,2.5);
\draw (1,-0.5) [->-] to (1,0.5);
\draw (2,-3) [->-] to [out=90,in=270] (1,-0.5);
\draw (2,3) [<-] to [out=270,in=90] (1,0.5);
\draw (0, -3.25) node {$j$};
\draw (2,-3.25) node {$m-j$};
\draw (0,3.25) node {$l$};
\draw (2,3.25) node {$m-l$};
\draw (1.6,0) node {$m$};
\draw (1,-2.4) node {$j-1$};
\draw (-0.75,-2) node {$1$};
\draw (0.1,-1) node {$j$};
\draw (-0.75,2) node {$1$};
\draw (1.1,2) node {$l-1$};
\draw (0.1,1) node {$l$};
\end{tikzpicture}\right)
= \frac{1}{[j]\,[l]\,[m-1]^2}\left( \begin{tikzpicture}[baseline=-0.65ex,thick, scale=0.85]
\draw (0,-3) [->-] to (0, -2.7);
\draw (0, -2.7) to [out=90,in=270] (0,-2);
\draw (0, -2) [->-] to [out=90,in=270] (1,-0.5);
\draw (1, -0.5) [->-] to [out=90,in=270] (1,0.5);
\draw (1, 0.5) [->-] to [out=90,in=270] (0,2) to [out=90,in=270] (0,2.7);
\draw (0, 2.7) [->] to (0, 3);
\draw (0, -2.7) node[circle,fill,inner sep=1pt]{};
\draw (1, -0.5) node[circle,fill,inner sep=1pt]{};
\draw (1, 0.5) node[circle,fill,inner sep=1pt]{};
\draw (0, 2.7) node[circle,fill,inner sep=1pt]{};
\draw (2,-3) [->-] to [out=90,in=270](2,-2);
\draw (2, -2) [->-] to [out=90,in=270]  (1,-0.5);
\draw (2,3) [<-] to [out=270,in=90] (2,2);
\draw (2,2) [-<-] to [out=270,in=90] (1,0.5);
\draw (0,2.7) [-<-] to [out=270, in=180] (0.5,2.35) to [out=0,in=180] (1.5,2.35) to [out=0,in=90] (2,2);
\draw (0,-2.7) [->-] to [out=90,in=180] (0.5,-2.35) to [out=0,in=180] (1.5,-2.35) to [out=0,in=270] (2,-2);
\draw (2, -2) node[circle,fill,inner sep=1pt]{};
\draw (2, 2) node[circle,fill,inner sep=1pt]{};
\draw (0, -3.25) node {$j$};
\draw (2,-3.25) node {$m-j$};
\draw (0,3.25) node {$l$};
\draw (2,3.25) node {$m-l$};
\draw (1.6,0) node {$m$};
\draw (1,2.6) node {$l-1$};
\draw (1,-2.6) node {$j-1$};
\draw (-0.25,2) node {$1$};
\draw (-0.25,-2) node {$1$};
\draw (2.2,1.1) node {$m-1$};
\draw (2.2,-1.1) node {$m-1$};
\end{tikzpicture}\right).
\end{equation}

Now we may assume that locally the diagram that contains the edge colored with $m$ looks like

\begin{equation*}
\begin{tikzpicture}[baseline=-0.65ex, thick]
\draw (0,-1) [->-] to [out=90,in=270] (0.5,-0.33);
\draw (0.5, -0.33) [->-] to [out=90,in=270] (0.5,0.33);
\draw (0.5, 0.33) [->] to [out=90,in=270] (0,1);
\draw (1,-1) [->-] to [out=90,in=270] (0.5,-0.33);
\draw (1,1) [<-] to [out=270,in=90] (0.5,0.33);
\draw (0, -1.25) node {$1$};
\draw (0.9,-1.25) node {$m-1$};
\draw (0,1.25) node {$1$};
\draw (1,1.25) node {$m-1$};
\draw (1,0) node {$m$};
\draw (0.5, -0.33) node[circle,fill,inner sep=1pt]{};
\draw (0.5, 0.33) node[circle,fill,inner sep=1pt]{};
\end{tikzpicture}
\end{equation*}

We then apply (ix) with $i=k=l=1$ and $j=m-1$ and (x) to get
\begin{equation*}
\left(
\begin{tikzpicture}[baseline=-0.65ex, thick]
\draw (0,-1) [->-] to [out=90,in=270] (0.5,-0.33);
\draw (0.5, -0.33) [->-] to [out=90,in=270] (0.5,0.33);
\draw (0.5, 0.33) [->] to [out=90,in=270] (0,1);
\draw (1,-1) [->-] to [out=90,in=270] (0.5,-0.33);
\draw (1,1) [<-] to [out=270,in=90] (0.5,0.33);
\draw (0, -1.25) node {$1$};
\draw (0.9,-1.25) node {$m-1$};
\draw (0,1.25) node {$1$};
\draw (1,1.25) node {$m-1$};
\draw (1,0) node {$m$};
\draw (0.5, -0.33) node[circle,fill,inner sep=1pt]{};
\draw (0.5, 0.33) node[circle,fill,inner sep=1pt]{};
\end{tikzpicture}\right)
= \frac{ [m]}{ [1] \,[2]\,[m-1]}\cdot
\left(\begin{tikzpicture}[baseline=-0.65ex, thick]
\draw (0,-1) [->-] to (0, -0.33);
\draw (0, -0.33) [->-] to (0, 0.33);
\draw (0, 0.33) [->] to (0,1);
\draw (2,-1) [->-] to (2, -0.66);
\draw (2, -0.66) [->-] to (2, 0.66);
\draw (2, 0.66) [->] to (2,1);
\draw (0,-0.33) to [out=270,in=180] (0.5,-0.5) [-<-] to [out=0,in=180] (1.5,-0.5) to [out=0, in=90] (2,-0.66);
\draw (0,0.33) to [out=90, in=180] (0.5,0.5) [->-] to [out=0,in=180] (1.5,0.5) to [out=0, in=270] (2,0.66);
\draw (0.1,-1.25) node {$1$};
\draw (2,-1.25) node {$m-1$};
\draw (0.4,0) node {$2$};
\draw (0.1,1.25) node {$1$};
\draw (1, -0.8) node {$1$};
\draw (1.3,0) node {$m-2$};
\draw (1, 0.8) node {$1$};
\draw (2,1.25) node {$m-1$};
\draw (0, -0.33) node[circle,fill,inner sep=1pt]{};
\draw (0, 0.33) node[circle,fill,inner sep=1pt]{};
\draw (2, -0.66) node[circle,fill,inner sep=1pt]{};
\draw (2, 0.66) node[circle,fill,inner sep=1pt]{};
\end{tikzpicture}\right)
-[m-2]\,[m]\cdot
\left(\begin{tikzpicture}[baseline=-0.65ex, thick]
\draw (0,-1)  to (0, 0.33) [->] to (0,1);
\draw (1,-1)  to (1, -0.33) [->] to (1,1);
\draw (0,1.25) node {$1$};
\draw (1,1.25) node {$m-1$};
\end{tikzpicture}\right).\\
\end{equation*}

Observe that the right-hand side diagrams of the equation has no local colors larger than $m-1$ now. This is exactly what we aimed for and thus concludes the proof.
\end{proof}


\section{Properties and applications}
In this section, we apply the state sum formula and the MOY-type relations established in the previous sections to prove some basic properties of the Alexander polynomials $\Delta_{(G, c)}(t)$ and $\Delta_{(\mathbb{G}, c)}(t)$.  As an application, we show that the Alexander polynomial can be used to detect the non-planarity and chirality of a spatial graph.

\subsection{Symmetries}

Suppose $D$ is a diagram of a graph $G$. Recall that the \emph{mirror image} of $D$, denoted $D^*$, is obtained from $D$ by changing all positive (resp. negative) crossings of $D$ into negative (resp. positive) ones.  The \emph{orientation reversal} of $D$, denoted $-D$, is obtained from reversing all the edge orientations of $D$.  We define $G^*$ and $-G$ as the graphs represented by $D^*$ and $-D$, respectively.  It is standard that $G^*$ and $-G$ only depends on $G$ and not on $D$.  When we have a framed graph $\mathbb{G}$, we can define $\mathbb{G}^*$ and $-\mathbb{G}$ in a similar way by also changing the framing symbols $\Ptwist$ and $\Ntwist$ in the corresponding diagrams accordingly.

\begin{prop} \label{prop:symm}
The normalized Alexander polynomial $\Delta_{(\mathbb{G}, c)}(t)$ of a framed trivalent graph with positive coloring satisfies the following symmetric properties.
\begin{enumerate}
\item 
$\Delta_{(\mathbb{G}, c)}(t)$=$\Delta_{(\mathbb{G}^*, c)}(t^{-1})$.
\item 
$\Delta_{(\mathbb{G}, c)}(t)$=$\Delta_{(-\mathbb{G}, c)}(t)$.
\end{enumerate}
In particular, a plane graph $\mathbb{G}$ with the blackboard framing satisfies the symmetric relation
$$ \Delta_{(\mathbb{G}, c)}(t)=\Delta_{(\mathbb{G}, c)}(t^{-1}).$$

\end{prop}

\begin{proof}
(i) Given a framed diagram $\mathbb{D}$ of $\mathbb{G}$, define $$\tilde{\Delta}_{(\mathbb{D}, c)}(t) :=\frac{\Delta_{(\mathbb{D}, c)}(t)+\Delta_{(\mathbb{D}^*, c)}(t^{-1})}{2}.$$  One can check that $\tilde{\Delta}$ satisfy all the MOY-type relations in Theorem \ref{moytheo}. On the other hand, we know from Theorem \ref{theorem:axiom} that those relations completely characterize the Alexander polynomial.  In other words,  $\tilde{\Delta}_{(\mathbb{D}, c)}(t)=\Delta_{(\mathbb{D}, c)}(t)$. Hence, $\Delta_{(\mathbb{D}, c)}(t)$=$\Delta_{(\mathbb{D}^*, c)}(t^{-1})$.

(ii) We can argue in a similar way.  This time let $$\tilde{\Delta}_{(\mathbb{D}, c)}(t) :=\frac{\Delta_{(\mathbb{D}, c)}(t)+\Delta_{(-\mathbb{D}, c)}(t)}{2},$$ and we are going to check that $\tilde{\Delta}$ satisfy all the MOY-type relations in Theorem \ref{moytheo}. Note that all relations except (vii) and (ix) are obviously true with the oppositely oriented edges, so it suffices for us to show that (vii) and (ix) also hold after the reverse of orientations.  More precisely, we need to verify the relations
\begin{align*}
& \left(\begin{tikzpicture}[baseline=-0.65ex, thick, scale=1.2]
\draw (0,-1) [<-] to [out=90,in=180] (0.5,-0.5);
\draw (0.5, -0.5) [-<-] to [out=0,in=180] (1,-0.5) ;
\draw (1, -0.5) [-<-] to [out=0, in=90] (1.5,-1);
\draw (0.5,-0.5) [->-] to [out=180, in=270] (0,0) to [out=90,in=180] (0.5,0.5);
\draw (1,-0.5) [-<-] to [out=0, in=270] (1.5,0) to [out=90,in=0] (1,0.5);
\draw (0,1) [->-] to [out=270,in=180] (0.5,0.5);
\draw (0.5, 0.5) [->-] to [out=0,in=180] (1,0.5);
\draw (1, 0.5) [->]  to [out=0, in=270] (1.5,1);
\draw (-0.25, -0.75) node {$j$};
\draw (1.75, -0.75) node {$i$};
\draw (-0.25, 0) node {$i$};
\draw (-0.25, 0.75) node {$j$};
\draw (1.75, 0.75) node {$i$};
\draw (1.75, 0) node {$j$};
\draw (0.75,-0.75) node {$i+j$};
\draw (0.75,0.75) node {$i+j$};
\draw (0.5, -0.5) node[circle,fill,inner sep=1pt]{};
\draw (1, -0.5) node[circle,fill,inner sep=1pt]{};
\draw (0.5, 0.5) node[circle,fill,inner sep=1pt]{};
\draw (1, 0.5) node[circle,fill,inner sep=1pt]{};
\end{tikzpicture}\right)
=\frac{[i+j]^3}{[i]}\cdot
\left(\begin{tikzpicture}[baseline=-0.65ex, thick, scale=1.2]
\draw (0,-1) [-<-] to [out=90, in=180] (0.5,-0.5);
\draw (0.5, -0.5)  to [out=0,in=180] (1,-0.5) [-<-]  to [out=0,in=90] (1.5,-1);
\draw (0.5, 0.5) [out=180, in=270]   [<-] to (0,1);
\draw (0.5,0.5) to [out=0,in=180] (1,0.5) [->] to [out=0,in=270] (1.5,1);
\draw (0.75,-0.5) to [out=180, in=270] (0.25,0) [->-] to [out=90,in=180] (0.75,0.5);
\draw (-0.25, -0.75) node {$j$};
\draw (1.75, -0.75) node {$i$};
\draw (1, 0) node {$i-j$};
\draw (-0.25, 0.75) node {$j$};
\draw (1.75, 0.75) node {$i$};
\draw (0.75, -0.5) node[circle,fill,inner sep=1pt]{};
\draw (0.75, 0.5) node[circle,fill,inner sep=1pt]{};
\end{tikzpicture}\right)
+[j]^2\,[i+j]^2\cdot
\left(\begin{tikzpicture}[baseline=-0.65ex, thick, scale=1.2]
\draw (0,-1) [<-] to (0,1);
\draw (1,-1) [->] to (1,1);
\draw (-0.25,0) node {$j$};
\draw (1.25, 0) node {$i$};
\end{tikzpicture}\right), 
\end{align*}
and
\begin{align*}
& \left(\begin{tikzpicture}[baseline=-0.65ex, thick, scale=1.2]
\draw (0,-1) [<-] to (0, -0.33);
\draw (0, -0.33) [-<-] to (0, 0.33);
\draw (0, 0.33) [-<-] to (0,1);
\draw (2,-1) [<-] to (2, -0.66);
\draw (2, -0.66) [-<-] to (2, 0.66);
\draw (2, 0.66) [-<-] to (2,1);
\draw (0,-0.33) to [out=270,in=180] (0.5,-0.5) [->-] to [out=0,in=180] (1.5,-0.5) to [out=0, in=90] (2,-0.66);
\draw (0,0.33) to [out=90, in=180] (0.5,0.5) [-<-] to [out=0,in=180] (1.5,0.5) to [out=0, in=270] (2,0.66);
\draw (0,-1.25) node {$i$};
\draw (2,-1.25) node {$j$};
\draw (0.5,0) node {$i+k$};
\draw (0.1,1.25) node {$i+k-l$};
\draw (1, -0.8) node {$k$};
\draw (1.5,0) node {$j-k$};
\draw (1, 0.8) node {$l$};
\draw (2,1.25) node {$j+l-k$};
\draw (0, -0.33) node[circle,fill,inner sep=1pt]{};
\draw (0, 0.33) node[circle,fill,inner sep=1pt]{};
\draw (2, -0.66) node[circle,fill,inner sep=1pt]{};
\draw (2, 0.66) node[circle,fill,inner sep=1pt]{};
\end{tikzpicture}\right)
= \frac{ [j]\,[l]\,[i+k]}{ [i+j]}\cdot
\left(
\begin{tikzpicture}[baseline=-0.65ex, thick, scale=1.2]
\draw (0,-1) [<-] to [out=90,in=270] (0.5,-0.33);
\draw (0.5, -0.33) [-<-] to [out=90,in=270] (0.5,0.33);
\draw (0.5, 0.33) [-<-] to [out=90,in=270] (0,1);
\draw (1,-1) [<-] to [out=90,in=270] (0.5,-0.33);
\draw (1,1) [->-] to [out=270,in=90] (0.5,0.33);
\draw (0, -1.25) node {$i$};
\draw (0.9,-1.25) node {$j$};
\draw (-0.5,1.25) node {$i+k-l$};
\draw (1.5,1.25) node {$j+l-k$};
\draw (1.2,0) node {$i+j$};
\draw (0.5, -0.33) node[circle,fill,inner sep=1pt]{};
\draw (0.5, 0.33) node[circle,fill,inner sep=1pt]{};
\end{tikzpicture}\right)\\
\quad & \hspace{4cm} + [i+k]\,[j-k]\cdot
\left(\begin{tikzpicture}[baseline=-0.65ex, thick, scale=1.2]
\draw (0,-1) [<-] to (0, 0.33);
\draw (0, 0.33) [-<-] to (0,1);
\draw (1,-1) [<-] to (1, -0.33);
\draw (1, -0.33) [-<-] to (1,1);
\draw (0,0.33) [->-]  to [out=270,in=180] (0.5,0) to [out=0,in=90] (1,-0.33);
\draw (-0.5,1.25) node {$i+k-l$};
\draw (0,-1.25) node {$i$};
\draw (1.5,1.25) node {$j+l-k$};
\draw (1,-1.25) node {$j$};
\draw (0.5,0.3) node {$k-l$};
\draw (1, -0.33) node[circle,fill,inner sep=1pt]{};
\draw (0, 0.33) node[circle,fill,inner sep=1pt]{};
\end{tikzpicture}\right).
\end{align*}
As both identities can be proved in the same way as the relations in Theorem \ref{moytheo}, we omit the details here.

On the other hand, we know from Theorem \ref{theorem:axiom} that those relations completely characterize the Alexander polynomial.  In other words,  $\tilde{\Delta}_{(\mathbb{D}, c)}(t)=\Delta_{(\mathbb{D}, c)}(t)$.  Hence, $\Delta_{(\mathbb{D}, c)}(t)$=$\Delta_{(-\mathbb{D}, c)}(t)$.

\end{proof}

\subsection{Integrality}

Recall from Definition \ref{Alex} that the Alexander polynomial $\Delta_{(G, c)}(t)$ equals $\frac{\langle D, c \rangle }{(t^{-1/2}-t^{1/2})^{\vert V \vert-1}}$ for any diagram $D$ of $G$.  A priori, $\Delta_{(G, c)}(t)$ is only a rational function of the variable $t^{\pm \frac{1}{2}}$.  Our next proposition shows that $\Delta_{(G, c)}(t)$ is a genuine polynomial of $t^{\pm \frac{1}{2}}$ when $G$ has at least one vertex.

\begin{prop}\label{prop:integrality}
Let $G$ be an MOY graph with at least one vertex.  Then the Alexander polynomial $\Delta_{(G, c)}(t) \in \mathbb{Z}[t^{\pm \frac{1}{2}}]$.

\end{prop}

\begin{proof}
Recall that $\langle D ,c \rangle=\vert \delta \vert^{-1}\sum_{s\in S(D, \delta)} M(s)\cdot A(s)$, where $A(s)=\prod_{p=1}^{N}A_{C_p}^{s(C_{p})}$ is defined by the local contribution as exhibited in Fig. \ref{fig:e1} (bottom).  Our key observation is that the contribution $A_{C_p}^{s(C_{p})}$ for a state $s$ at the crossing $C_p$ of type \diaCircle assigning $s(C_{p})$ the north corner inside the circular region is equal to $t^{-i/2}-t^{i/2}$ for some $i$, which is a factor divisible by $t^{-1/2}-t^{1/2}$.

To compute $\langle D , c\rangle$ for the graph with $|V| \geq 1$ vertices, we place the base point $\delta$ on an edge $e_1$ of color $i_1$ just before it goes into a vertex $v_1$.  Then, all states in this decorated diagram must assign the north corner inside the circular region corresponding to $v_1$, which contributes a factor $A_{C_1}^{\triangle}=t^{-i_1/2}-t^{i_1/2}$.  As $$|\delta|=t^{n-i_1}-t^n=t^{n-i_1/2}\cdot (t^{-i_1/2}-t^{i_1/2})$$ for some $n$,  the factor $|\delta|$ is cancelled by $A_{C_1}^{\triangle}$.  In addition, note that $A(s)$ of each state $s$ consists of a total of $|V|-1$ additional factors of the form $t^{-i/2}-t^{i/2}$, each of which corresponds to the contribution of $A_{C_p}^{s(C_{p})}$ coming from the remaining $|V|-1$ circular regions.  Consequently, $\frac{A(s)}{\vert \delta \vert (t^{-1/2}-t^{1/2})^{\vert V \vert-1}} \in \mathbb{Z}[t^{\pm \frac{1}{2}}]$ for all states $s$, and so is the Alexander polynomial $\Delta_{(G, c)}(t)$.

\end{proof}

\subsection{Positivity}

One of the fundamental problems in classical knot theory is to determine whether a given diagram is a projection of the unknot. The Alexander polynomial of knots gives a simple primary test for the triviality of knots.  Analogously, an interesting problem in spatial graph theory is to determine whether a given diagram is isotopic to a plane graph diagram.  We give in this section a necessary condition in terms of the Alexander polynomial of graphs.

When $D$ is a plane graph diagram, we first reformulate the state sum formula of $\Delta_{(D, c)}(t)$.  To this end, we have only the crossings of type \diaCircle and we define a new local contribution $P_{C_p}^{\triangle}$ as in Fig. \ref{fig:f2}.

\begin{figure}[h!]
\begin{tikzpicture}[baseline=-0.65ex, thick, scale=0.9]
\draw (0, 0.5) ellipse (1.5cm and 0.8cm);
\draw (0,-2) [->-] to (0,-0.3);
\draw (-0.5, -0.6) node {$t^{-i/2}$};
\draw (0.5, -0.6) node {$t^{i/2}$};
\draw (0, 0.2) node {$[i]$};
\draw (0.3, -1.5) node {$i$};
\end{tikzpicture}
\hspace{2cm}
\begin{tikzpicture}[baseline=-0.65ex, thick, scale=0.9]
\draw (0, 0.5) ellipse (1.5cm and 0.8cm);
\draw (0,-2) [->-] to (0,-0.3);
\draw (-0.7, -0.8) node {$n$};
\draw (1, -0.8) node {$n-i$};
\draw (0, 0.2) node {$t^{-n+i_1/2}$};
\draw (0.4, -1.3) node {$i_1$};
\draw (0, -1.7) node {$*$};
\draw (-0.3, -1.7) node {$\delta$};
\end{tikzpicture}
	\caption{The local contributions $P_{C_p}^{\triangle}$ for crossings that correspond to generic edges (left) and the marked edge (right), respectively.}
	\label{fig:f2}
\end{figure}

For each state $s$, let $P(s)= \prod_{p=1}^{N} P_{C_p}^{s(C_p)}$.  Note that for the crossing $C_1$ corresponding to the edge $e_1$ where $\delta$ lies,
$$P_{C_1}^{s(C_1)}=t^{-n+i_1/2}=   \frac{t^{-i_1/2}-t^{i_1/2}}  {t^{n-i_1}-t^n}= \frac{ M_{C_1}^{s(C_1)}\cdot A_{C_1}^{s(C_1)}  }{\vert \delta \vert  } ;                 $$
and for the north corner $s(C_p)$ of a generic crossing $C_p$,
$$P_{C_p}^{s(C_p)}=[i]=\frac{t^{-i/2}-t^{i/2}}{t^{-1/2}-t^{1/2}}=\frac{1}{t^{-1/2}-t^{1/2}} \cdot M_{C_p}^{s(C_p)}\cdot A_{C_p}^{s(C_p)}.$$
The proof of Proposition \ref{prop:integrality} shows that
$$\sum_{s\in S(D, \delta)}P(s)=  \vert \delta \vert^{-1} \frac{\sum_{s\in S(D, \delta)} M(s)\cdot A(s) }{(t^{-1/2}-t^{1/2})^{\vert V \vert-1}}$$
Hence,
\begin{equation} \label{equation: planar}
\Delta_{(D, c)}(t)=\sum_{s\in S(D, \delta)}  P(s).
\end{equation}

Using this new version of the state sum formula for the Alexander polynomial, it is straightforward to see $\Delta_{(G, c)}(t) \in \mathbb{Z}[t^{\pm \frac{1}{2}}]$, but we can actually prove more.  Note that
$$[i]= \frac{t^{i/2}-t^{-i/2}}{t^{1/2}-t^{-1/2}}=t^{\frac{i-1}{2}}+\cdots +t^{\frac{1-i}{2}} $$ is a polynomial of positive coefficient when $i>0$; so are all the other local contribution $P_{C_p}^{\triangle}$. Thus, every term $P(s)$ is a polynomial of positive coefficient.  Summing up, we obtain the following positivity condition on the Alexander polynomial of a plane graph.

\begin{theo}\label{condition}
Suppose $G$ is isotopic to a plane MOY graph with at least one vertex.  Then $\Delta_{(G, c)}(t) \in \mathbb{Z}_{\geq 0}[t^{\pm \frac{1}{2}}]$.
\end{theo}

\begin{rem}
\rm
Note that the ambiguity of a power of $t$ in the definition of $\Delta_{(G, c)}(t)$ does not matter, since the sign of the coefficients is unchanged.

\end{rem}

\bigskip
\begin{ex}
\rm

We conclude this section with a detailed calculation of the graph $G$ exhibit in Fig. \ref{fig:f1}, whose undirected version is denoted $5_1$ in Litherland's table of $\theta$-curve diagrams \cite[Fig. 1]{MR2463504}. In this example, we show the non-planarity and chirality of this graph.

\begin{figure}[h!]
	\centering
		\includegraphics[width=0.3\textwidth]{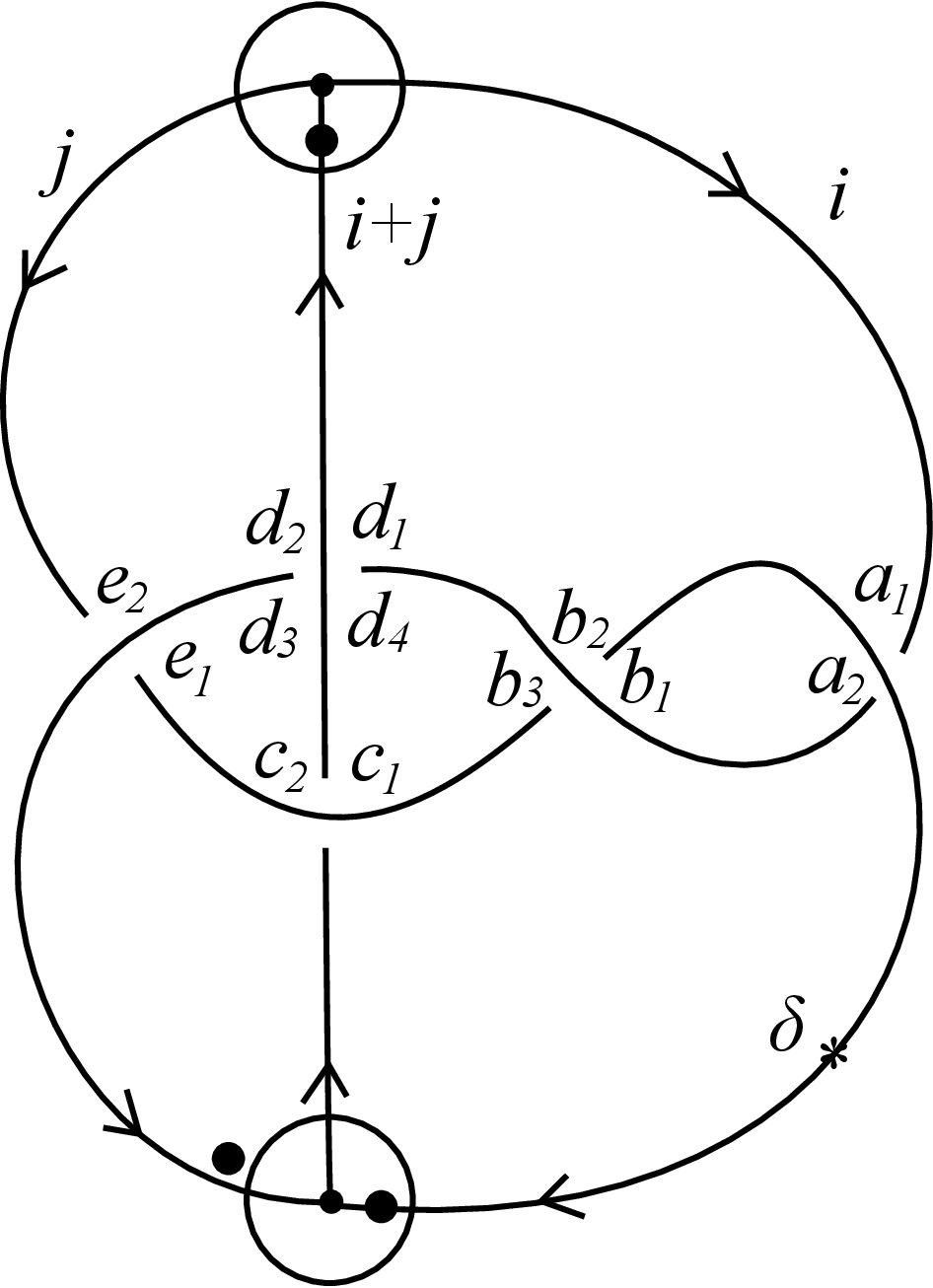}
	\caption{The $5_1$ in Litherland's table of $\theta$-curve diagrams.  Observe that any two of three curves of the graph form an unknot. With the given position of $\delta$, the local assignment of all states at the circular regions is unique and is marked out by big black dots. }
	\label{fig:f1}
\end{figure}

In the following calculation, we assign the balanced colors $i$, $j$ and $i+j$ and place the base point $\delta$ as in Fig. \ref{fig:f1}.  We want to compute $\Delta_{(G, c)}(t)$ using the state sum formula, so we begin by enumerating all Kauffman states in the given decorate diagram.  Note that the assignment at the 3 crossings of type \diaCircle is unique; thus, in order to specify the states, it is enough to describe the local assignment at the 5 crossings of types \diaCrossP and \diaCrossN.  By carefully exhausting all possibilities, we find 7 states in total.  We then compute the value of $\frac {M(s)\cdot A(s)} {\vert \delta \vert (t^{-1/2}-t^{1/2})^{\vert V \vert-1}}$ for each state $s$ and summarize the result in Table \ref{table2}.

\begin{table}[ht!]
  \begin{center}
  \setlength\extrarowheight{10pt}
\begin{tabular}{|c|c|}
  \hline
  state $s$ & Value of $\frac  {M(s)\cdot A(s)}        {\vert \delta \vert (t^{-1/2}-t^{1/2})^{\vert V \vert-1}}$      \\
  \hline
  $a_1b_1c_1d_2e_1$ & $t^{\frac{-3i-3j}{2}}\cdot[i+j]$ \\
  $a_1b_1c_1d_3e_2$ & $-t^{\frac{-i-3j}{2}} \cdot[i+j]$ \\
  $a_1b_1c_2d_4e_2$ & $t^{\frac{i-j}{2}} \cdot[i+j]$ \\
  $a_2b_2c_1d_2e_1$ & $-t^{\frac{-3i-j}{2}} \cdot[i+j]$ \\
  $a_2b_2c_1d_3e_2$ & $t^{\frac{-i-j}{2}} \cdot[i+j]$ \\
  $a_2b_2c_2d_4e_2$ & $-t^{\frac{+i+j}{2}} \cdot[i+j]$ \\
  $a_2b_3c_2d_1e_2$ & $t^{\frac{j-i}{2}} \cdot[i+j]$\\
  \hline
\end{tabular}
\end{center}
\bigskip
  \caption{All states and their values of $\frac  {M(s)\cdot A(s)}        {\vert \delta \vert (t^{-1/2}-t^{1/2})^{\vert V \vert-1}}$.   }
  \label{table2}
  \end{table}

Therefore,
\begin{align*}
\Delta_{(G, c)}(t)
&= \sum_{s\in S(D, \delta)}\frac  {M(s)\cdot A(s)}        {\vert \delta \vert (t^{-1/2}-t^{1/2})^{\vert V \vert-1}} \\
&= (t^{\frac{-3i-3j}{2}}-t^{\frac{-3i-j}{2}}-t^{\frac{-i-3j}{2}}+t^{\frac{-i-j}{2}}
+t^{\frac{i-j}{2}}+t^{\frac{j-i}{2}}-t^{\frac{i+j}{2}})\cdot[i+j]
\end{align*}

Suppose $i,j >0$.  Since $[i+j]=t^{\frac{i+j-1}{2}}+\cdots +t^{\frac{1-i-j}{2}}$, the term of the highest degree in $\Delta_{(G, c)}(t)$ is $-t^{-\frac{1}{2}+i+j}$, which has a negative coefficient $-1$.  By Theorem \ref{condition}, the given graph $5_1$ in Fig. \ref{fig:f1}  is not isotopic to any plane graph.

Furthermore, if we set the color $i=j=1$, the Alexander polynomial
$$\Delta_{(G, c)}(t) = (t^{-3}-2t^{-2}+t^{-1}+2-t)\cdot (t^{\frac{1}{2}}+t^{-\frac{1}{2}}).$$
As $\Delta_{(G, c)}(t) \neq \Delta_{(G, c)}(t^{-1})\cdot t^k$ for any $k$, Proposition \ref{prop:symm} in fact implies the stronger result that $5_1$ is \emph{chiral}, i.e., the graph is not isotopic to its mirror.

\end{ex}

\subsection{Non-vanishing properties}

In this section, our main goal is to establish the non-vanishing property for the Alexander polynomial of a connected trivalent plane graph with a positive coloring.  One should compare it with Kronheimer-Mrowka \cite[Theorem 1.1]{MR3880205}, although it is at present unclear to the authors about the connection of the Alexander polynomial to Tait colorings.

\begin{theo}\label{Nonvanishing}
Suppose $G$ is a connected trivalent plane graph with a positive coloring $c$.  Then $\Delta_{(G,c)}(t)\neq 0$.
\end{theo}

Recall that a \emph{bridge} is an edge of a graph whose removal increases its number of connected components.  Note that the color of a bridge has to be 0 for a balanced coloring. Hence, the graph $G$ that has a positive coloring in Theorem \ref{Nonvanishing} must not have a bridge.  In contrast, we claim that $\Delta_{(G,c)}(t)$ vanishes when $G$ is either not connected or contains a bridge.

\begin{prop}
Suppose $G$ is a plane graph that is either not connected or contains a bridge. Then $\Delta_{(G,c)}(t)=0$.
\end{prop}

\begin{proof}
Suppose $D$ is a diagram of $G$.  Remember that the state sum $\langle D ,c \rangle$ is defined to be $0$ when $D$ is not connected.  On the other hand, if $D$ has a bridge, then the color of the edge must be 0.  We can apply the relation (x) in Theorem \ref{moytheo} to obtain a disconnected diagram whose Alexander polynomial is $0$.

\end{proof}

The rest of the section is devoted to the proof of Theorem \ref{Nonvanishing}. When the vertex set $V$ of $G$ is empty, the graph $G$ is simply a trivial knot.  In that case, $\Delta_{(G,c)}(t)$ is non-zero by Theorem \ref{moytheo}(i).  From now on, let us assume that $D$ is a plane graph diagram of $G$ with at least one vertex with a positive coloring $c$, unless otherwise specified.

\begin{lemma}\label{LemmaOne}
Suppose $D$ is a trivalent plane graph diagram with a positive coloring $c$.  Then
$\Delta_{(D,c)}(t)\neq 0$ if and only if the set of states $S(D, \delta) \neq \emptyset$.
\end{lemma}

\begin{proof}
This follows from Equation (\ref{equation: planar}) and the observation that each term $P(s)$ is a polynomial of positive coefficients.
\end{proof}

Our plan is to prove $S(D, \delta) \neq \emptyset$ by induction.  Note that from Lemma \ref{LemmaOne}, the set $S(D, \delta) \neq \emptyset$ if and only if $S(D, \delta ') \neq \emptyset$ for any two points $\delta$ and $\delta '$.  Thus we have the freedom of placing the base point anywhere on $D$.  The next step is to relate the set $S(D, \delta)$ to the set $S(D', \delta)$ of a simpler graph diagram $D'$.  To this end, we start with a graph diagram $D$ that contains a local configuration as Fig. \ref{fig: SmoothOut} and resolve the vertices to get a new graph diagram $D'$.  Clearly, $D'$ is also a trivalent plane graph with a positive coloring.  In some cases, $D'$ may be disconnected.  Then decompose $D'=D_1'\sqcup D_2'$ so that the graphs $D_1'$ and $D_2'$ are connected.

\begin{figure}[h!]
	\centering
	\begin{equation*}D=
\begin{tikzpicture}[baseline=-0.65ex, thick]
\draw (0,-1) [->-] to [out=90,in=270] (0.5,-0.33);
\draw (0.5, -0.33) [->-] to [out=90,in=270] (0.5,0.33);
\draw (0.5, 0.33) [->] to [out=90,in=270] (0,1);
\draw (1,-1) [->-] to [out=90,in=270] (0.5,-0.33);
\draw (1,1) [<-] to [out=270,in=90] (0.5,0.33);
\draw (0, -1.25) node {$i$};
\draw (0.9,-1.25) node {$j$};
\draw (0,1.25) node {$i$};
\draw (1,1.25) node {$j$};
\draw (1.1,0) node {$i+j$};
\draw (0.5, -0.33) node[circle,fill,inner sep=1pt]{};
\draw (0.5, 0.33) node[circle,fill,inner sep=1pt]{};
\end{tikzpicture}
\xrightarrow{\text{resolve the vertices}}
\begin{tikzpicture}[baseline=-0.65ex, thick, scale=1.2]
\draw (0,-1) [->] to (0,1);
\draw (1,-1) [->] to (1,1);
\draw (-0.25,0) node {$i$};
\draw (1.25, 0) node {$j$};
\end{tikzpicture}=D'
\end{equation*}
	\caption{Resolve the vertices of a local configuration}
\label{fig: SmoothOut}
\end{figure}

\begin{lemma}\label{LemmaTwo}
For $D$, $D'$ as above, if $D'$ is connected, then $S(D', \delta) \neq \emptyset$ implies that $S(D, \delta) \neq \emptyset$; if $D'$ is disconnected, then $S(D'_i, \delta) \neq \emptyset$ for both $i=1,2$ implies that $S(D, \delta) \neq \emptyset$.

\end{lemma}

\begin{proof}
We split the set of states (if there are any) $S(D,\delta)=S_1 \sqcup S_2$ according to the local assignment shown in Fig. \ref{fig:f3}.

\begin{figure}
\begin{tikzpicture}[baseline=-0.65ex, thick, scale=0.8]
\draw (0, 1) ellipse (0.8cm and 0.5cm);
\draw (0, -1) ellipse (0.8cm and 0.5cm);
\draw (0,-1) [->-] to (0,1);
\draw (0.7, 0) node {$i+j$};
\draw (-1, -2)  [->-]  to (-0.3, -1.3);
\draw (-0.3, -1.3) to (0, -1);
\draw (1, -2)  [->-]  to (0.3, -1.3);
\draw (0.3, -1.3) to (0, -1);
\draw (1, 2) [<-] to (0, 1);
\draw (-1, 2) [<-] to (0, 1);
\draw (-1, -1.6) node {$i$};
\draw (1, -1.6) node {$j$};
\draw (-1, 1.6) node {$i$};
\draw (1, 1.6) node {$j$};
\draw (-0.7, 1.7) node {$*$};
\draw (-0.4, 1.9) node {$\delta$};
\draw (0, 1) node[circle,fill,inner sep=1pt]{};
\draw (0, -1) node[circle,fill,inner sep=1pt]{};
\draw (0.35, -1.6) node{$\bullet$};
\draw (0.3, -1.3) node {$\circ$};
\draw (-0.35, -1.6) node{$\circ$};
\draw (-0.3, -1.3) node {$\bullet$};
\draw (0, 0.65) node {$\bullet$};
\draw (0, -2.5) node {$S_2\subset S(D, \delta)$};
\end{tikzpicture}\,\,$\xleftrightarrow{\,1:1\,}$
\begin{tikzpicture}[baseline, thick, scale=0.4]
\draw (-1, 2) [<-]arc (110:250:2);
\draw (-5, 2) [<-]arc (70:-70:2);
\draw (-3.8, 1.9) node {$\delta$};
\draw (-4.4, 1.7) node {$*$};
\draw (-3.5, -3) node {$D': \text{ connected}$};
\end{tikzpicture} \quad\,\,
\begin{tikzpicture}[baseline=-0.65ex, thick, scale=0.8]
\draw (0, 1) ellipse (0.8cm and 0.5cm);
\draw (0, -1) ellipse (0.8cm and 0.5cm);
\draw (0,-1) [->-] to (0,1);
\draw (0.7, 0) node {$i+j$};
\draw (-1, -2)  [->-]  to (-0.3, -1.3);
\draw (-0.3, -1.3) to (0, -1);
\draw (1, -2)  [->-]  to (0.3, -1.3);
\draw (0.3, -1.3) to (0, -1);
\draw (1, 2) [<-] to (0, 1);
\draw (-1, 2) [<-] to (0, 1);
\draw (-1, -1.6) node {$i$};
\draw (1, -1.6) node {$j$};
\draw (-1, 1.6) node {$i$};
\draw (1, 1.6) node {$j$};
\draw (-0.7, 1.7) node {$*$};
\draw (-0.4, 1.9) node {$\delta$};
\draw (0, 1) node[circle,fill,inner sep=1pt]{};
\draw (0, -1) node[circle,fill,inner sep=1pt]{};
\draw (0.7, -1.5) node{$\bullet$};
\draw (-0.3, -1.3) node {$\bullet$};
\draw (0, 0.65) node {$\bullet$};
\draw (0, -2.5) node {$S_1\subset S(D, \delta)$};
\end{tikzpicture} \quad $\xleftrightarrow{\,1:1\,}$
\begin{tikzpicture}[baseline, thick, scale=0.4]
\draw (-1, 2) [<-]arc (110:250:2);
\draw (-5, 2) [<-]arc (70:-70:2);
\draw (-3.8, 1.9) node {$\delta$};
\draw (-4.4, 1.7) node {$*$};
\draw (-2.1, 1.9) node {$\delta$};
\draw (-1.5, 1.7) node {$*$};
\draw (-3.5, -3) node {$D': \text{disconnected}$};
\draw (-3.2, -4.5) node {$D'=D'_1\sqcup D'_2$};
\end{tikzpicture}
	\caption{Splitting the set of states $S(D,\delta)$ to $S_1$ and $S_2$.  Note that there are natural correspondences $S_1 \xleftrightarrow{1:1} S(D'_1, \delta)\times  S(D'_2, \delta)$ when $D'=D_1'\sqcup D_2'$, and $S_2 \xleftrightarrow{1:1} S(D', \delta)$ when $D'$ is connected.}
	\label{fig:f3}
\end{figure}

Suppose the resulted graph $D'$ from the above resolution is connected.  We claim that $|S_2|= |S(D',\delta)|$, because every state in $S(D',\delta)$ naturally extends to a state in $S_2$.  In particular, $S(D', \delta) \neq \emptyset$ implies that $S(D, \delta) \neq \emptyset$.

Otherwise, $D'$ is disconnected and assume that $D'=D_1'\sqcup D_2'$.  We claim that $|S_1|= |S(D_1',\delta)|\times |S(D_2',\delta)| $, because a pair of states in $S(D_i',\delta)$ can be naturally extended to a state in $S_1$. In particular, $S(D_i', \delta) \neq \emptyset$ for $i=1,2$ implies that $S(D, \delta) \neq \emptyset$.

\end{proof}

\begin{lemma} \label{LemmaThree}
Suppose $D$ is a connected trivalent plane graph diagram with at least one vertex, and there is a balanced color $c$ so that each edge is colored with $\{1, 2\}$.  Then $S(D, \delta) \neq \emptyset$.
\end{lemma}

\begin{proof}
We prove by induction on the number of vertices $|V|$ of $D$.  Since $D$ is trivalent, $|V|$ is an even number.  The base case is $|V|=2$.  The trivial $\theta$-curve is the unique planar trivalent graph of 2 vertices.  In that case, $|S(D, \delta)|=1$.

Suppose that the statement is true for all $|V|\leq 2n$.  We now consider a graph diagram $D$ with $2n+2$ vertices.  Note that $D$ must contain a local configuration of Fig. \ref{fig: SmoothOut} with $i=j=1$, so we resolve the vertices and obtain $D'$.  Since $D'$ (or $D_1'$ and $D_2'$ in the case when $D'$ is disconnected) is a connected trivalent plane graph diagram with a strictly fewer number of vertices, we have the inductive hypothesis that $S(D',\delta) \neq \emptyset$ (resp. $S(D_i',\delta) \neq \emptyset$ for $i=1,2$ when $D'$ is disconnected).  Hence, Lemma \ref{LemmaTwo} implies that $S(D, \delta) \neq \emptyset$.
\end{proof}

\begin{proof}[Proof of Theorem \ref{Nonvanishing}]
We prove by induction on the maximal color $m$ on the graph $G$, which is the same type of argument we used in proving Theorem \ref{theorem:axiom}. When $m=1$, the vertex set $V$ is empty and we have already discussed the case.  When $m=2$, the statement follows from Lemma \ref{LemmaOne} and Lemma \ref{LemmaThree}.  Thus it suffices to consider $m\geq 3$.

Suppose now that $\Delta_{(D,c)}(t)\neq 0$, or equivalently, $S(D, \delta) \neq \emptyset$ for all graph diagram $D$ with maximal color $m-1$.  We want to prove the same thing for a graph diagram $D$ with maximal color $m$.

First consider the case when there is a unique edge on $D$ with the maximal color $m$.
From Equation (\ref{RelationMaximalColor}), we see that
\begin{equation*}
\left( \begin{tikzpicture}[baseline=-0.65ex, thick]
\draw (0,-1) [->-] to [out=90,in=270] (0.5,-0.33);
\draw (0.5, -0.33) [->-] to [out=90,in=270] (0.5,0.33);
\draw (0.5, 0.33) [->] to [out=90,in=270] (0,1);
\draw (1,-1) [->-] to [out=90,in=270] (0.5,-0.33);
\draw (1,1) [<-] to [out=270,in=90] (0.5,0.33);
\draw (0, -1.25) node {$j$};
\draw (0.9,-1.25) node {$m-j$};
\draw (0,1.25) node {$l$};
\draw (1,1.25) node {$m-l$};
\draw (1,0) node {$m$};
\draw (0.5, -0.33) node[circle,fill,inner sep=1pt]{};
\draw (0.5, 0.33) node[circle,fill,inner sep=1pt]{};
\end{tikzpicture}\right)
= \frac{1}{[j]\,[l]\,[m-1]^2}\left( \begin{tikzpicture}[baseline=-0.65ex,thick, scale=0.85]
\draw (0,-3) [->-] to (0, -2.7);
\draw (0, -2.7) to [out=90,in=270] (0,-2);
\draw (0, -2) [->-] to [out=90,in=270] (1,-0.5);
\draw (1, -0.5) [->-] to [out=90,in=270] (1,0.5);
\draw (1, 0.5) [->-] to [out=90,in=270] (0,2) to [out=90,in=270] (0,2.7);
\draw (0, 2.7) [->] to (0, 3);
\draw (0, -2.7) node[circle,fill,inner sep=1pt]{};
\draw (1, -0.5) node[circle,fill,inner sep=1pt]{};
\draw (1, 0.5) node[circle,fill,inner sep=1pt]{};
\draw (0, 2.7) node[circle,fill,inner sep=1pt]{};
\draw (2,-3) [->-] to [out=90,in=270](2,-2);
\draw (2, -2) [->-] to [out=90,in=270]  (1,-0.5);
\draw (2,3) [<-] to [out=270,in=90] (2,2);
\draw (2,2) [-<-] to [out=270,in=90] (1,0.5);
\draw (0,2.7) [-<-] to [out=270, in=180] (0.5,2.35) to [out=0,in=180] (1.5,2.35) to [out=0,in=90] (2,2);
\draw (0,-2.7) [->-] to [out=90,in=180] (0.5,-2.35) to [out=0,in=180] (1.5,-2.35) to [out=0,in=270] (2,-2);
\draw (2, -2) node[circle,fill,inner sep=1pt]{};
\draw (2, 2) node[circle,fill,inner sep=1pt]{};
\draw (0, -3.25) node {$j$};
\draw (2,-3.25) node {$m-j$};
\draw (0,3.25) node {$l$};
\draw (2,3.25) node {$m-l$};
\draw (1.6,0) node {$m$};
\draw (1,2.6) node {$l-1$};
\draw (1,-2.6) node {$j-1$};
\draw (-0.25,2) node {$1$};
\draw (-0.25,-2) node {$1$};
\draw (2.2,1.1) node {$m-1$};
\draw (2.2,-1.1) node {$m-1$};
\end{tikzpicture}\right).
\end{equation*}
Since the Alexander polynomial of the two diagrams are either simultaneously zero or simultaneously non-zero, it suffices to prove the non-vanishing property of a diagram $D$ on the right-hand side of the equation that contains the local configuration:
\begin{equation*}
\begin{tikzpicture}[baseline=-0.65ex, thick]
\draw (0,-1) [->-] to [out=90,in=270] (0.5,-0.33);
\draw (0.5, -0.33) [->-] to [out=90,in=270] (0.5,0.33);
\draw (0.5, 0.33) [->] to [out=90,in=270] (0,1);
\draw (1,-1) [->-] to [out=90,in=270] (0.5,-0.33);
\draw (1,1) [<-] to [out=270,in=90] (0.5,0.33);
\draw (0, -1.25) node {$1$};
\draw (0.9,-1.25) node {$m-1$};
\draw (0,1.25) node {$1$};
\draw (1,1.25) node {$m-1$};
\draw (1,0) node {$m$};
\draw (0.5, -0.33) node[circle,fill,inner sep=1pt]{};
\draw (0.5, 0.33) node[circle,fill,inner sep=1pt]{};
\end{tikzpicture}
\end{equation*}
Now, we can smooth the crossing and obtain $D'$ as in Fig. \ref{fig: SmoothOut}.  Since the maximal color of $D'$ (or $D_1'$ and $D_2'$ in the case when $D'$ is disconnected) is at most $m-1$, we have $S(D',\delta) \neq \emptyset$ (resp. $S(D_i',\delta) \neq \emptyset$ for $i=1,2$ when $D'$ is disconnected) by induction.  Hence, Lemma \ref{LemmaTwo} implies that $S(D, \delta) \neq \emptyset$, and Lemma \ref{LemmaOne} implies that $\Delta_{(D,c)}(t)\neq 0$.

In general, suppose $D$ has $k$ edges of the maximal color $m$.  We can apply another induction on $k$ and repeat the above argument to reduce $k$. This concludes the proof.

\end{proof}

\subsection{An intrinsic invariant}

So far, we have implicitly assumed \emph{embedded graphs} in this paper, i.e., they are the graphs that exist in a specific position.  In contrast, an \emph{abstract graph} is a graph that is considered to be independent of any particular embedding.  Whereas an embedded graph $G$ has a unique underlying abstract graph $g$, a given abstract graph $g$ can be typically embedded in many different ways and gives rise to distinct embedded graphs.  A property or invariant of a graph is called \emph{intrinsic} if it only depends on the underlying abstract graph (and is independent of the embedding).

Suppose $G$ is an MOY graph.  Recall that (\ref{Def:Alex}) defines $\Delta_{(G,c)}(t)$ up to a power of $t$; hence, $\Delta_{(G,c)}(1)$ is a well-defined invariant that is independent of the graph diagrams.  Our next result shows that $\Delta_{(G,c)}(1)$ is furthermore an intrinsic invariant of graphs. This may be viewed as a generalization of the classical result that the Alexander polynomial of a knot satisfies $\Delta_K(1)=1$.

\begin{prop}\label{intrinsic}
The value of the Alexander polynomial evaluated at $t=1$, $\Delta_{(G,c)}(1)$, is an integer-valued invariant of the underlying abstract MOY graph $g$.

\end{prop}

\begin{proof}

Note that the local contributions $M_{C_p}^{\triangle}$ and $A_{C_p}^{\triangle}$ evaluated at $t=1$ is invariant under crossing changes (Compare the first and second column of Fig. ~\ref{fig:e1}).  Since any other embedded graphs $G'$ with the same underlying abstract graph $g$ can be obtained from $G$ via the Reidemeister moves in Fig. ~\ref{fig:e25} and crossing changes, it follows that $\Delta_{(G,c)}(1)=\Delta_{(G',c)}(1)$; in other words, it is indeed an intrinsic invariant of the abstract MOY graph $g$.

\end{proof}

In light of Proposition \ref{intrinsic}, we propose to define an \emph{Alexander invariant} for an abstract MOY graph $g$ by letting
\begin{equation}
\Delta_{(g,c)}:= \Delta_{(G,c)}(1),
\end{equation}
where $G$ is an arbitrary MOY embedding of $g$.

\newpage


\appendix
\section{Replacing $\mathrm{sign}(s)\cdot m(s)$ with $M(s)$  }

The goal of this appendix is to show Proposition \ref{equivalence}.  To this end, it suffices to prove the following Proposition. Let $(D, \delta)$ be a decorated MOY graph diagram. For any given total orders on $\mathrm{Cr(D)}$ and  $\mathrm{Re(D)}$, recall that $\mathrm{sign}(s)$ is defined to be the sign of the state $s\in S(D, \delta)$ as a permutation with respect to the given orders.

\begin{prop}
\label{sign}
Let $(D, \delta)$ be a decorated MOY graph diagram, we have
\begin{eqnarray}
\label{appen}
\frac{\mathrm{sign}(s_{1})}{\mathrm{sign}(s_{2})}=\frac{M(s_{1})m(s_{1})}{M(s_{2})m(s_{2})},
\end{eqnarray}
for any two states $s_{1}, s_{2} \in S(D, \delta)$. Namely we have $\mathrm{sign}(s)=\prod_{p=1}^{N}\mathrm{sign}_{C_p}^{s(C_{p})}$ up to an overall sign change, where $\mathrm{sign}_{C_p}^{s(C_{p})}$ is the local contribution defined in Fig. \ref{fig:e7}.
\end{prop}
\begin{proof}
It is easy to see that any diagram $D$ can be transformed into a trivalent graph diagram by surgery ($\Lambda I$) or ($I\Lambda$) and surgery (I), which are defined in Fig. \ref{fig:e9} and \ref{fig:e8}. Therefore the proof is a combination of Lemmas \ref{trivalent}, \ref{I} and \ref{lambdaI}.
\end{proof}

\begin{figure}[h!]
\begin{tikzpicture}[baseline=-0.65ex, thick, scale=0.9]
\draw (-1,-1) [->] to (1,1);
\draw (1,-1) -- (0.2,-0.2);
\draw (-0.2,0.2) [->] to (-1,1);
\draw (0, 0.5) node {$1$};
\draw (0, -0.5) node {$-1$};
\draw (0.5, 0) node {$1$};
\draw (-0.6, 0) node {$1$};
\end{tikzpicture}\hspace{1.5cm}
\begin{tikzpicture}[baseline=-0.65ex, thick, scale=0.9]
\draw (1,-1) [->] to (-1,1);
\draw (-1,-1) -- (-0.2,-0.2);
\draw (0.2,0.2) [->] to (1,1);
\draw (0, 0.5) node {$1$};
\draw (0, -0.5) node {$-1$};
\draw (0.5, 0) node {$1$};
\draw (-0.5, 0) node {$1$};
\end{tikzpicture}
\hspace{1.5cm}
\begin{tikzpicture}[baseline=-0.65ex, thick, scale=0.9]
\draw (0, 0.5) ellipse (1.5cm and 0.8cm);
\draw (0,-1) [->-] to (0,-0.3);
\draw (-0.5, -0.6) node {$-1$};
\draw (0.4, -0.6) node {$1$};
\draw (0, 0) node {$1$};
\end{tikzpicture}
	\caption{The local contribution of $\mathrm{sign}_{C_p}^{\Delta}$.}
	\label{fig:e7}
\end{figure}

\begin{lemma}
\label{singular}
Let $D$ be a diagram of a singular link. Namely at each vertex, there are two edges pointing inwards and two outwards. Then the statement of Theorem \ref{sign} holds.
\end{lemma}

The Heegaard Floer homology and the Alexander polynomial of a singular link are studied in \cite{MR2529302}. Using statements there we can get Lemma \ref{singular}. However to keep the combinatorial flavor of the paper, we give a proof by recalling some facts in \cite{MR712133}.

A {\it universe} $U$ is a connected diagram of a link in $\mathbb{R}^{2}$ where the data of which strand is over and which is under at each crossing is suppressed. Choose a base point $\delta$ on $U$ and consider the set of states $S(U, \delta)$. The Clock Theorem in \cite{MR712133} says that any two states in $S(U, \delta)$ are connected by a sequence of clockwise or counterclockwise transpositions, which are defined in Fig. \ref{clock}.

\begin{figure}
\begin{tikzpicture}[baseline=-0.65ex, thick, scale=0.9]
\draw (-2.5, 0) -- (-0.5,0);
\draw [dotted](-0.5, 0) -- (0.5,0);
\draw (0.5, 0) -- (2.5,0);
\draw (-1.5, -1) -- (-1.5, 1);
\draw (1.5, -1)--(1.5, 1);
\draw (-1.3, 0.2) node {$\bullet$};
\draw (1.3, -0.2) node {$\bullet$};
\draw (5, 0.5) node {$\xrightarrow{\text{clockwise transposition}}$};
\draw (5, -0.5) node {$\xleftarrow{\text{counterclockwise transposition}}$};
\end{tikzpicture} \quad
\begin{tikzpicture}[baseline=-0.65ex, thick, scale=0.9]
\draw (-2.5, 0) -- (-0.5,0);
\draw [dotted](-0.5, 0) -- (0.5,0);
\draw (0.5, 0) -- (2.5,0);
\draw (-1.5, -1) -- (-1.5, 1);
\draw (1.5, -1)--(1.5, 1);
\draw (-1.3, -0.2) node {$\bullet$};
\draw (1.3, 0.2) node {$\bullet$};
\end{tikzpicture}
	\caption{Transpositions in the Clock Theorem.}
	\label{clock}
\end{figure}

Two states that are connected by a transposition satisfy \eqref{appen}. By the Clock Theorem, any two states in $S(U, \delta)$ satisfy \eqref{appen}. Therefore the sign of a state for a universe is given by the local contribution in Fig. \ref{fig:e7}.

\begin{proof}[Proof of Lemma \ref{singular}]
Since $D$ is a connected diagram of a singular link, we call the underlying universe $U_D$. The difference of $\mathrm{Cr}(D)$ and $\mathrm{Cr}(U_D)$ occurs around the singular crossings, and the same is true for $\mathrm{Re}(D)$ and $\mathrm{Re}(U_D)$.

Suppose $C_i$ is a crossing of $U_D$ which corresponds to a singular crossing of $D$.
There are two crossings in $\mathrm{Cr}(D)$ and one circle region in $\mathrm{Re}(D)$ around $C_i$, which we call $C_i, C_{i.5}$ and $R_{i.5}$, as in Fig. \ref{fig:e12}. We choose orders on $\mathrm{Cr}(D)$ and $\mathrm{Re} (D)$ so that
$$C_{i}< C_{i.5}<C_{i+1} \text{ and } R_i<R_{i.5}<R_{i+1}.$$
Then by forgetting $C_{i.5}$ and $R_{i.5}$ we get orders on $\mathrm{Cr} (U_D)$ and $\mathrm{Re} (U_D)$.

We see that any state $s$ in $S(D, \delta)$ corresponds to a state $\tilde{s}$ in $S(U_D, \delta)$ by ignoring the circle regions:
\begin{equation*}
\tilde{s}(C_i)=\begin{cases}
\text{souther corner} ,&  \,\,s(C_i)=\text{east corner} \lor s(C_{i.5})=\text{west corner} \\
\text{west corner}, &  \,\, s(C_i)=\text{west corner} \\
\text{east corner}, &  \,\, s(C_i)=\text{east corner}
\end{cases}
\end{equation*}

Considering the orders above, we have
\begin{equation*}
\mathrm{sign}(s)=\begin{cases}
\mathrm{sign}(\tilde{s}), &\text{ if $s(C_{i.5})=R_{i.5}$} \\
-\mathrm{sign}(\tilde{s}), &\text{ if $s(C_{i})=R_{i.5}$.}
\end{cases}
\end{equation*}
Up to an overall sign change, it is easy to check that the signs defined from the local contributions in Fig. ~\ref{fig:e12} satisfies the relation above.
\end{proof}

\begin{figure}
\begin{tikzpicture}[baseline=-0.65ex, thick, scale=1.2]
\draw (-1,-1) [->] to (1,1);
\draw (1,-1) [->] to (-1,1);
\draw (0, 0.5) node {$1$};
\draw (0, -0.5) node {$-1$};
\draw (0.5, 0) node {$1$};
\draw (-0.6, 0) node {$1$};
\draw (0, -1.9) node {$C_i\in \mathrm{Cr}(U_D)$};
\end{tikzpicture}
\hspace{1.5cm}
\begin{tikzpicture}[baseline=-0.65ex, thick, scale=1.2]
\draw (0, 0) ellipse (0.5cm and 0.5cm);
\draw (-1,-1) [->] to (1,1);
\draw (1,-1) [->] to (-1,1);
\draw (-0.7, -0.35) node {$\tiny{-1}$};
\draw (-0.4, -0.6) node {$\tiny{1}$};
\draw (0.3, -0.6) node {$\tiny{-1}$};
\draw (0.65, -0.35) node {$\tiny{1}$};
\draw (-0.2, -0.2) node {$\tiny{1}$};
\draw (0.2, -0.2) node {$\tiny{1}$};
\draw (-0.6, -0.9) node {$C_i$};
\draw (0.6, -0.9) node {$C_{i.5}$};
\draw (1, 0) node {$R_{i.5}$};
\draw (0.7, 0)[->] to (0.3, 0);
\draw (0, -1.7) node {$C_i, C_{i.5}\in \mathrm{Cr}(D)$};
\draw (0, -2.2) node {$R_{i.5}\in \mathrm{Re}(D)$};
\end{tikzpicture}
	\caption{The local contribution of sign for a universe $U_D$ (left) and a graph $D$ (right).}
	\label{fig:e12}
\end{figure}

\begin{lemma}
\label{trivalent}
When $D$ is a diagram of an oriented trivalent graph without sinks or sources, the statement of Theorem \ref{sign} holds.
\end{lemma}

To prove Lemma \ref{trivalent}, we separate the vertices of a trivalent graph into two groups. We call a vertex with indegree two and outdegree one an {\it even vertex}, and a vertex with indegree one and outdegree two an {\it odd vertex}. It is easy to see that the number of even vertices equals that of odd vertices. We use an oriented simple arc to connect an even vertex to an odd vertex, and call it surgery (X), which transforms two trivalent vertices into singular crossings.

\begin{figure}
\begin{tikzpicture}[baseline=-0.65ex, thick, scale=1]
\draw (-2.5, 2.5) ellipse (0.5cm and 0.5cm);
\draw (-3.5,1.5) [->-] to (-2.5,2.5);
\draw (-2.5,2.5) [->] to (-1.5,3.5);
\draw (-2.5,2.5) [->] to (-3.5,3.5);
\draw (0.5, -0.5) ellipse (0.5cm and 0.5cm);
\draw (-0.5,-1.5) [->-] to (0.5,-0.5);
\draw (0.5,-0.5) [->] to (1.5,0.5);
\draw (1.5,-1.5) [->-] to (0.5, -0.5);
\draw (-2.7, 0.7) -- (-0.7, 2.7);
\draw (-1.3, -0.7) -- (0.7, 1.3);
\draw (-2.5, 1.5) node{$R_1$};
\draw (-1.7, 0.7) node{$R_2$};
\draw (-0.5, -0.5) node{$R_k$};
\draw [dotted](-1.3, 0.3)--(-1, 0);
\draw (2.7, 1.5) node {$\xleftrightarrow{\,\text{surgery (X)}\,}$};
\end{tikzpicture} \quad \quad
\begin{tikzpicture}[baseline=-0.65ex, thick, scale=1]
\draw (-2.5, 2.5) ellipse (0.5cm and 0.5cm);
\draw (-3.5,1.5) [->-] to (-2.5,2.5);
\draw (-2.5,2.5) [->] to (-1.5,3.5);
\draw (-2.5,2.5) [->] to (-3.5,3.5);
\draw (0.5, -0.5) ellipse (0.5cm and 0.5cm);
\draw (-0.5,-1.5) [->-] to (0.5,-0.5);
\draw (0.5,-0.5) [->] to (1.5,0.5);
\draw (1.5,-1.5) [->-] to (0.5, -0.5);
\draw (-2.7, 0.7) -- (-0.7, 2.7);
\draw (-1.3, -0.7) -- (0.7, 1.3);
\draw (-2.8, 1.2) node{$R_1$};
\draw (-2, 0.4) node{$R_2$};
\draw (-0.8, -0.8) node{$R_k$};
\draw [dotted](-1.6, 0)--(-1.3, -0.3);
\draw [dotted](0.5,-0.5) [->-] to (-2.5, 2.5);
\draw (-2.2, 1.8) node{$C_1$};
\draw (-1.65, 1.25) node{$C_2$};
\draw (-0.3, -0.1) node{$C_k$};
\draw (-1, 3) node{$R_{1.5}$};
\draw (-0.2, 2.2) node{$R_{2.5}$};
\draw (1, 1) node{$R_{k.5}$};
\end{tikzpicture}
	\caption{The surgery (X) transforms a trivalent graph into a singular link.}
	\label{fig:e6}
\end{figure}

\begin{proof}[Proof of Lemma \ref{trivalent}]
Let $D_{1}$ and $D_{2}$ be the graph diagrams before and after applying a surgery (X).
We show that if the statement of Theorem \ref{sign} holds for $D_{2}$, it also holds for $D_{1}$.

As shown in Fig. \ref{fig:e6}, the interior of the newly added arc intersects $D_{1}$ at several points. We assume that the regions that the arc goes across, which we call $R_1, R_2, \cdots, R_k$, are distinct regions, otherwise we can replace the arc by an arc with less intersection points with $D_{1}$.

The arc separates each region $R_q$ for $1\leq q \leq k$ into two regions since $D_{1}$ is a connected diagram. We label the regions in $D_{2}$ around the arc by $R_1, R_{1.5}, R_2, R_{2.5}, \cdots, R_k, R_{k.5}$ as shown in Fig. \ref{fig:e6}, and label the newly created crossings by $C_1, C_2, \cdots, C_k$.
Note that $\mathrm{Cr}(D_2)=\mathrm{Cr}(D_1)\cup \{C_1, C_2, \cdots, C_k\}$. Consider an order on $\mathrm{Cr}(D_1)$ and extend it to an order on $\mathrm{Cr}(D_2)$ by requiring that
$$C_1 < C_2 <\cdots <C_k < \text{ any other crossing in $\mathrm{Cr}(D_1)$}.$$
Consider an order on $\mathrm{Re}(D_2)$ so that
$$R_1<R_{1.5}<R_{2}<R_{2.5}<\cdots<R_k<R_{k.5}.$$
Then we get an order on $\mathrm{Re}(D_1)$ by forgetting $R_{i.5}$'s.

Surgery (X) naturally induces an injective map $\phi: S(D_{1}, \delta) \to S(D_{2}, \delta)$ described as below. Given $s\in S(D_{1}, \delta)$, $\phi (s)$ sends the crossing $C_p$ to the region $R_p$ (resp. $R_{p.5}$) if $s$ does not occupy any corner in $R_p$ (resp. $R_{p.5}$), for $1\leq p \leq k$.

By considering the orders above, we can check that
$$\displaystyle \frac{\mathrm{sign} (s_{1})}{\mathrm{sign} (s_{2})}=\frac{\mathrm{sign} (\phi (s_{1})) \prod_{p=1}^{k} \mathrm{sign}_{C_p}^{\phi (s_{1})(C_{p})}}{\mathrm {sign} (\phi (s_{2}))\prod_{p=1}^{k} \mathrm {sign}_{C_p}^{\phi (s_{2})(C_{p})}}=\frac{\prod_{C_p>C_k} \mathrm{sign}_{C_p}^{\phi (s_{1})(C_{p})}}{\prod_{C_p>C_k} \mathrm{sign}_{C_p}^{\phi (s_{1})(C_{p})}}$$ for any two states $s_{1}, s_{2} \in S(D_1, \delta)$. The first equality comes from a direct calculation of the sign using the orders defined above, and the second equality follows from the assumption that the statement of Theorem \ref{sign} holds for $D_{2}$. Note that the crossings of $D_2$ that are greater than $C_k$ are exactly the crossings of $D_1$. Therefore $$\displaystyle \frac{\mathrm{sign} (s_{1})}{\mathrm{sign} (s_{2})}=\frac{\prod_{C_{p}\in \mathrm{Cr}(D_{1})} \mathrm{sign}_{C_p}^{s_{1}(C_{p})}}{\prod_{C_{p}\in \mathrm{Cr}(D_{1})} \mathrm{sign}_{C_p}^{s_{1}(C_{p})}}.$$ This completes the proof of the lemma.

\end{proof}

\begin{lemma}
\label{I}
If the statement of Theorem \ref{sign} holds for the diagram after a surgery (I), then it also holds for the one before the surgery (I) (see Fig. \ref{fig:e8}).
\end{lemma}
\begin{proof}
Suppose the diagrams before and after a surgery (I) are $D_1$ and $D_2$. It is easy to see that the surgery (I) induces a one-one map $\phi : s(D_{1}, \delta)\to s(D_{2}, \delta)$ since $C_1$ must be mapped to $R_1$. Therefore $$\frac{\mathrm{sign} (s_{1})}{\mathrm{sign} (s_{2})}=\frac{\mathrm{sign} (\phi (s_{1}))}{\mathrm{sign} (\phi (s_{2}))},$$ the right-hand side of which, by assumption, is defined by the local contribution of sign. This completes the proof.

\end{proof}

\begin{figure}
\begin{tikzpicture}[baseline=-0.65ex, thick, scale=1.2]
\draw (0, 0) ellipse (0.5cm and 0.5cm);
\draw (-1,-1) [->] to (1,1);
\draw (1,-1) [->] to (-1,1);
\draw (0, 0.8) node{$....$};
\draw (0, -0.8) node{$....$};
\end{tikzpicture} $\xrightarrow{\text{surgery (I)}}$
\begin{tikzpicture}[baseline=-0.65ex, thick, scale=1.2]
\draw (0, -0.5) ellipse (0.5cm and 0.3cm);
\draw (-1,-1.5) [->-] to (0,-0.5);
\draw (1,-1.5) [->-] to (0,-0.5);
\draw (0,-0.5) [->-] to (0,0.5);
\draw (0, 0.5) ellipse (0.5cm and 0.3cm);
\draw (0,0.5) [->] to (1,1.5);
\draw (0,0.5) [->] to (-1,1.5);
\draw (0, 1.3) node{$....$};
\draw (0, -1.3) node{$....$};
\draw (0.45, 0.15) node{$\leftarrow C_1$};
\draw (0.7, 0.5) node{$\leftarrow R_1$};
\end{tikzpicture}
	\caption{Surgery (I).}
	\label{fig:e8}
\end{figure}

\begin{lemma}
\label{lambdaI}
If the statement of Theorem \ref{sign} holds for the diagram after a surgery ($\Lambda I$) or ($I\Lambda$), then it also holds for the one before the surgery ($\Lambda I$) or ($I\Lambda$) (see Fig. \ref{fig:e9}).
\end{lemma}
\begin{proof}
We prove the lemma for surgery ($\Lambda \mathrm{I}$), and the case of surgery ($\mathrm{I}\Lambda$) can be proved similarly. Suppose the diagrams before and after surgery ($\Lambda \mathrm{I}$)) are $D_1$ and $D_2$. Consider an order on $\mathrm{Cr}(D_2)$, which induces an ordering in $\mathrm{Cr}(D_1)$ by forgetting $C_1$. Since $\mathrm{Re}(D_2)=\mathrm{Re}(D_1)\cup \{R_1\}$, an order on $\mathrm{Re}(D_2)$ also induces an order on $\mathrm{Re}(D_1)$ by forgetting $R_1$.

The surgery ($\Lambda \mathrm{I}$)) induces an injective map $\phi : s(D_{1}, \delta)\to s(D_{2}, \delta)$ as below. Given $s\in s(D_1, \delta)$, if $s(C_{k})\neq R_2$, let $\phi (s)(C_{1})=R_1$ and $\phi (s)$ maps the other crossings the same way as $s$. If $s(C_{k})= R_2$ and $s(C_{k-1})$ is the east corner of $C_{k-1}$, let $\phi (s)(C_k)=R_1$, $\phi (s)(C_{k-1})=R_2$ and $\phi (s)(C_1)$ be the east corner. If $s(C_{k})= R_2$ and $s(C_{k-1})$ is the west corner, then $s(C_2)$ must be its west corner. In this case, let $\phi (s)(C_k)=R_1$, $\phi (s)(C_2)=R_2$ and $\phi (s)(C_1)$ the west corner.

In each case, we can check that
$\mathrm{sign}(s)=\mathrm{sign}(\phi (s))$ for any $s\in s(D_1, \delta)$. The right-hand side of the equality, by assumption, is defined by the local contribution of sign. Since $\phi$ is an injective map, the local definition of sign works as well for $D_1$.

\end{proof}

\begin{figure}
\begin{tikzpicture}[baseline=-0.65ex, thick, scale=1.2]
\draw (0.5, -0.5) ellipse (0.5cm and 0.3cm);
\draw (1.5,-1.5) [->-] to (0.5,-0.5);
\draw (1,-1.5) [->-] to (0.5,-0.5);
\draw (-0.5,-1.5) [->-] to (0.5,-0.5);
\draw (0.5,-0.5) [->-] to (0,0.5);
\draw (-1.5,-1.5) [->-] to (0,0.5);
\draw (0, 0.5) ellipse (0.5cm and 0.3cm);
\draw (0,0.5) [->] to (1,1.5);
\draw (0,0.5) [->] to (-1,1.5);
\draw (0, 1.3) node{$....$};
\draw (0.5, -1.3) node{$....$};
\end{tikzpicture} $\xleftarrow{\text{surgery (I$\Lambda$)}}$
\begin{tikzpicture}[baseline=-0.65ex, thick, scale=1.2]
\draw (0, 0) ellipse (0.5cm and 0.5cm);
\draw (-1,-1) [->-] to (0,0);
\draw (0,0) [->] to (1,1);
\draw (1,-1) [->-] to (0,0);
\draw (0,0) [->] to (-1,1);
\draw (0.5,-1) [->-] to (0,0);
\draw (-0.5,-1) [->-] to (0,0);
\draw (0, 0.8) node{$....$};
\draw (0, -0.8) node{$....$};
\draw (-0.9, -0.35) node{$C_2\rightarrow$};
\draw (0.9, -0.35) node{$\leftarrow C_k$};
\end{tikzpicture} $\xrightarrow{\text{surgery ($\Lambda$I)}}$
\begin{tikzpicture}[baseline=-0.65ex, thick, scale=1.2]
\draw (-0.5, -0.5) ellipse (0.5cm and 0.3cm);
\draw (-1.5,-1.5) [->-] to (-0.5,-0.5);
\draw (-1,-1.5) [->-] to (-0.5,-0.5);
\draw (0.5,-1.5) [->-] to (-0.5,-0.5);
\draw (-0.5,-0.5) [->-] to (0,0.5);
\draw (1.5,-1.5) [->-] to (0,0.5);
\draw (0, 0.5) ellipse (0.5cm and 0.3cm);
\draw (0,0.5) [->] to (1,1.5);
\draw (0,0.5) [->] to (-1,1.5);
\draw (0, 1.3) node{$....$};
\draw (-0.5, -1.3) node{$....$};
\draw (-0.6, 0.15) node{$C_1\rightarrow$};
\draw (0.75, 0.2) node{$\leftarrow C_k$};
\draw (0.7, 0.6) node{$\leftarrow R_1$};
\draw (0.4, -0.8) node{$\leftarrow C_{k-1}$};
\draw (0.2, -0.4) node{$\leftarrow R_2$};
\draw (-1.3, -0.8) node{$C_2\rightarrow$};
\end{tikzpicture}
	\caption{Surgery ($\mathrm{I}\Lambda $) and ($\Lambda \mathrm{I}$).}
	\label{fig:e9}
\end{figure}

\newpage

\bibliographystyle{siam}
\bibliography{Alexgraph}

\end{document}